\documentclass[12pt]{amsart}
\usepackage{amsthm,amssymb,amscd,amsmath,amsfonts,amssymb,amscd,stmaryrd,mathrsfs} 
\usepackage[shortlabels]{enumitem}
\usepackage[all]{xy}
\usepackage{hyperref} 					
\usepackage{texmac} 	
\usepackage{color} 				

\newtheorem{MainThm}{Theorem}

\theoremstyle{definition}
\newtheorem{defn}{Definition}[subsection]
\newtheorem{thm}[defn]{Theorem}
\newtheorem{cor}[defn]{Corollary}
\newtheorem{prop}[defn]{Proposition}
\newtheorem{lem}[defn]{Lemma}
\newtheorem{ex}[defn]{Example}

\newtheorem{rmk}[defn]{Remark}

\newtheorem{setup}[defn]{Hypothesis}

\calsymbols{c}{A,B,C,D,E,F,G,H,I,J,K,L,M,N,O,P,Q,R,S,T,U,V,W,X,Y,Z}
\bbsymbols{b}{A,B,C,D,E,F,G,H,I,J,K,L,M,N,O,P,Q,R,S,T,U,V,W,X,Y,Z}
\scrsymbols{s}{A,B,C,L,M,P,Q,R,S,T,U,}
\fraksymbols{f}{g,m,h,t,n,b,p}
\bsymbols{b}{g,h,x,y,A,B,C,D,E,F,G,H,I,J,K,L,M,N,O,P,Q,R,S,T,U,V,W,X,Y,Z}
\bbsymbols{}{B,R,C,Z,N,P,Q,F,G,H,U,W,X,Y,T}

\operators{Spec,Loc,Stab,Hom,Lie,rig,Sp,Ad,mod,op,Res,GL,Supp,ind,an,val,wh,Ext,Hol}

\DeclareMathOperator{\Frech}{Frech}
\DeclareMathAlphabet{\mathpzc}{OT1}{pzc}{m}{it}
\DeclareSymbolFont{largesymbols}{OMX}{yhex}{m}{n}
\DeclareMathAccent{\wideparen}{\mathord}{largesymbols}{"F3}

\newcommand{\h}[1]{\widehat{#1}}

\newcommand{\hK}[1]{\h{#1_K}}

\newcommand{\w}[1]{\wideparen{#1}}

\newcommand{\fr}[1]{\mathfrak{{#1}}}

\newcommand{\be}{\begin{enumerate}[{(}a{)}]}
\newcommand{\ee}{\end{enumerate}}
\newcommand{\qmb}[1]{\quad\mbox{#1}\quad}

\newcommand{\hsp}{\hspace{0.1cm}}

\newcommand{\congs}{\stackrel{\cong}{\longrightarrow}}

\newcommand{\wUg}[1]{\w{U}(\fr{g},{#1})}
\newcommand{\tocong}{\stackrel{\cong}{\longrightarrow}}

\newcommand{\utimes}[1]{\underset{#1}{\otimes}{}}
\newcommand{\wotimes}[1]{\underset{#1}{\w\otimes}{}}
  \let\leq=\leqslant
  \let\geq=\geqslant

\newcommand{\uM}{\underline{M}}

\begin{document}

\title{Irreducibility results for equivariant $\cD$-modules on rigid analytic spaces}
\author{Konstantin Ardakov}
\address{K. Ardakov\\Mathematical Institute\\University of Oxford\\Oxford OX2 6GG\\UK}

\author{Tobias Schmidt}
\address{T. Schmidt\\Bergische Universität Wuppertal\\42119 Wuppertal\\ GERMANY}

\begin{abstract} 
We prove a general irreducibility result for geometrically induced 
coadmissible equivariant $\cD$-modules on rigid analytic spaces.
As an application, we geometrically reprove the irreducibility of certain locally analytic representations 
previously constructed by Orlik-Strauch. 
\end{abstract}
\maketitle
\tableofcontents
\section{Introduction}

Let $G$ be a $p$-adic Lie group and let $K$ be a non-archimedean field of mixed characteristic $(0,p).$ Let $\bX$ be a smooth rigid $K$-analytic space endowed with an action of $G$.
The category $\cC_{\bX/G}$ of coadmissible $G$-equivariant $\cD$-modules
is the $p$-adic analogue of the classical category of equivariant coherent $D$-modules on a smooth complex variety endowed with an action of a real or complex Lie group. Similar to the classical case, one of the main motivations for the construction of $\cC_{\bX/G}$ comes from representation theory: if $\bX$ is the analytic flag variety of a $p$-adic reductive group $G$, then there is a Beilinson-Bernstein style localisation theorem which provides an equivalence of categories between $\cC_{\bX/G}$ and the category of admissible locally analytic $G$-representations with trivial infinitesimal character \cite{EqDCapTwo,HPSS}. This opens up the way to study locally analytic $G$-representations geometrically through techniques from $\cD$-module theory. In particular, one may try to construct irreducible representations geometrically on the flag variety $\bX$. 

\vskip5pt 

In this light, it is natural to study the preservation of irreducibility under various operations on equivariant $\cD$-modules. Such operations may come in two flavours, which correspond, vaguely speaking, to change of space or change of group. An example for the first case is the equivariant Kashiwara theorem 
\cite{EqDCapTwo}: given a smooth rigid analytic space $\bX$ together with an embedding $i: \bY\rightarrow\bX$ of a Zariski closed $G$-stable subspace, the equivariant direct image 
$i_+^G$ induces an equivalence between $\cC_{\bY/G}$ and the full subcategory $\cC^\bY_{\bX/G}$ of $\cC_{\bX/G}$ consisting of modules with support in $\bY$. An example for the second case is the induction equivalence \cite{EqDCapTwo}: if $P\subseteq G$ is a closed cocompact subgroup, then there is a geometric induction functor 
$$\ind_P^G: \cC_{\bX /P} \rightarrow  \cC_{\bX /G}.$$ Suppose that $P$ equals the stabilizer of a Zariski-closed subspace $\bY\subseteq\bX$, which is irreducible and quasi-compact and suppose additionally that $\bX$ is separated. If the $G$-orbit of $\bY$ is regular in $\bX$, i.e. distinct $G$-translates of $\bY$ are disjoint, then $\ind_P^G$ induces an equivalence of categories between 
$ \cC^\bY_{\bX /G}$ and $\cC^{G\bY}_{\bX /G}$. 
\vskip5pt 
Requiring the $G$-orbit of a given Zariski-closed subspace $\bY\subseteq\bX$ to be regular is a strong condition on $\bY$ and often not satisfied in practice. The present paper is motivated by the following question: under which weaker conditions does the functor $\ind_P^G$ still preserve irreducibility? In order to address this question properly, we establish first some useful foundational results.
\subsection{Induction and side-changing} According to \cite{EqDCapTwo}, there are side-changing functors $\Omega_\bX\otimes (-)$ and 
 ${\mathcal Hom} (\Omega_\bX,-)$ yielding mutually inverse equivalences of categories 
 between $\cC_{\bX /G}$ and and its right module version ${^r\cC}_{\bX /G}$. Similarly, there is a right module version ${^r\ind}_P^G$ of the induction functor, going from right $P$-equivariant $\cD_\bX$-modules to right $G$-equivariant $\cD_\bX$-modules.
\begin{MainThm} Let $P\subseteq G$ be a closed cocompact subgroup. Let $\cN\in \cC_{\bX /P}$. There is a natural isomorphism in 
${^r\cC}_{\bX /G}$
\[ {^r\ind}_P^G(\Omega_\bX\otimes \cN)\congs \Omega_\bX\otimes \ind_P^G\cN.\]
\end{MainThm}
See Theorem \ref{thm_IndSide} for the proof.
\subsection{Induction and duality} The category $\cC_{\bX /G}$ contains the full subcategory $\cC^{\rm wh}_{\bX /G}$ consisting of weakly holonomic equivariant modules \cite{VuThesis}. This is a $G$-equivariant version of the category of weakly holonomic $\cD$-modules appearing in \cite{DCapThree}. Whenever Bernstein's inequality holds in $\cC_{\bX /G}$, then there is an involutive duality functor $\mathbb{D}_G$ on $\cC^{\rm wh}_{\bX /G}$. Note that Bernstein's inequality holds, for example, whenever $\bX$ has a smooth formal model.
\begin{MainThm} \label{thm_IndDual} Assume that Bernstein's inequality holds
in $\cC_{\bX /P}$ and $\cC_{\bX /G}$. Let $\cN\in  \cC^{\rm wh}_{\bX /P}$. There is a natural isomorphism in $\cC^{\rm wh}_{\bX /G}$
\[\mathbb{D}_G(\ind_P^G\cN)\congs \ind_P^G(\mathbb{D}_P\cN).\]
\end{MainThm}
See Theorem \ref{thm_IndDual} for the proof.
\subsection{The main irreducibility result.} The set up is as follows. Suppose that $\bX$ is a connected, smooth, rigid $K$-analytic variety and $G$ is a {\it compact} $p$-adic Lie group acting continuously on $\bX$. Let 
$\bY$ be a connected Zariski closed subset of $\bX$, with stabilizer $P := G_{\bY}$.
We suppose that the triple $(\bX, \bY, G)$ satisfies condition (LSC) from \cite[Definition 2.5.6]{EqDCapTwo}. The condition is a little too technical for the purpose of this introduction, but it is always satisfied, for example, if $\bX$ is separated and $\bY$ is irreducible and quasi-compact.
We call a module $\cN \in \cC^{\bY}_{\bX/P}$ {\it locally simple}, if 
$\cN_{|\bU}$ is a simple object in $\cC_{\bU/P_{\bU}}$ whenever $\bU \in \bX_w(\cT)$ is connected and $\bU \cap \bY$ is connected and non-empty. Here, $P_\bU$ denotes the stabilizer of $\bU$ in $P$. 
 
\begin{MainThm} Suppose there is a Zariski closed subset $\bZ$ of $\bY$ with $\dim \bZ < \dim \bY$ which has the following properties: 
\begin{enumerate}
\item $\bigcup\limits_{\stackrel{g,h \in G}{g\bY \neq h\bY}} g\bY \cap h\bY \subseteq G \bZ$,
\item $\bY \cap \Sigma$ is connected, where $\Sigma=\bX\setminus G.\bZ$.
\end{enumerate}
Let $\cN \in \cC^{\bY}_{\bX/P}$ be weakly holonomic. If $\cN$ and $\mathbb{D}_P(\cN)$ are locally simple, then the induced module $\ind_P^G(\cN)$ is a simple object in $\cC_{\bX/G}$.
\end{MainThm}
See Theorem \ref{MainResult} for the proof. It is the first point (1) involving self-intersections, which is the crucial condition. For example, if $\bY$ has a regular $G$-orbit in $\bX$, then $\bZ=\emptyset$ trivially satisfies (1) and (2). Below we explain how to verify condition (1) in practice. In many situations, the locally simple module $\cN$ is in fact self-dual, so that the condition on $\mathbb{D}(\cN)$ is redundant. For example, assume that Bernstein's inequality holds
in $\cC_{\bX /G}$ and that $\bY\subseteq\bX$ is a $G$-stable Zariski closed subvariety. Denote by $i_+^G: \cC_{\bY/P}\rightarrow \cC_{\bX/G}$ the equivariant direct image functor 
\cite{EqDCapTwo}. We show in Thm. \ref{prop_SelfGDual} that $i_+^G\cO_\bY$ is self-dual. 

\subsection{The set of self-intersections}
 In this subsection, we give a criterion to verify condition (1) in the preceding theorem in practice. 
 Let $\bX$ be a rigid analytic variety and $G$ a $p$-adic Lie group (possibly non-compact) acting continuously on $\bX$. Let $\bY$ a Zariski closed subset of $\bX$ with the stabilizer $G_\bY$.
Let $S$ be a set of representatives for the double cosets $G_\bY \setminus G \;/ G_\bY$ containing $1\in G$ and define $S^\ast := S \backslash \{1\}$. We write
\[ \bR_v := \bY \cap v \bY \qmb{for every} v \in S, \qmb{and} \bZ :=  \bigcup\limits_{v\in S^\ast} \bR_v .\] The set $\bZ$ is Zariski closed in $\bX$, whenever $G_\bY \setminus G \;/ \;G_\bY$ is finite, and one has
\[\bigcup\limits_{\stackrel{g,h \in G}{g\bY \neq h\bY}} g\bY \cap h\bY = G \bZ,\]
whenever $\bX$ is quasi-compact. In this situation, the set $\bY$ has a regular $G$-orbit in $\bX$ if and only if $\bZ$ is empty. Hence $\bZ$ is the obstruction to $\bY$ having a regular $G$-orbit in $\bX$, at least if $\bX$ is quasi-compact. A major open question which we cannot answer completely at the moment is: under what general conditions is the complement $\Sigma=\bX \setminus G\bZ$ an admissible open subset of $\bX$? The main problem is that the stabilizer $G_\bZ$ in general is not cocompact in $G$, even when $G_\bY$ is cocompact in $G$. It is this open problem, which limits our current range of applications.

\subsection{Schubert varieties.} A first class of examples to which the main result applies are Schubert varieties in projective space. Consider $G={\rm GL}_n(L)$, where $L$ is a finite extension of $\Qp$ contained in $K$, acting on rigid analytic projective space $\P_K^{n-1,\rm an}$. Let 
\[\bX_{1} \subset \bX_{2} \subset \cdot\cdot\cdot \subset \bX_{n}\]
be the chain of Schubert varieties in $\P_K^{n-1,\rm an}$, i.e. 
$\bX_j$ is the Zariski closed subvariety of $\bX$ where the last $n-j$ homogeneous coordinates vanish. 
 \begin{MainThm}\label{thm-SchubProjIrred} Fix $j$ and let $P=G_{\bX_j}$.
Let $i: \bX_j\hookrightarrow \P_K^{n-1,{\rm an}}$ denote the closed embedding. 
 Let $\cN:=i^{P}_{+}\cO_{\bX_j}\in \cC_{\P_K^{n-1,{\rm an}}/P}^{\bX_j}$ be the $P$-equivariant pushforward of the structure sheaf $\cO_{\bX_j}$. Then the induced module
 $\ind_P^G \cN$ is a simple object in $\cC_{\P_K^{n-1,{\rm an}}/G}$.
  \end{MainThm}
See Theorem \ref{thm-SchubProjIrred} for the proof. We may also consider Schubert varieties in full flag varieties. So let $\G$ be a split connected reductive $K$-group $\G$, with its natural $\G$-action given by conjugating the Borel subgroups of $\G$. Let $G$ be a $p$-adic Lie group with a continuous homomorphism $G\rightarrow\G(K)$. Let $\T\subseteq\B$ be a Borel subgroup in $\G$ containing a split maximal torus $\T$. Let $W$ be the Weyl group of the pair $(\G,\T)$. The $\B$-orbits $C_w$ in the full flag variety $\G/\B$ can be indexed by the Weyl elements $w\in W$ and their Zariski closures $\X_w$ are the classical Schubert varieties. Let $\bX=(\G/\B)^{\rm an}$ and $\bX_w=(\X_w)^{\rm an}$ the corresponding analytic spaces. For a Schubert variety $i: \bX_w\subseteq \bX$ denote by $\bZ_w$ its set of self-intersections corresponding to a (finite) set of representatives 
 for $G_{\bX_w}\setminus G / \;G_{\bX_w}$.  
\begin{MainThm} Let $w\in W$ and $P:=G_{\bX_w}$. Suppose the following three conditions. 
\be 

\item $G\bZ_{w}=G_0\bZ_{w}$ with $G_0\subset G$ some compact open subgroup such that $G=G_0P$.
\item $\bX_w\setminus G\bZ_w$ is connected.
\item $\bX_{w}$ is smooth.
\ee
 Let $\cN:=i^{P}_{+}\cO_{\bX_w}\in \cC_{\bX/P}^{\bX_w}$ be the $P$-equivariant pushforward of the structure sheaf $\cO_{\bX_w}$. Then the induced module
 $\cM:=\ind_{P}^G \cN$ is a simple object in $\cC_{\bX/G}$.
  \end{MainThm}
  
See Theorem \ref{MainResult2} for the proof.  We briefly comment on the two conditions (a), (b) and (c) of the Theorem. 
Condition (a) does not hold for all Schubert varieties $\bX_w$ in $\bX=(\G/\B)^{\rm an}$, and is directly related to the open problem alluded above. A first case in which it fails, appears in the case $\G={\rm GL}_4$ and $\bX_w$ equal to the inverse image of the unique Schubert divisor in the analytic Grassmannian ${\rm Gr}(2,4)^{\rm an}$. 
This $\bX_w$ is non-smooth, so condition (c) also fails in this case.  One may imagine to eventually remove condition (c) by replacing the push-forward of $\cO_{\bX_w}$ by some
intermediate extension of $\cO_{\bC_w}$ where $\bC_w$ equals the corresponding Bruhat cell. However, a rigid analytic theory of intermediate extensions is currently not available. 
As for condition (b), we are not aware of any counterexamples where this conditions fails.

\vskip5pt 
Assume that $\B\subseteq\P$ is a parabolic subgroup and consider the projection
 $f^{\rm an}: (\G/\B)^{\rm an}\rightarrow (\G/\P)^{\rm an}$. Suppose that $\bX_w$ is the inverse image of a Schubert variety $\bX_{vP}$ in $(\G/\P)^{\rm an}$. We show that the conditions of the preceding theorem are satisfied as soon as the (analogous) conditions are satisfied for 
 $\X_{vP}$. This produces many examples. For example, the theorem covers open and closed Schubert varieties, the Schubert curves $\bX_{s}$ (for simple reflections $s\in W$) or Schubert varieties of the form $\bX_{w_{o,\P}}$ where $w_{o,\P}$ is the longest element in a parabolic subgroup $W_{\P}$ of $W$. In these cases, the $G$-orbit of $\bX_w$ is in fact regular (so that $\bZ_w=\emptyset$). In the case $\G={\rm GL}_n$, all Schubert varieties arising as inverse images from Schubert varieties in projective space are covered. In particular, all Schubert varieties for the groups $\G={\rm GL}_2$ or ${\rm GL}_3$ are covered.

\subsection{Application to locally analytic representations.} 
We give some first applications to the locally analytic representation theory of $p$-adic groups.
To this end, we fix a finite extension $L/\Qp$ and let $\G_L$ be a connected semisimple algebraic group over $L$. Let $L\subseteq K$ be a complete non-archimedean splitting field for $\G_L$. Set $\G:=\G_L \times_L K$ and let $\fr{g}$ be the Lie algebra of $\G$.
Let $\P_L\subseteq\G_L$ be a parabolic subgroup. 
Let $\T_L\subset \mathbb{L}_L\subset \P_L$ be a maximal split torus and a Levi subgroup respectively. Let $T, P, G$ be the groups of $L$-rational points
of $\T_L, \P_L, \G_L$ respectively. Let  $\T,\mathbb{L},\P$ be the base change from $L$ to $K$ of the groups  $\T_L,\mathbb{L}_L,\P_L$ respectively. 
Let $\fr{t}, \fr{l}, \fr{p}$ be the $K$-Lie algebras of $\T,\mathbb{L},\P$ respectively. 
Let $D(G,K)$ be the algebra of $K$-valued locally analytic distributions. Denote by $\X$ the algebraic flag variety of the split $K$-group $\G=\G_L \times_L K$, with its natural $\G$-action given by conjugating the Borel subgroups of $\G$. Let $\bX=\X^{\an}$ be the rigid analytification of $\X$, with its induced $G$-action. In the case where $\G_L$ is $L$-split, Orlik-Strauch introduce in \cite{OrlikStrauchCatO} a certain locally analytic lift $\cO^P$ of the parabolic BGG category $\cO^{\fr{p}}$. The definition extends without difficulty to our case of a $K$-split group $\G_L$. 
The category $\cO^P$ is abelian, of finite length and comes with an exact functor 
$\cF_P^G(-)': \cO^P\rightarrow \cC_{D(G,K)}$ into the category of coadmissible $D(G,K)$-modules, which preserves irreducibility under certain conditions. Our last main result proves a general compatibility of this functor with geometric induction in the following sense.
Classical Beilinson-Bernstein localization composed with rigid analytification gives a functor 
$\Loc^{U(\fr{g})}_\bX$ from $\cO_0^P$ into coherent $\cD_{\bX}$-modules. 

Let $\cD_{\bX}\subset\w\cD_{\bX}$ be the sheaf of analytic infinite order differential operators as constructed and studied in \cite{DCapOne}, together with the extension functor $E_{\bX}$ from coherent $\cD_{\bX}$-modules into the category $\cC_{\bX}$ of coadmissible $\w\cD_{\bX}$-modules \cite[Lemma 4.14]{BodeCTP}, \cite[\S 7.2]{DCapThree}. 
\begin{MainThm} The functor $E_{\bX}\circ \Loc^{U(\fr{g})}_\bX$, restricted to the category 
$\cO_0^P$, takes values in $\cC_{\bX/P}$. The resulting
diagram of functors 
\[ \xymatrix{\cO_0^{P}  \ar[rr]^{\cF_P^G(-)' } \ar[d] _{E_{\bX}\circ  \Loc^{U(\fr{g})}_\bX}&& \cC_{D(G,K),0} \ar[d]^{ \Loc^{D(G,K)}_\bX} \\ \cC_{\bX/P}  \ar[rr]_{\ind_P^G} &&  \cC_{\bX/G}.}\]
is commutative up to natural isomorphism.
\end{MainThm}

See Theorem \ref{thm-compatible} for the proof. The irreducible $U(\fr{g})_0$-modules $L_w:=L(-w(\rho)-\rho)$ for $w\in W$ exhaust the irreducible objects in 
$\cO_0$. If $\P$ denotes the stabilizer of $\X_w$, then $L_w\in\cO^{\mathfrak{p}}$ for $\fr{p}=\Lie(\P)$ and the main theorem of Orlik-Strauch in \cite{OrlikStrauchJH} proves that
$\cF_{P_w}^G(L_w)'$ is an irreducible $D(G,K)$-module provided that 
(H1)  $\G_L$ is $L$-split and (H2) that $p>2$ if the root system of $\G$ has irreducible components of type $B$, $C$ or $F_4$, and $p>3$ if the root system has
irreducible components of type $G_2$. Their argument relies on
the delicate calculation of explicit formulae for the action of certain nilpotent generators on highest weight modules of the BGG category $\cO$. Theorem F. allows us to deduce the irreducibility of $\cF_{P_w}^G(L_w)'$ for a non-split semisimple group $\G_L$ and for any $p$, whenever the geometric conditions (a), (b) and (c) from Theorem E. are satisfied for the analytic Schubert variety $\bX_w$.

\vskip5pt

For more details we refer to the main body of the text.

\section{Some complements on $D$-modules}
\subsection{Sheaves and supports}\label{Support}

Given an abelian sheaf $\cF$ on a topological space $X$, its support is defined as $\Supp \cF=\{x\in X: \cF_x\neq 0  \}$.

\vskip5pt



Let $\bX$ be a rigid $K$-analytic space. We denote by $\mathscr{P}(\bX)$ its associated Huber space. We have an inclusion $\bX\rightarrow 
\mathscr{P}(\bX)$ which sends a point $x\in \bX$ to the principal maximal filter $\mathfrak{m}_x:=\{ \text{admissible open }\bU\subseteq \bX: x\in \bU  \} $.
The sets of the form 
$\tilde{\bU}=\{p\in  \mathscr{P}(\bX): \bU \in p\}$ as $\bU$ runs over the admissible open subsets of $\bX$ form a basis of the topology for $\mathscr{P}(\bX)$.
There is an equivalence of categories $\cM\mapsto\tilde{\cM}$ between the abelian sheaves on $\bX$ and on $\mathscr{P}(\bX)$ \cite[\S 5]{SchVdPut}.
One has $\cM(\bU)=\tilde{\cM}(\tilde{\bU})$ for any admissible open subset $\bU$ of $\bX$, as follows from the proof of \cite[Theorem 1]{SchVdPut}.

\vskip5pt

Let $\cM$ be an abelian sheaf on $\bX$. Its {\it support} $\Supp \cM$ is defined to be the support of the associated sheaf $\tilde{\cM}$. Given a subset $S\subseteq\bX$, the
sheaf $\cM$ is said to be {\it supported on} $S$, in the sense of \cite[2.1.1]{EqDCapTwo}, if $\cM|_{\bV}=0$ for any admissible open subset $\bV$ of $\bX\setminus S$.

\begin{prop}\label{prop-zariskiclosedsupport} Let  $\bY \subseteq\bX$ be a Zariski-closed subset. Let $\cM$ be a sheaf on $\bX$ such that $\Supp \cM=\overline{\bY}$. Then $\cM$ is supported on $\bY$ and \[\bY =\bX - \bigcup \big\{\bV \text{ admissible open in } \bX \text{ with } \cM|_{\bV}=0\big\}.\]
\end{prop}
\begin{proof}
Let $\bU=\bX\setminus \bY$ so that $\overline{\bY} = \mathscr{P}(\bX) \backslash \tilde{\bU}$ by \cite[2.1.4]{EqDCapTwo}. Hence $\tilde{\cM}|_{\tilde{\bU}}=0$ and hence $\cM |_{\bU}=0$. This shows that $\cM$ is supported on $\bY$. We now have to show that \[\bU =\bigcup \big\{\bV \text{ admissible open in } \bX \text{ with } \cM|_{\bV}=0\big\}.\]
Let $\bV$  be an admissible open of $\bX$. If $\bV\cap \bY \neq \emptyset$, take $x\in \bV\cap \bY$. Then 
$\mathfrak{m}_x\in  \tilde{\bV}\cap \Supp \cM $ and hence $\tilde{\cM}_{\mathfrak{m}_x}\neq 0$. So  $\tilde{\cM}|_{\tilde{\bV}}\neq 0$ and therefore also 
$\cM|_{\bV}\neq 0$. So any $\bV$ with $\cM|_{\bV}= 0$ is contained in $\bU$.
\end{proof}

\subsection{Analytification} \label{Analytification}

Let $\X$ be a smooth $K$-scheme which is locally of finite type. We assume that 
$\X$ is quasi-compact and quasi-separated. 
Let $D_{\X}$ be the sheaf of algebraic finite order differential operators on $\X$. Denote by $\bX=\X^{\an}$ the rigid analytification of $\X$, a quasi-compact and quasi-separated rigid $K$-analytic space. \vskip5pt

Denote by $\rho:\bX\rightarrow\X$ the canonical morphism of locally ringed spaces and consider the functor 
$$\cM\rightsquigarrow \rho^*\cM:=\cO_{\bX}\otimes_{\rho^{-1}\cO_\X}\rho^{-1}\cM$$ 
from $\cO_{\X}$-modules to $\cO_{\bX}$-modules. We recall some basic properties. 

\begin{prop}\label{prop-analytification0}  

\be 

\item The functor $\rho^*$ is exact and faithful.

\item  If $\X$ is proper, then $\rho^*$ induces an equivalence between coherent modules.

\item If $\X$ is proper, then 
 \[H^{i}(\bX,\rho^*\cM)=H^{i}(\X,\cM)\] for any $i\geq 0$ and any
quasi-coherent $\cO_{\X}$-module $\cM$.

\ee
\end{prop}
\begin{proof} 
 According to \cite[5.1.2]{ConradIrreducible}, the functor $\rho^{*}$ is exact and faithful and preserves coherence. Suppose now that $\X$ is a proper $K$-scheme. The statement of (b) and the fact that
 $H^{i}(\bX,\rho^*\cM)=H^{i}(\X,\cM)$ for any $i\geq 0$ for coherent $\cM$ follow from \cite[3.3.3/4 and 3.4.10/11]{BerkovichBook}. So (c) holds in the coherent case. Since $\X$ is noetherian, any quasi-coherent module is the union of its coherent subsheaves and cohomology on $\X$ commutes with the formation of direct limits of abelian sheaves \cite[Ex. II.5.15 and III.2.9]{Hart}. On the other hand, since $\bX$ is a quasi-compact and quasi-separated rigid space, cohomology on $\bX$ also commutes with the formation of direct limits \cite[2.1.7 and its proof]{ConradAmple}. This completes the proof of (c).
\end{proof}

\begin{prop}\label{prop-analytification1}  

\be 

\item  $D_{\bX}$ and $D_{\X}$ are coherent sheaves of rings on $\bX$ and $\X$ respectively and one has $\rho^*D_{\X}=D_{\bX}$.

\item $\rho^*$ induces a functor from $D_{\X}$-modules to $D_{\bX}$-modules. 

\item If $\cM$ is a coherent $D_{\X}$-module, then $\rho^*\cM$ is a 
coherent $D_{\bX}$-module.
\ee
\end{prop}
\begin{proof}
$D_{\bX}$ and $D_{\X}$ are coherent sheaves of rings by 
\cite[1.1.1]{ChiaLeStum} and \cite[1.4.9]{HTT} respectively. 
The canonical morphism $\rho^{*}\Omega_{\X}\congs\Omega_{\bX}$
is an isomorphism and so is its $\cO_{\bX}$-linear dual $\cT_{\bX}\congs \rho^{*}\cT_{\X}.$ This implies $\rho^*D_{\X}=D_{\bX}$ and shows (a).
Suppose that $\cM$ is a $D_{\X}$-module. 
Denote by $\tilde{\theta}=\sum_j f_j\otimes\theta_j$ with $f_j\in\cO_{\bX}, \theta_j\in\cT_{\X}$ the image of a vector field $\theta$ under the map $\cT_{\bX}\rightarrow \rho^{*}\cT_{\X}$. As in the algebraic setting \cite[1.3]{HTT}, one defines an action of $\cT_{\bX}$ on $\rho^{*}\cM$ via 
$$ \theta(f \otimes m):= \theta(f)\otimes m + \sum_{j}f f_j\otimes\theta_j(m)$$
where $f\in\cO_{\bX}, m\in\cM$. 
The action extends to a $D_{\bX}$-module structure on $\rho^{*}\cM$ which is functorial in $\cM$. This gives (b).
Finally, locally on $\X$ and by (a), the functor $\rho^*$ transforms a finite presentation of $\cM$ as $D_{\X}$-module into a finite presentation of $\rho^{*}\cM$ as $D_{\bX}$-module. Since $D_{\bX}$ and $D_{\X}$ are coherent sheaves of rings, this shows (c).
\end{proof}

\begin{prop}\label{prop-analytification2}  
Suppose that $\X$ is proper and $D_{\X}$-affine.
Let $\cM$ be a $D_{\X}$-module which is quasi-coherent as $\cO_{\X}$-module.
\be 

\item One has $H^{0}(\bX,\rho^*\cM)=H^{0}(\X,\cM)$ for any $i\geq 0$ and $H^{i}(\bX,\rho^*\cM)=0$ for all $i>0$.
\item $\rho^*$ induces an exact and fully faithful functor from $D_{\X}$-modules, which are quasi-coherent over $\cO_{\X}$, to $D_{\bX}$-modules. 
 
\ee
\end{prop}
\begin{proof}
Part (a) follows from \ref{prop-analytification0} and the $D_{\X}$-affinity. Moreover, $\cM\rightsquigarrow \rho^*\cM$ now has a left quasi-inverse given by taking global sections $M:=H^{0}(\bX,\rho^*\cM)=H^{0}(\X,\cM)$ followed by the functor $M\rightsquigarrow D_{\X}\otimes_{H^{0}(\X,\cD_{\X})} M$. This implies (b).
\end{proof}


\vskip5pt

Let $\tilde{\rho}: \mathscr{P}(\bX)\rightarrow\X$ be the canonical map.
Note that $\mathscr{P}(\bX)$ is the underlying topological space of the adic space $(\X^{\rm ad},\cO_{\X^{\rm ad}})$ associated to the finite type $K$-scheme $\X$. According to \cite[Remark 4.6(i)]{Hu94}, the morphism $\tilde{\rho}$ extends to a flat morphism of locally ringed spaces being the base change to $\X$ of the morphism ${\rm Spa}(K,K^{\circ})\rightarrow \Spec K$ induced by the identity $K\rightarrow K$. 

\begin{lem}\label{lem-zariskiclosedalg}
 Let  $\Y \subseteq\X$ be a closed subset and let $\bY=\rho^{-1}(\Y).$
 Then  $\bY\subseteq\bX$ is Zariski-closed and $\overline{\bY}=\tilde{\rho}^{-1}(\Y)$ in $\mathscr{P}(\bX)$.
\end{lem}
\begin{proof}
 It is clear that the subset $\bY\subseteq\bX$ is Zariski-closed.
 If $\bY$ has its reduced rigid-analytic structure \cite[Prop. 9.5.3.4]{BGR}, then $\overline{\bY}=\mathscr{P}(\bY)$. From the above description of the morphism $\tilde{\rho}$ as a base change of ${\rm Spa}(K,K^{\circ})\rightarrow \Spec K$, it is clear that $\mathscr{P}(\bY)=\tilde{\rho}^{-1}(\Y)$.
 \end{proof}

\begin{prop} \label{prop-supportpreservedA} Let $\cM$ be a coherent $\cD_{\X}$-module. Then $\Supp \cM \subseteq \X$ is closed and 
\[\Supp \rho^*\cM =\tilde{\rho}^{-1} ( \Supp \cM ).\]
\end{prop}
\begin{proof}
One has $\Supp \cM={\rm Char}(\cM) \cap T_{\X}^*\X$ inside the cotangent bundle $T^*\X$ and so this is a closed subset of $\X$. 
Note that $\mathscr{P}(\bX)$ is the underlying topological space of the adic space $(\X^{\rm ad},\cO_{\X^{\rm ad}})$ associated to the finite type $K$-scheme $\X$. The map $\tilde{\rho}$ extends to a morphism of locally ringed spaces as explained above.
One has $\widetilde{\rho^*\cM}= \tilde{\rho}^*\cM$ where $\tilde{\rho}^*$ denotes the analogue of $\rho^*$ for $(\X^{\rm ad},\cO_{\X^{\rm ad}})$. For
$x\in \X^{\rm ad}$ one has $$(\tilde{\rho}^*\cM)_x=\cO_{\X^{\rm ad},x}\utimes{\cO_{\X,\tilde{\rho}(x)}} \cM_{\tilde{\rho}(x)}.$$
The homomorphism $\cO_{\X,\tilde{\rho}(x)}\rightarrow\cO_{\X^{\rm ad},x}$ is flat and hence faithfully flat (being a flat 
local homomorphism between local rings). So $(\tilde{\rho}^*\cM)_x=0$ if and only if  $\cM_{\tilde{\rho}(x)}=0$, which shows $\Supp \tilde{\rho}^*\cM=\rho^{-1} (\Supp \cM )$.\end{proof}


\vskip5pt 

Let $\w\cD_{\bX}$ be the sheaf of analytic (infinite order) differential operators as constructed and studied in 
\cite{DCapOne}.\footnote{In the case of the flag variety, the associated sheaf $\widetilde{\w\cD_{\bX}}$ on the space $\mathscr{P}(\bX)$ was independently constructed and studied 
in \cite{HPSS} where it is called $\mathscr{D}_{\infty}$.} Let $\cC_{\bX}$ be the category of coadmissible $\w\cD_{\bX}$-modules. 
The natural inclusion $\cD_{\bX}\subset\w\cD_{\bX}$ gives rise to an extension functor 
$$ E_{\bX}:  \cD_{\bX}{\rm -mod} \longrightarrow  \w\cD_{\bX}{\rm -mod},\;\; \cM\rightsquigarrow \w\cD_{\bX}\otimes_{\cD_{\bX}}\cM$$ 

which is exact and faithful and takes coherent $\cD_{\bX}$-modules into $\cC_{\bX}$
\cite[7.2]{DCapThree}.

\begin{prop} \label{prop-supportpreservedB}
One has $$\Supp E_{\bX}(\cM) =  \Supp \cM $$
for any coherent $\cD_{\bX}$-module $\cM$.
\end{prop}
\begin{proof} Fix $x \in \sP(\bX)$; we will abuse notation and write $\cN_x$ to denote $\tilde{\cN}_x$ for every abelian sheaf $\cN$ on $\bX$. We will first show that $\w\cD_x$ is a faithful right $\cD_x$-module. Let $I$ be a left ideal in $\cD_x$ and suppose that $\w\cD_x \otimes_{\cD_x} (\cD_x/I) = 0$. Then $1 \in \w\cD_x \cdot I$, so we can find $Q_1,\ldots,Q_n \in \w\cD_x$ and $P_1,\ldots,P_n \in I$ such that $1 = \sum_{i=1}^n Q_iP_i$. We can find $\bU \in \bX_w(\cT)$ such that $\tilde{\bU} \ni x$, the maps $\w\cD(\bU) \to \w\cD_x$ and $\cD(\bU) \to \cD_x$ are injective and $Q_i \in \w\cD(\bU)$ and $P_i \in \cD(\bU)$ for each $i$. Then $1 = \sum_{i=1}^n Q_iP_i$ holds already in $\w\cD(\bU)$ so
\[\w\cD(\bU) \utimes{\cD(\bU)} \left( \frac{\cD(\bU)}{\sum_{i=1}^n \cD(\bU) P_i}\right) = 0.\]
But $\w\cD(\bU)$ is a faithfully flat $\w\cD(\bU)$-module by \cite[Theorem B]{DCapThree}, so $\sum_{i=1}^n\cD(\bU) P_i = \cD(\bU)$ and therefore $I = \cD_x$. Hence $\w\cD_x$ is a faithful right $\cD_x$-module as claimed. 

Now since $\cM$ is a coherent $\cD_{\bX}$-module, we can find $\bU \in \bX_w(\cT)$ such that $\tilde{\bU} \ni x$ and such that there is an exact sequence $\cD_{\bU}^m \to \cD_{\bU}^n \to \cM_{|\bU} \to 0$ for some integers $m, n$. Because the functors 
\[\cN \mapsto \w\cD_x \utimes{\cD_x} \cN_x\qmb{and} \cN \mapsto (\w\cD \utimes{\cD} \cN)_x=E_{\bX}(\cN)_x\]
are right exact and agree when $\cN$ is a free $\cD$-module of finite rank, we conclude using the Five Lemma that there is a natural isomorphism
\[ \w\cD_x \utimes{\cD_x} \cM_x \tocong E_{\bX}(\cM)_x\]
for every $x \in \sP(\bX)$. The faithfulness of $\w\cD_x$ as a right $\cD_x$-module established in the first paragraph now shows that $\cM_x \neq 0$ if and only if $E_{\bX}(\cM)_x \neq 0$, and the result follows.
\end{proof}

Combined with \ref{prop-supportpreservedA} we obtain:

\begin{cor}\label{cor-supportpreserved} Let $\cM$ be a coherent $\cD_{\X}$-module. One has 
$$\Supp E_{\bX}(\rho^*\cM) =\tilde{\rho}^{-1} ( \Supp \cM ).$$
\end{cor}

\section{Some complements to localization, induction and duality}
Let $K$ be a non-archimedean complete field of mixed characteristic $(0,p)$.
 Let $G$ be a $p$-adic Lie group acting continuously on a smooth rigid $K$-analytic space $\bX$.
Let $P\subset G$ be a closed subgroup. 

\subsection{Intersection obstructions and inverse images}\label{sec-IntObstr}
Let for a moment $G$ be an abstract group acting on a set $X$ and let $Y$ be a subset of $X$. 
Denote by $$G_Y:=\Stab_G(Y)=\{g\in G: gY\subseteq Y\}$$ the stabilizer of the set $Y$ in $G$ and define
the {\it intersection obstruction} of $Y$ in $G$ to be
\[ Z_Y:=\{ g\in G : Y\cap gY \neq \emptyset\} .\]
We have $G_Y\subseteq Z_Y$. Recall \cite[Def. 2.1.8]{EqDCapTwo} that the $G$-orbit $G.Y$ of $Y$ is called {\it regular in $X$} if $Z_Y=G_Y$.

\begin{lem}\label{lem-elementary} Let $f:\tilde{X}\rightarrow X$ be a $G$-equivariant map between two $G$-sets $\tilde{X}$ and $X$. Let $\tilde{Y}=f^{-1}(Y)$. 

\be 
\item One has $Z_{\tilde{Y}}\subseteq Z_Y$. If $f$ is surjective, then $Z_{\tilde{Y}}= Z_Y$. 

\item One has $G_{\tilde{Y}}\supseteq G_Y$. If $f$ is surjective, then $G_{\tilde{Y}}= G_Y.$ 
\item If $f$ is surjective, then: \text{$G.\tilde{Y}$ is regular in $\tilde{X} \Longleftrightarrow G.Y$ is regular in $X.$}
\ee 
\end{lem}

\begin{proof}
(a) Let $g \in Z_{\tilde{Y}}$. Then $f^{-1}(Y \cap g Y) = \tilde{Y} \cap g \tilde{Y} \neq \emptyset$, so $Y \cap g Y \neq \emptyset$ and $g \in Z_Y$.  Conversely, suppose $f$ surjective and $g\in Z_Y$. Then $Y \cap gY \neq \emptyset$ and hence $f^{-1}(Y \cap g Y) \neq \emptyset$ because $f$ is surjective. Hence $\tilde{Y} \cap g \tilde{Y} = f^{-1}(Y \cap gY)  \neq \emptyset$ as well, so $g \in Z_{\tilde{Y}}$. 

(b) Let $g\in G_Y$. Then $f(g\tilde{Y})=gf(\tilde{Y})\subseteq gY\subseteq Y$ so that $g\tilde{Y}\subseteq f^{-1}(Y)=\tilde{Y}$ which implies $g\in G_{\tilde{Y}}$. Conversely, suppose $f$ surjective so that $Y = f(\tilde{Y})$. If $g\in G_{\tilde{Y}}$, then $gY = gf(\tilde{Y}) = f(g\tilde{Y}) \subseteq f(\tilde{Y}) = Y$ shows that $g \in G_Y$.

(c) This follows from (a) and (b).
\end{proof}

\subsection{Induction and restriction}\label{IndRes}
We start by recalling the geometric induction functor 
$$ \ind_P^G: \cC_{\bX /P} \rightarrow  \cC_{\bX /G}$$
from \cite{EqDCapTwo}. For a local description, let $\bU\in \bX_w(\cT)$. Then by definition \cite[2.2.12]{EqDCapTwo}
$$ \ind_P^G(\cN)(\bU)= \invlim\limits\bigoplus\limits_{Z\in H\setminus G / P} \lim\limits_{s\in Z}\w\cD(\bU,H) \underset{\w\cD(\bU,H\cap ^sP)}{\w\otimes} [s]\cN(s^{-1}\bU).$$
where the first inverse limit is over all $\bU$-small compact open subgroups $H$ of $G$. 

\vskip5pt
For any subset $S\subseteq\bX$ we denote by $\cC^{S}_{\bX /P}$ the full subcategory of $\cC_{\bX /P}$ consisting of those $\cM\in \cC_{\bX /P}$ which are supported on $S$, i.e. $\cM|_{\bV}=0$ for any admissible open subset $\bV$ of $\bX\setminus S$.

\begin{thm}\label{inductionequiv} Suppose that $\bX$ is separated. 
Let $\bY\subseteq\bX$ be a Zariski closed subset of $\bX$.
Suppose that $\bY$ is irreducible and quasi-compact, has a regular $G$-orbit in $\bX$ and a co-compact stabilizer $G_{\bY}$ in $G$.
Then the functor $\ind_{G_{\bY}}^G$ induces an equivalence of categories 
$$\ind_{G_{\bY}}^G: \cC^{\bY}_{\bX / G_{\bY}} \congs  \cC^{G\bY}_{\bX /G}.$$
A quasi-inverse is given by the functor $\cH^0_{\bY}$ of sections supported on $\bY$.
\end{thm}
\begin{proof} \cite[Theorem A]{EqDCapTwo}.
\end{proof}

Now let $\cN\in  \cC_{\bX /P}$. In the following, we first construct a certain morphism 
\[\alpha_{\cN}: \cN\rightarrow  \ind_P^G(\cN)\] 
in $\Frech(P - \cD_\bX)$ which is natural in $\cN$.
Let $\bU\in \bX_w(\cT)$. As we have recalled above,
$$ \ind_P^G(\cN)(\bU)= \invlim\limits\bigoplus\limits_{Z\in H\setminus G / P} \lim\limits_{s\in Z}\w\cD(\bU,H) \underset{\w\cD(\bU,H\cap ^sP)}{\w\otimes} [s]\cN(s^{-1}\bU).$$
where the first inverse limit is over all $\bU$-small compact open subgroups $H$ of $G$. For the double class $Z_0=HP$, we have a
$\w\cD(\bU,H\cap P)$-linear map
$$\cN(\bU)\rightarrow \w\cD(\bU,H) \underset{\w\cD(\bU,H\cap P)}{\w\otimes} \cN(\bU)\rightarrow \lim\limits_{s\in Z_0}\w\cD(\bU,H) \underset{\w\cD(\bU,H\cap ^sP)}{\w\otimes} [s]\cN(s^{-1}\bU),$$
where the first map is the canonical map $x\mapsto 1\w\otimes x$ and the second map is the inverse of the canonical isomorphism \cite[2.2.9 (a)]{EqDCapTwo}. This gives a
$\w\cD(\bU,H\cap P)$-linear map $$\cN(\bU)\rightarrow \ind_P^G(\cN)(\bU).$$ If $\bV\subseteq\bU$ is an affinoid subdomain, then this map is compatible with the restriction map $\tau^\bU_\bV$ and therefore extends to a morphism of sheaves
 $\cN\rightarrow  \ind_P^G(\cN)$ on $\bX_w(\cT)$ and then on $\bX$. The morphism lies in $\Frech(P - \cD_\bX)$. This defines $\alpha_{\cN}$.
 
 \vskip5pt
 
There is the obvious restriction functor $\Res^G_P$ from $\Frech(G - \cD_\bX)$ to $\Frech(P - \cD_\bX)$.

\begin{prop}\label{Adjunction}
Let $\cN\in  \cC_{\bX /P}, \cM\in \cC_{\bX /G}$ and $f: \cN\rightarrow \Res^G_P(\cM)$ a morphism in $\Frech(P - \cD_\bX)$.
Then there is a morphism \[\bar{f}: \ind_P^G(\cN)\rightarrow \cM\] in  $\Frech(G - \cD_\bX)$ such that $f=\bar{f}\circ\alpha_{\cN}$. This induces a $K$-linear isomorphism 
\[ \Hom_{\Frech(P - \cD_\bX)} (  \cN,  \Res^G_P(\cM) ) \congs  \Hom_{\Frech(G - \cD_\bX)}  (  \ind_P^G(\cN), \cM )
\]
which is natural in $\cN$ and $\cM$. 
\end{prop}
\begin{proof} Suppose that $(\bU,H)$ is small. Let $Z\in H\setminus G / P$ and $s\in Z$.
There is the $\w\cD(\bU, H \cap ^sP)$-linear  morphism
$$ [s]s_*\cN(\bU)=\cN(s^{-1}\bU)\stackrel{f}{\longrightarrow} \cM(s^{-1}\bU)\longrightarrow \cM(\bU)$$ where the second map equals $g^\cM(\bU)^{-1}$ for $g=s^{-1}$. The map extends to a $\w\cD(\bU, H)$-linear morphism
$$\w\cD(\bU,H) \underset{\w\cD(\bU, H \cap ^sP)}{\w\otimes} [s]\cN(s^{-1}\bU)\rightarrow \cM(\bU)$$ and then to a morphism $\ind_P^G(\cN)(\bU)\rightarrow \cM(\bU)$. This gives the morphism $\bar{f}$. The inverse to the map $f\mapsto \bar{f}$ is given by precomposing a morphism with $\alpha_{\cN}$. 
\end{proof}
We deduce the right exactness of geometric induction (although this will not be used in the following). 
\begin{cor}
The functor $ \ind_P^G: \cC_{\bX /P} \rightarrow  \cC_{\bX /G}$ is right-exact.
\end{cor}
\begin{proof} This is a variant of the standard argument of deducing exactness properties from an adjoint pair of functors. 
Let $L= (\ind_P^G)^{\op}: \cC_{\bX /P}^{\op}\rightarrow \cC_{\bX /G}^{\op}$ be the opposite functor and similarly $R=  (\Res^G_P)^{\op}$. 
We will show that $L$ is left-exact. Let therefore $0\rightarrow \cN_1\rightarrow \cN_0\rightarrow \cN_2$ be an exact sequence in  
 $\cC_{\bX /P}^{\op}$ and let $\cM\in \cC_{\bX /G}^{\op}$. Then 
\begin{equation}\label{eq-adjoint1} 0\rightarrow \Hom(R(\cM),\cN_1) \rightarrow \Hom(R(\cM),\cN_0)\rightarrow \Hom(R(\cM),\cN_2)
\end{equation}
is an exact sequence of abelian groups. Indeed, if a morphism $f:R(\cM)\rightarrow \cN_0$ maps to zero in 
  $\Hom(R(\cM),\cN_2)$, then the morphism $\cN_0\rightarrow R(\cM)$ in $\Frech(P - \cD_\bX)$ factors through the quotient $\cN_1$ of $\cN_0$. Since the canonical topology on local sections over open affinoids $\bU\in \bX_w(\cT)$ of the coadmissible module $\cN_1$ is the quotient topology, the induced morphism $\cN_1\rightarrow R(\cM)$ is continuous, i.e. lies in $\Frech(P - \cD_\bX)$. Its opposite is then a preimage of $f$ in the sequence (\ref{eq-adjoint1}). This shows the exactness of (\ref{eq-adjoint1}) in the middle and the exactness on the left is clear.
  By Proposition \ref{Adjunction} we obtain an exact sequence
$$  0\rightarrow \Hom(\cM,L(\cN_1)) \rightarrow \Hom(\cM,L(\cN_0))\rightarrow \Hom(\cM,L(\cN_2))$$
for any $\cM\in  \cC_{\bX /G}^{\op}$.
Since $\cM$ is coadmissible, one has  $$\Hom_{\Frech(G - \cD_\bX)^{\op}}(\cM,L(\cN_i))=\Hom_{ \cC_{\bX /G}^{\op}}(\cM,L(\cN_i))$$
for all $i$ and so the Yoneda lemma in the abelian category $\cC_{\bX /G}^{\op}$
implies that 
 $$0\rightarrow L(\cN_1) \rightarrow L(\cN_0) \rightarrow L(\cN_2)$$
  is exact. Hence $L: \cC_{\bX /P}^{\op}\rightarrow  \cC_{\bX /G}^{\op}$ is left-exact.
 \end{proof}
 \begin{prop}\label{prop_iwasawa} 
 Let $\cN\in \cC_{\bX /P}$. Suppose that there exists an open subgroup $G_0\subset G$ such that 
 $G_0P=G$. Let $P_0=G_0\cap P$. There is a natural isomorphism in $\cC_{\bX /G_0}$
 $$\ind_{P_0}^{G_0}(\Res^{P}_{P_0}\cN)\congs \Res^G_{G_0} \ind_P^G(\cN).$$
 \end{prop}
 \begin{proof} Let $\cN_0=\Res^{P}_{P_0}\cN$ and 
 $\cM_0=\Res^G_{G_0} \ind_P^G(\cN)$. It suffices to construct a bijective morphism of $G_0$-equivariant sheaves of $\cD_\bX$-modules on $\bX_w(\cT)$
 $$\ind_{P_0}^{G_0}(\cN_0)\congs \cM_0$$
compatible with local Fr\'echet topologies. 
Let $\bU\in \bX_w(\cT)$. Then 
$$\ind_{P_0}^{G_0}(\cN_0)(\bU)=\varprojlim_H \bigoplus\limits_{Z\in H\setminus G_0 / P_0}\;
\varprojlim_{s\in Z}\; \w\cD(\bU,H) \underset{\w\cD(\bU, H \cap ^sP)}{\w\otimes} [s]\cN_0(s^{-1}\bU),$$
where the limit runs through all open $\bU$-small subgroups $H\subseteq G_0$. On the other hand, $$ \cM_0(\bU)=\varprojlim_H \bigoplus\limits_{Z\in H\setminus G / P}\;
\varprojlim_{s\in Z}\; \w\cD(\bU,H) \underset{\w\cD(\bU, H \cap ^sP)}{\w\otimes} [s]\cN(s^{-1}\bU),$$ where we may again assume, since $G_0$ is open in $G$, that the limit runs through $\bU$-small subgroups $H$ contained in $G_0$ .  The equality $G_0P=G$ implies $G_0/P_0\simeq G/P$ as left $G_0$-sets, whence $H\setminus G_0 / P_0\simeq H\setminus G / P$. The latter isomorphism thus induces a bijective morphism 
$$\ind_{P_0}^{G_0}(\cN_0)(\bU)\congs \cM_0(\bU)$$
compatible with local Fr\'echet topologies. It is also compatible with the restriction maps, if $\bV\subset\bU$ is an affinoid subdomain in 
 $\bX_w(\cT)$ and induces the desired isomorphism $\ind_{P_0}^{G_0}(\cN_0)\simeq \cM_0$. \end{proof}
 
 \subsection{Induction and side-changing}\label{IndSide}
 
 The aim of this subsection is to show that induction commutes with the side-changing operations.
 Recall from \cite[3.1.15]{EqDCapTwo} that the functors $\Omega_\bX\otimes (-)$ and 
 ${\mathcal Hom} (\Omega_\bX,-)$ are mutually inverse equivalences of categories 
 between $\cC_{\bX /G}$ and ${^r\cC}_{\bX /G}$. Denote by ${^r\ind}_P^G$ the right module version of the induction functor \cite[2.2]{EqDCapTwo} from right $P$-equivariant $\cD_\bX$-modules to right $G$-equivariant $\cD_\bX$-modules.
\begin{thm} \label{thm_IndSide} Let $\cN\in \cC_{\bX /P}$. There is a natural isomorphism in 
${^r\cC}_{\bX /G}$

\[ {^r\ind}_P^G(\Omega_\bX\otimes \cN)\congs \Omega_\bX\otimes \ind_P^G\cN.\]
\end{thm}
\begin{proof} It suffices to construct an isomorphism between ${^r\ind}_P^G(\Omega_\bX\otimes \cN)$ and $\Omega_\bX\otimes \ind_P^G\cN$ as sheaves of right equivariant $\cD_\bX$-modules on $\bX_w(\cT)$. Let $\bU\in \bX_w(\cT)$. According to \cite[2.2.12]{EqDCapTwo},
$$ \ind_P^G(\cN)(\bU)= \invlim\limits\bigoplus\limits_{Z\in H\setminus G / P} \lim\limits_{s\in Z}\w\cD(\bU,H) \underset{\w\cD(\bU,H\cap ^sP)}{\w\otimes} [s]\cN(s^{-1}\bU).$$
where the first inverse limit is over all $\bU$-small compact open subgroups $H$ of $G$. 

Fix $H$ and $s\in Z$ as above and set $Q:= [s]\cN(s^{-1}\bU)$. Consider Fréchet-Stein
presentations \begin{center}
$\w\cD(\bU,H\cap ^sP)\simeq \varprojlim_n S_n$ and $\wideparen{\mathcal{D}}(\bU,H)\simeq \varprojlim_n T_n$
\end{center} 
with noetherian Banach algebras $S_n$ and $T_n$. 
Using a version of  \cite[3.1.13]{EqDCapTwo}, the two right $S_{n}$-linear morphisms 
\begin{center}$\Omega(\bU)\otimes_{\cO(\bU)} S_{n}\rightarrow  \Omega(\bU)\otimes_{\cO(\bU)} T_n$
 \hskip5pt and \hskip5pt $\Omega(\bU)\oslash_{\cO(\bU)} S_{n}\rightarrow  \Omega(\bU)\oslash_{\cO(\bU)} T_n$
\end{center} 
(using the notation of loc.cit.) coming from functoriality are actually isomorphic. 
Let $Q_n=S_n\otimes_S Q$. 
As in the proof of \cite[3.1.14]{EqDCapTwo}, the right $T_n$-linear map 
$$\theta_{Q_n}: \big(\Omega(\bU) \otimes_{\cO(\bU)} Q_n\big){\otimes}_{S_n}T_n  \congs 
\Omega(\bU) \otimes_{\cO(\bU)} \big(T_n {\otimes}_{S_n} Q_n\big)$$
given by $\theta_{Q_n}((\omega\otimes m)\otimes r)=(\omega \otimes 1\otimes m)r$
is an isomorphism. It is compatible in $n$ and induces an isomorphism $$\theta_Q: \big(\Omega(\bU) \otimes_{\cO(\bU)} Q\big)\underset{\w\cD(\bU,H\cap ^sP)}{\w\otimes} \w\cD(\bU,H) \congs 
\Omega(\bU) \otimes_{\cO(\bU)} \big(\w\cD(\bU,H) \underset{\w\cD(\bU,H\cap ^sP)}{\w\otimes} Q\big)$$
of right $\w\cD(\bU,H)$-modules. Next, we may identify canonically 
\begin{center} 
$ [s]\big(\Omega(s^{-1}\bU) \otimes_{\cO(s^{-1}\bU)}\cN(s^{-1}\bU)\big)$
and 
$ \Omega(\bU) \otimes_{\cO(\bU)} [s]\cN(s^{-1}\bU)$
\end{center}
as right $\w\cD(\bU,H\cap ^sP)$-modules. Passing to the limit in $s\in Z$, taking the direct sum over $Z$ and passing finally to the limit over $H$ yields therefore an isomorphism  
 $$\big({^r\ind}_P^G(\Omega_\bX\otimes \cN)\big)(\bU)\congs \big(\Omega_\bX\otimes \ind_P^G\cN\big)(\bU)$$
 of right $\w\cD(\bU,H)$-modules. It is compatible with restriction maps and yields the desired isomorphism of sheaves on $\bX_w(\cT)$.
 \end{proof}

\subsection{Induction and duality}\label{IndDual} The aim of this subsection is to prove the following theorem. Recall from \cite[Definition 5.9]{VuThesis} the full subcategory $\cC^{\rm wh}_{\bX /G}$ of $\cC_{\bX /G}$ consisting of {\it weakly holonomic} modules. If Bernstein's inequality holds in $\cC_{\bX /G}$ in the sense of \cite[5.2]{VuThesis}, then by \cite[Definition 5.14]{VuThesis}, there is a duality functor $\mathbb{D}_G$ on $\cC^{\rm wh}_{\bX /G}$.
\begin{thm} \label{thm_IndDual} Assume Bernstein's inequality holds
in $\cC_{\bX /P}$ and $\cC_{\bX /G}$. Let $\cN\in  \cC^{\rm wh}_{\bX /P}$. There is a natural isomorphism in $\cC^{\rm wh}_{\bX /G}$
\[\mathbb{D}_G(\ind_P^G\cN)\congs \ind_P^G(\mathbb{D}_P\cN).\]
\end{thm}
\begin{proof} Let $\cM=\ind_P^G\cN$. Let $d=\dim\bX$ and denote by ${^r\mathcal{E}}^d_G(\cM)$ the right $G$-equivariant $\cD_\bX$-module on $\bX_w(\cT)$, equal to the $d$-th Ext sheaf of $\cM$ (denoted by $E^d(\cM)$ in \cite[Theorem 4.25/29]{VuThesis}). In particular, $\mathbb{D}_G={\mathcal Hom} (\Omega_\bX,-)\circ {^r\mathcal{E}}^d_G$, with the side-changing functor ${\mathcal Hom} (\Omega_\bX,-).$

In view of Thm. \ref{thm_IndSide}, it suffices to construct a bijective morphism of right equivariant $\cD_\bX$-modules on $\bX_w(\cT)$
$${^r\mathcal{E}}_G^d(\cM)\congs {^r\ind}_P^G( {^r\mathcal{E}}_P(\cN)),$$
compatible with local Fr\'echet topologies, where ${^r\ind}_P^G$ denotes the induction functor for right modules, as in the preceding subsection. Let $\bU\in \bX_w(\cT)$ and let $H$ be a $\bU$-small subgroup of $G$. Fix a system $s_1,...,s_m$ of representatives for the double cosets in $H\setminus G / P$. On the one hand,
$$\Ext^d_{\wideparen{\mathcal{D}}(\bU,H)}(\mathcal{M}(\bU), \wideparen{\mathcal{D}}(\bU,H))\simeq \bigoplus\limits_{i=1}^m \Ext^d_{\wideparen{\mathcal{D}}(\bU,H)}(
\w\cD(\bU,H) \underset{\w\cD(\bU, H \cap ^{s_i}P)}{\w\otimes} [s_i]\cN(s_i^{-1}\bU), \wideparen{\mathcal{D}}(\bU,H)). $$
Using the proof of \cite[Lemma $2.5.3$]{EqDCapTwo} applied to the morphism $\wideparen{\mathcal{D}}(\bU,{}^{s_i}P\cap H) \longrightarrow \wideparen{\mathcal{D}}(\bU,H)$, \cite[Lemma 8.4]{ST} and the fact that the twisting functor $[s_i]$ commutes with Ext groups, we see that this is isomorphic to
$$ \bigoplus\limits_{i=1}^m\Ext^d_{\w\cD(\bU, H^{s_i} \cap P)}(
\cN(s_i^{-1}\bU),\w\cD(\bU, H^{s_i} \cap P))[s_i]\underset{\w\cD(\bU, H \cap ^{s_i}P)}{\w\otimes} \w\cD(\bU,H).$$
By definition of ${^r\mathcal{E}}^d_P$, the latter module is canonically isomorphic to 
$$ \bigoplus\limits_{i=1}^m {^r\mathcal{E}}^d_P(\cN)(s_i^{-1}\bU)[s_i]\underset{\w\cD(\bU, H \cap ^{s_i}P)}{\w\otimes} \w\cD(\bU,H) \simeq  ({^r\ind}_P^G {^r\mathcal{E}}^d_P(\cN))(\bU).$$ 
This isomorphism is right $\w\cD(\bU,H)$-linear and therefore is compatible with the Fr\'echet topologies. Moreover, it is compatible with variation in $H$. Hence, taking the limit over all $\bU$-small subgroups $H$ of $G$ and recalling from \cite[Def. 4.12]{VuThesis} that
$${^r\mathcal{E}}_G^d (\mathcal{M})(\bU)= \varprojlim_H \Ext^d_{\wideparen{\mathcal{D}}(\bU,H)}(\mathcal{M}(\bU),\wideparen{\mathcal{D}}(\bU,H)),$$
we obtain a bijective morphism 
\[{^r\mathcal{E}}_G^d (\mathcal{M})(\bU)\congs  (\ind_P^G {^r\mathcal{E}}^d_P(\cN))(\bU).\]

It is compatible with the restriction maps if $\bV\subset\bU$ is an affinoid subdomain in 
 $\bX_w(\cT)$ and induces the desired isomorphism  ${^r\mathcal{E}}_G^d(\cM)\simeq \ind_P^G {^r\mathcal{E}}^d_P(\cN)$.
\end{proof}
\subsection{Duality and pushforward}\label{DualPush} 
Denote by $\Hol(\cD_{\bX})$ the $\cD_\bX$-modules of minimal dimension, cf. \cite[7.2]{DCapThree}. Let $\mathbb{D}^{\rm cl}$ denote the classical duality on $\Hol(\cD_{\bX})$. 
Note that $\cD_\bX$ has a natural $G$-equivariant structure, so we may speak about $G$-equivariant $\cD_\bX$-modules. Our first lemma is well-known, but we could not find a reference. 
\begin{lem}\label{lem_func}
Let $R,R'$ be two rings and $f: R\rightarrow R'$ a ring isomorphism. Let $M$ and $M'$ be a module over $R$ and $R'$ respectively. Let $i\geq 0$. Any $R$-linear map $M\rightarrow M'$ induces a right $R$-linear map 
$\Ext^{i}_{R'}(M',R')\rightarrow   \Ext^{i}_{R}(M,R).$ \end{lem}
\begin{proof}
By functoriality, the map in question induces a $R$-linear map 
$\Ext^{i}_R(M',R)\rightarrow   \Ext^{i}_{R}(M,R).$ 
A projective resolution $P_\bullet\rightarrow M'$ as $R'$-module remains a projective resolution of $M$ as $R$-module. The map of complexes
$\Hom_R(P_\bullet,R)\rightarrow \Hom_{R'}(P_\bullet,R'), F\mapsto f\circ F$
induces a bijection $\Ext^{i}_{R}(M',R)\simeq \Ext^{i}_{R'}(M',R')$. 
Combining the inverse of this map with the first map gives the result.
\end{proof}
\begin{prop}\label{prop_G_str} If $\cM\in \Hol(\cD_{\bX})$ is $G$-equivariant, then its dual 
$\mathbb{D}^{\rm cl}\cM$ has a natural $G$-equivariant structure.  
\end{prop}
\begin{proof} 
Let $d=\dim\bX$, let 
$\bU\subset\bX$ be an affinoid subdomain and $g\in G$. Applying \ref{lem_func} to the ring isomorphism
$g^{\mathcal{D}}:\mathcal{D}(\bU)\rightarrow \mathcal{D}(g\bU)$, the modules $\cM(\bU)$ and $\cM(g\bU)$ and the map $g^{\cM}: \cM(\bU)\rightarrow \cM(g\bU)$ yields a $\mathcal{D}(\bU)$-linear map $$\Ext^d_{\mathcal{D}(\bU)}(\cM(\bU), \mathcal{D}(\bU))\rightarrow \Ext^d_{\mathcal{D}(g\bU)}(\cM(g\bU), \mathcal{D}(g\bU)).$$ This defines on 
 $\Ext^d_{\mathcal{D}_\bX}(\cM, \mathcal{D}_\bX)$ the structure of a $G$-equivariant right $\mathcal{D}_\bX$-module. Applying the side-changing functor ${\mathcal Hom} (\Omega_\bX,-)$ and using the natural $G$-equivariant structure on $\Omega_\bX$ produces a $G$-equivariant structure on $\mathbb{D}^{\rm cl}\cM$.
\end{proof}
\vskip5pt 
Now let $i: \bY\subset\bX$ be a smooth Zariski closed subvariety. 
Denote by $i_+: \mathcal{C}_{\bY}\rightarrow\mathcal{C}_{\bX}$ the direct image
and by $\mathbb{D}$ the duality on $\mathcal{C}^{\rm wh}_{\bX}.$ Note that $\wideparen{\mathcal{D}}_\bX$ has a natural $G$-equivariant structure, so we may speak about $G$-equivariant $\wideparen{\mathcal{D}}_\bX$-modules. 
\begin{prop} \label{prop_SelfDual} 
The module $i_+\cO_\bY$
is self-dual, i.e one has a $\wideparen{\mathcal{D}}_\bX$-linear isomorphism $$\mathbb{D}i_+\cO_\bY\simeq i_+\cO_\bY.$$ If $i_+\cO_\bY$ is $G$-equivariant, then so is $\mathbb{D}i_+\cO_\bY$ and the latter isomorphism is $G$-equivariant.
\end{prop}
\begin{proof} 
Let $\cB_{\bY\mid\bX}:=i_+\cO_\bY$. Similarly, let $\cB^{\rm cl}_{\bY\mid\bX}:=i^{\rm cl}_+\cO_\bY$, where $i^{\rm cl}_+$ denotes the classical push-forward functor from  $\Hol(\cD_\bY)$ to $\Hol(\cD_{\bX})$. Then $\cB^{\rm cl}_{\bY\mid\bX}$ is self-dual with respect to the classical duality on $\Hol(\cD_{\bX})$, arguing as \cite[2.6.9]{HTT}. Since $i_+$ commutes with the extension functors 
$E_\bX$ and $E_\bY$, cf. \cite[7.3]{DCapThree}, and since $E_\bY\cO_\bY\simeq \cO_\bY$, one has $\cB_{\bY\mid\bX}\simeq E_\bX\cB^{\rm cl}_{\bY\mid\bX}$. Moreover, the extension functor $E_\bX$ intertwines the dualities on $\Hol(\cD_\bX)$ and $\mathcal{C}^{\rm wh}_{\bX}$, cf. \cite[3.1]{DCapThree}. Hence
$$\mathbb{D}\cB_{\bY\mid\bX}\simeq E_\bX  \mathbb{D}^{\rm cl}\cB^{\rm cl}_{\bY\mid\bX}\simeq E_\bX\cB^{\rm cl}_{\bY\mid\bX}\simeq \cB_{\bY\mid\bX}.$$
Now suppose that $\cB_{\bY\mid\bX}$ has a $G$-equivariant structure. The existence of a $G$-equivariant structure on $\mathbb{D}\cB_{\bY\mid\bX}$ follows very similarly to \ref{prop_G_str}, using the $G$-equivariant structure $\{g^{\wideparen{\mathcal{D}}}: g\in G \}$.
As we have just seen, the isomorphism $\mathbb{D}\cB_{\bY\mid\bX}\simeq \cB_{\bY\mid\bX}$ comes by extension from the classical isomorphism \cite[2.7.2]{HTT}. This makes it possible to verify $G$-equivariance by direct inspection.
\end{proof}

Now assume that $\bY$ is $G$-stable. We then have the $G$-equivariant push-forward functor $i_+^G: \mathcal{C}_{\bY/G}\rightarrow \mathcal{C}_{\bX/G}$. 

\begin{prop} \label{prop_resG1} 
As $\wideparen{\mathcal{D}}_\bX$-modules
$$\Res^G_1i_+^G\cO_\bY\simeq i_+\cO_\bY,$$
where the functor $\Res^G_1$ forgets the $G$-equivariant structure.
\end{prop}
\begin{proof}
We may check this on the basis $\cB$ for $\bX$ described in \cite[Def. 4.4.1]{EqDCapTwo}.
Let $\bU\in\cB$ with $\bU\cap\bY\neq\emptyset$ given by an ideal $I\subset\cO(\bU)$ and choose a corresponding basis 
$\partial_1,...,\partial_d$ of a free $\cA$-Lie lattice $\cL$ of $\cT(\bU)$, for some affine formal model $\cA\subset\cO(\bU)$, and let $H$ be $\bU$-good, cf. \cite[Lem. 4.4.2]{EqDCapTwo}. Let $\cI= I\cap \cA$ and denote by $\cN$ the quotient $\cN_\cL(\cI)/\cI\cL$. Choose a good chain $H_n$ in $H$ for $\mathcal{L}$. Then put
\begin{center}
$S:=\wideparen{\mathcal{D}}(\bU\cap\bY,H)\simeq \varprojlim_n S_n$ with $S_n:= W_n\rtimes_{H_n} H$ and $W_n:=\widehat{U(\pi^n\mathcal{N})}_K$
\end{center}
and
\begin{center}
$T:=\wideparen{\mathcal{D}}(\bU,H)\simeq \varprojlim_n T_n$ with 
$T_n:= U_n\rtimes_{H_n} H$ and $U_n=\widehat{U(\pi^n(\mathcal{L}))}_K.$
\end{center}
Write the coadmissible right $S$-module $N:=\Omega_\bY(\bU\cap\bY)$ as $N=\varprojlim_n N_n$ with finitely generated right $S_n$-modules. By construction of the direct image $i^G_{+,r}$ for right modules, the coadmissible $T$-module 
$M:=(i_{+,r}^{G}\Omega_\bY)(\bU)$ admits the presentation $M=\varprojlim_n M_n$ where
$M_n=N_n\otimes_{S_n} T_n/ IT_n$. But according to \cite[3.3.6]{EqDCapTwo} and its proof, one has 
$$ N_n\otimes_{S_n} T_n/ IT_n=N_n\otimes_{W_n} U_n/ IU_n,$$
compatibly in $n$. Passing to the limit over $n$ yields 
$$ (i_{+}^G\cO_\bY)(\bU)=({\mathcal Hom} (\Omega_\bX,-)\circ i_{+,r}^G\Omega_\bY)(\bU)\simeq ({\mathcal Hom} (\Omega_\bX,-)\circ i_{+,r}\Omega_\bY)(\bU)=i_{+}\cO_\bY(\bU).$$ 
as $\wideparen{\mathcal{D}}(\bU)$-modules. 
This isomorphism is compatible with restriction maps relative to the inclusion of an affinoid subdomain $\bV\subset\bU$ in $\cB$ and induces the asserted isomorphism 
\[\Res^G_1 i_{+}^G\cO_\bY\simeq i_{+}\cO_\bY.  \qedhere\]
\end{proof}
Before we come to the main result in this subsection, we need an auxiliary result on Ext groups over crossed product rings. For crossed product rings and their basic properties, we refer to \cite{Pass}. Let $T=R*G$ be a crossed product ring and let $M$ be a left $T$-module. We will consider the two Ext groups $\Ext^{i}_R(M,R)$ and $\Ext^{i}_T(M,T)$ for any $i\geq 0$.

\vskip5pt

Let $\bar{g}\in T$ for some $g\in G$. We may apply \ref{lem_func}
to the ring isomorphism $R\rightarrow R, r\mapsto \bar{g}r\bar{g}^{-1}$ and the map 
$M\rightarrow M, m\mapsto \bar{g}m$, which is linear relative to this isomorphism. This produces a $R$-semilinear map $$\bar{g}: \Ext^{i}_R(M,R)\longrightarrow  \Ext^{i}_R(M,R).$$ 
For example, if $i=0$ and $\lambda\in \Hom_R(M,R)$, then $\lambda\bar{g}\in \Hom_R(M,R)$ is given as
$m\mapsto \bar{g}^{-1}\lambda(\bar{g}m)\bar{g}$. On the other hand, the Ext group $\Ext^{i}_T(M,T)$ is naturally a right $T$-module, which induces an action of $\bar{g}$ by right multiplication on this group for any $g\in G$. Now the left-version of \cite[Lemma 5.4]{ArdBro2007} gives a canonical isomorphism 
$$\Ext^{i}_T(M,T)\simeq \Ext^{i}_R(M,R)$$ as right $R$-modules. 
\begin{lem} \label{lem_crossed_prod2}Let $T=R*G$ be a crossed product ring for some group $G$ and let $M$ be a left $T$-module.  For any $i\geq 0$, the above isomorphism as right $R$-modules 
$$\Ext^{i}_T(M,T)\simeq \Ext^{i}_R(M,R)$$
intertwines the actions of $\bar{g}$ on both sides, for any $g\in G$.
\end{lem} 
\begin{proof} Let $g\in G$. By construction of the isomorphism \cite[Lemma 5.4]{ArdBro2007} we may use a projective resolution of the $T$-module $M$ (which is then also a projective resolution of $M$ as $R$-module) to reduce the verification of the intertwining property for $\bar{g}$ to the case $i=0$. Next, let us review the construction of the isomorphism 
$$\Hom_T(M,T)\simeq \Hom_R(M,R)$$ as right $R$-modules from loc.cit.
First, $\Hom_R(T,R)$ is a $(T,R)$-bimodule, where the left $T$-module structure is given as 
$(t_0f)(t)=f(tt_0)$ for any $t_0,t\in T$ and $f\in\Hom_R(T,R)$.
By tensor-hom adjunction, we have the isomorphism of right $R$-modules
$$ F: \Hom_R(M,R)\longrightarrow \Hom_T(M,\Hom_R(T,R)),\lambda\mapsto F_\lambda.$$
Here, the map $F_\lambda(m)\in \Hom_R(T,R)$, for any $m\in M$, is given as $t\mapsto \lambda(tm)$. Furthermore, the map 
$$\alpha: \Hom_R(T,R)\longrightarrow T, f\mapsto \sum_{g\in G} \bar{g}^{-1}f(\bar{g})$$
is a $(T,R)$-bimodule isomorphism. Combining $F$ and $\alpha$ yields the canonical isomorphism 
$$\Hom_T(M,T)\simeq \Hom_R(M,R)$$ as right $R$-modules. To see that this latter isomorphism intertwines the $\bar{g}$-actions, it suffices to show that both $F$ and $\alpha$ intertwine the $\bar{g}$-actions in a suitable sense. To start with, $\Hom_R(T,R)$ has a right $\bar{g}$-action, as we have just seen above, whence $\Hom_T(M,\Hom_R(T,R))$ has a right $\bar{g}$-action through right multiplication. To show that $F$ intertwines these actions reduces to show that for fixed $m\in M$, one has $F_{\lambda\bar{g}}(m)=(F_{\lambda}\bar{g})(m)$ as functions on $T$. But 
for $t\in T$, 
$$F_{\lambda\bar{g}}(m)(t)=(\lambda\bar{g})(tm)=\bar{g}^{-1}\lambda(\bar{g}tm)\bar{g}=
\bar{g}^{-1}(F_{\lambda}(m)(\bar{g}t))\bar{g}=(F_{\lambda}(m)\bar{g})(t)=(F_{\lambda}\bar{g})(m)(t).$$
In a final step, we show that $\alpha$ intertwines the $\bar{g}$-actions, i.e.
$\alpha_{f\bar{g}}=\alpha(f)\bar{g}$. Now $\bar{g}\bar{h}=r_h\overline{gh}$ for some $r_h\in R$, whence $\bar{h}=\bar{g}^{-1}r_h\overline{gh}$, which gives
\[ \begin{array}{lllllll}  \alpha_{f\bar{g}}&= &\sum_{h\in G} \bar{h}^{-1}(f\bar{g})(\bar{h}) &= &\sum_{h\in G} \bar{h}^{-1}\bar{g}^{-1}f(\bar{g}\bar{h})\bar{g}&= &\sum_{h\in G} \bar{h}^{-1}\bar{g}^{-1}f(r_h\overline{gh})\bar{g}\\ \; \\
&& & &&= &\sum_{h\in G} \overline{gh}^{-1}f(\overline{gh})\bar{g}\\ \; \\
& & & &&= &\alpha(f)\bar{g}.
\end{array}\]
\end{proof}

If Bernstein's inequality holds in $\cC_{\bX /G}$, then there is the duality functor $\mathbb{D}_G$ on $\mathcal{C}^{\rm wh}_{\bX/G}$.
\begin{thm} \label{prop_SelfGDual} 
Assume that Bernstein's inequality holds
in $\cC_{\bX /G}$. Then $i_+^G\cO_\bY$
is self-dual, i.e $\mathbb{D}_Gi_+^G\cO_\bY\simeq i_+^G\cO_\bY$.
\end{thm}
\begin{proof} Let $d=\dim\bX$. As in the proof of \ref{prop_resG1}
, we work over the basis $\cB$ for $\bX$ described in \cite[Def. 4.4.1]{EqDCapTwo}. Let $\bU\in\cB$ and let $H$ be $\bU$-good. 
We first construct an isomorphism
$$ \Ext^d_{\wideparen{\mathcal{D}}(\bU,H)}(i_+^G\cO_\bY(\bU), \wideparen{\mathcal{D}}(\bU,H))\simeq 
\Ext^d_{\wideparen{\mathcal{D}}(\bU)}(i_+^G\cO_\bY(\bU), \wideparen{\mathcal{D}}(\bU))$$
in ${\rm Frech}(H-\wideparen{\mathcal{D}}(\bU))$, where the $H$-action on the right-hand side is induced by the given $H$-action on $i_+^G\cO_\bY(\bU)$ and the $H$-action on 
$\wideparen{\mathcal{D}}(\bU)$. The choice of $\bU$ comes with a basis 
$\partial_1,...,\partial_d$ of a free $\cA$-Lie lattice $\cL$ of $\cT(\bU)$, for some affine formal model $\cA\subset\cO(\bU)$. Choose a good chain $H_n$ in $H$ for $\mathcal{L}$. 
We use the notation developed in the preceding proof: 
\begin{center}
$T:=\wideparen{\mathcal{D}}(\bU,H)\simeq \varprojlim_n T_n$ with 
$T_n:= U_n\rtimes_{H_n} H$ and $U_n=\widehat{U(\pi^n(\mathcal{L}))}_K$
\end{center}
and 
\begin{center}
$M:=(i_{+}^{G}\cO_\bY)(\bU)=\varprojlim_n M_n\in\cC_T$ 
\end{center}
with finitely generated $T_n$-modules $M_n$. Note that $U:=\varprojlim_n U_n=\wideparen{\mathcal{D}}(\bU)$. Every $M_n$ is a finitely generated $U_n$-module and according to Proposition \ref{prop_resG1}, $M=\varprojlim_n M_n$ is a Fr\'echet-Stein presentation for $M$, viewed as a coadmissible $U$-module. 
Now 
$$\Ext^d_{T_n}(M_n, T_n)\simeq \Ext^d_{U_n}(M_n, U_n)
$$
as right $(H-U_n)$-modules, cf. \ref{lem_crossed_prod2}. 
Since everything is compatible with variation in $n$, one obtains an isomorphism  
$\Ext^d_{T}(M, T)\simeq \Ext^d_{U}(M, U)$
in ${\rm Frech}(H-\w\cD(U))$, as claimed. The isomorphism 
$$\Ext^d_{\wideparen{\mathcal{D}}(\bU,H)}(i_+^G\cO_\bY(\bU), \wideparen{\mathcal{D}}(\bU,H))\simeq 
\Ext^d_{\wideparen{\mathcal{D}}(\bU)}(i_+^G\cO_\bY(\bU), \wideparen{\mathcal{D}}(\bU))$$
is compatible with the restriction maps arising from an inclusion of an affinoid subdomain $\bV\subset\bU$ in $\cB$. We obtain a bijection 
$$ {^r\mathcal{E}}_G^d(i_+^G\cO_\bY)\simeq {^r\mathcal{E}}^d(\Res^G_1 i^G_+\cO_\bY)$$
in $\cC_\bX$, which is $G$-equivariant with respect to the induced $G$-structure on the right-hand side. Applying the side-changing functor ${\mathcal Hom} (\Omega_\bX,-)$ together with \ref{prop_resG1} and \ref{prop_SelfDual}, yields a bijection
$$ \mathbb{D}_Gi_+^G\cO_\bY\simeq \mathbb{D}(\Res^G_1 i^G_+\cO_\bY)\simeq \Res^G_1 i^G_+\cO_\bY,$$
which is $G$-equivariant with respect to the induced $G$-structure on the right-hand side. 
But this means $ \mathbb{D}_Gi_+^G\cO_\bY\simeq i^G_+\cO_\bY$
in ${\rm Frech}(G-\cD_\bX)$.
\end{proof}

\subsection{The ring $\wUg{P}$}\label{subsec_rings}
We now place ourselves in the setting of \cite[\S 6.2]{EqDCap}. 
In particular, we suppose given an affine algebraic group $\G$ of finite type over $K$ and a continuous group homomorphism $G\rightarrow \G(K)$. We write $\fr{g}=\Lie(\G)$ and suppose that the center of $\fr{g}$ is trivial. 





\vskip5pt

The functor $\wUg{-}$ from \cite[6.2.11]{EqDCap} can be evaluated on any closed subgroup $H$ of $G$; in particular we have at our disposal the associative $K$-algebra 
\[ \wUg{P}\]
which is equal to $\wUg{H}\otimes_{K[H]} K[P]$ for any choice of compact open subgroup $H$ of $P$. 

\vskip5pt 

As a first basic result, we prove the following double coset decompositions. 

\vskip5pt

\begin{prop}\label{decomp} Suppose that there is an open subgroup $G_0$ in $G$ such that $G=G_0 P$. Let $P_0=G_0\cap P$. Let $H\subset G_0$ be a compact open subgroup. Then
\be \item  As $(\wUg{H},\wUg{P_0})$-bimodules, one has the decomposition
\[\wUg{G_0}=\underset{Z\in H\setminus G_0 / P_0}{\bigoplus} \wUg{H}Z\wUg{P_0}.\]
\item As $(\wUg{H},\wUg{P})$-bimodules, one has the decomposition \[\wUg{G}=\bigoplus\limits_{Z\in H\setminus G / P} \wUg{H}Z\wUg{P}.\]
\ee
\end{prop}
\begin{proof} Because of $G=G_0P$, we may choose a system of representatives $S\subseteq G_0$ for the $(H,P_0)$-double cosets in $G_0$ which is, at the same time, a system of representatives for the $(H,P)$-double cosets in $G$. Recall that

\[\wUg{G_0} = \underset{(\cL,N)\in\cJ(G_0)}{\invlim\limits}{} \hK{U(\cL)} \rtimes_N G_0\]
where $\cJ(G_0)$ denotes the set of all pairs $(\cL, N)$, where $\cL$ is an $G_0$-stable Lie lattice in $\fr{g}$ and $N$ is an open subgroup of $(G_0)_{\cL}$ which is normal in $G_0$. For each pair $(\cL, N)$, the group $N_P:=P\cap N$ is an open subgroup of $(P_0)_{\cL}$ which is normal in $P_0$. Consider a pair $(\cL, N)$ with the additional property $N\subseteq H_{\cL}$. Then for any $s \in S$,
\[ \big( \hK{U(\cL)}\rtimes_N H \big).s.\big( \hK{U(\cL)}\rtimes_{N_P} P_0 \big) =
\big( \hK{U(\cL)}\rtimes_N H \big).s.P_0 \]
and so, because of $\cup_{s\in S} HsP_0=G_0$, the natural inclusion

\begin{equation}\label{Decomposition} \sum_{s\in S}  \big( \hK{U(\cL)}\rtimes_N H \big).s.\big( \hK{U(\cL)}\rtimes_{N_P} P_0 \big)\longrightarrow  \hK{U(\cL)}\rtimes_N G_0\end{equation}
is bijective. On the other hand,
\[ \big( \hK{U(\cL)}\rtimes_N H \big).s.P_0 = \big( \hK{U(\cL)}\rtimes_N H \big)\utimes{K[H]} K[H].s.K[P_0]  \]
and
\[\hK{U(\cL)}\rtimes_N G_0= \big( \hK{U(\cL)}\rtimes_N H \big)\utimes{K[H]} K[G_0].\] Since $K[G_0]=\oplus_{s\in S} K[H].s.K[P_0]$, the sum on the left-hand side of formula (\ref{Decomposition}) is direct. Taking inverse limits over the cofinal set of all pairs $(\cL, N)$ satisfying additionally $N\subseteq H_{\cL}$ yields the decomposition
\[ \bigoplus\limits_{s\in S} \wUg{H}.s.\wUg{P_0}=\wUg{G_0}.\]
This proves (a). Working inside $\wUg{G}$ 
we obtain from this
\[   \sum_{s\in S} \wUg{H}.s.\wUg{P}= \sum_{s\in S} \wUg{H}.s.\wUg{P_0}.P=\wUg{G_0}.P=\wUg{G}.\] By a similar argument as above, this sum is direct. This gives (b).
\end{proof}

\begin{prop} $\wUg{H}$ is a Fr\'echet-Stein algebra for any compact open subgroup $H$ of $P$.
\end{prop}
\begin{proof} This follows from \cite[6.2.9]{EqDCap}.
\end{proof}

\subsection{Localization} 

We now place ourselves in the setting of \cite[Theorem 7.4.8]{ArdWad2023b}. 
In particular, we keep all the hypothesis of the preceding subsection and suppose the following additional hypothesis. Let $\G_0$ be a connected, split semisimple affine algebraic group scheme over $o_K$ and let $\G=\G_0\otimes K$ be its generic fibre. 
We fix a closed smooth Borel $o_K$-subgroup scheme $\B_0$ of $\G_0$ and we set 
$\B=\B_0\otimes K$. Let $\X_0=\G_0/\B_0, \X=\G/\B$ and $\bX=\X^{\rm an}$.

\begin{prop}\label{prop_comp} The algebra $\wUg{P}$ acts on $\bX$ compatibly with $P$.
\end{prop}
\begin{proof} This follows from \cite[6.4.4]{EqDCap}.
\end{proof}
According to these results, we have the category $\cC_{\wUg{P}}$ of coadmissible $\wUg{P}$-modules and the localization functor \cite[3.6.8]{EqDCap}
$$\Loc^{\wUg{P}}_\bX: \cC_{\wUg{P}} \rightarrow \cC_{\bX/P}.$$ Let $\cC_{\wUg{P},0}$ be the full subcategory of $\cC_{\wUg{P}}$ consisting of modules $M$ satisfying $\fr{m}_0 M=0$ for the maximal ideal $\fr{m}_0=Z(\fr{g})\cap U(\fr{g})\fr{g}$ of the center $Z(\fr{g})$ of $U(\fr{g})$. 

\begin{thm}\label{thm-localizationP} The functor $\Loc^{\wUg{P}}_\bX$ induces an equivalence of categories $$\cC_{\wUg{P},0}\congs \cC_{\bX/P}.$$

A quasi-inverse is given by the global sections functor $H^0(\bX,-)$.
\end{thm}
\begin{proof}
This follows from \cite[6.4.9]{EqDCap}.
\end{proof}

Similarly, given $s\in G$, there is the parabolic subgroup $^{s}P=sPs^{-1}$ and the category  $\cC_{\wUg{^{s}P}}$. In analogy to \cite[2.2.4]{EqDCapTwo} there is a twisting functor
$$\cC_{\wUg{P}} \rightarrow \cC_{\wUg{^{s}P}}, \quad M\mapsto [s]M$$ where $[s]M=\{[s]m: m\in M\}$ equals $M$ as abelian group and receives an $\wUg{^{s}P}$-action via the ring isomorphism ${\w s}^{-1}: \wUg{^{s}P}\congs  \wUg{P}$ which is induced by the conjugation automorphism $g\mapsto s^{-1}g s$ of $G$.

\begin{lem}\label{twist}
Let $s\in G$ and $M\in \cC_{\wUg{P}}$ and $H$ a compact open subgroup of $G$. There is a canonical $\wUg{H}$-linear isomorphism
\[ \wUg{H}\wotimes{ \wUg{H\cap ^{s}P}} [s]M \congs \wUg{H}s\wUg{P}\wotimes{\wUg{P}} M.\]
\end{lem}
\begin{proof}
As in the proof of \cite[2.2.10]{EqDCapTwo}. The rule $ a \w\otimes [s]m\mapsto as\w\otimes m$ defines a $\wUg{H}$-linear isomorphism
\[ \varphi_s :  \wUg{H}\wotimes{ \wUg{H\cap ^{s}P}} [s]M \congs \wUg{H}s\wUg{P}\wotimes{\wUg{P}} M\]
whose inverse is given by $asb\w\otimes m \mapsto a\w\otimes[s]b m$.
\end{proof}

\begin{prop}\label{compatwist}
Let $s\in G$ and $M\in \cC_{\wUg{P}}$. The conjugation automorphism $s^{-1}$ induces a canonical isomorphism in $\cC_{\bX/^{s}P}$
\[ \Loc_{\bX}^{\wUg{^sP}}([s]M)\congs [s]s_* \Loc_{\bX}^{\wUg{P}}(M).\]
\end{prop}
\begin{proof} Let $\bU\in \bX_w(\cT)$ and let $H$ be a $\bU$-small subgroup of $G$.
Then \[\big(\Loc_{\bX}^{\wUg{^{s}P}} [s]M\big)(\bU) = \w\cD(\bU,H\cap {}^sP) \wotimes{\wUg{H\cap ^{s}P}} [s]M\]
and \[\big( [s]s_* \Loc^{\wUg{P}}(M)\big)(\bU)= \w\cD(s^{-1}\bU,H^s\cap P) \wotimes{\wUg{H^s\cap P}} M.\]
The lattes receives its $\w\cD(\bU,H\cap ^{s}P)$-module structure from the ring isomorphism
\[{\w s}^{-1}: \w\cD(\bU,H\cap ^{s}P)\congs \w\cD(s^{-1}\bU,H^s\cap P)\]
induced from $s^{-1}.$ Since the latter induces the obvious isomorphism \[\w\cD(\bU,H\cap ^{s}P) \wotimes{\wUg{H\cap ^{s}P},s^{-1}} \wUg{H^s\cap P}\congs \w\cD(s^{-1}\bU,H^s\cap P)\]
as $(\w\cD(\bU,H\cap^{s} P), \wUg{H^s\cap P})$-bimodules, it induces a canonical $\w\cD(\bU,H\cap ^{s}P)$-linear isomorphism
\[ s^{-1}: \big(\Loc_{\bX}^{\wUg{^{s}P}} [s]M\big)(\bU)\congs \big( [s]s_* \Loc^{\wUg{P}}(M)\big)(\bU).\]
This is compatible with restriction maps and establishes then the required isomorphism $\Loc^{\wUg{^{s}P}}([s]M)\cong [s]s_* \Loc^{\wUg{P}}(M)$.
\end{proof}

We finally establish a simple compatibility between the functors $\Loc^{\wUg{P}}_\bX$ and $\Loc^{\wUg{G}}_\bX$. Let $M\in\cC_{\wUg{G}}$ and $N\in \cC_{\wUg{P}}$ and let $f: N\rightarrow M$ be a continuous $\wUg{P}$-linear map.
Recall that $\Loc^{\wUg{G}}_\bX(M)$ is the unique sheaf on $\bX$ whose restriction to $\bX_w(\cT)$ equals the presheaf $\cP^{\wUg{G}}_\bX(M)$ \cite[3.5.12]{EqDCap}. Here,
\[\cP^{\wUg{G}}_\bX(M)(\bU) = \invlim \w\cD(\bU,H) \underset{\wUg{H}}{\w\otimes} M\]
for $\bU\in  \bX_w(\cT)$ where, in the inverse limit, $H$ runs over all the $\bU$-small subgroups of $G$
\cite[3.5.3]{EqDCap}. In this case, $H\cap P$ runs over a cofinal subset of all the $\bU$-small subgroups of $P$ and we similarly have
\[\cP^{\wUg{P}}_\bX(N)(\bU)=\invlim \w\cD(\bU,H\cap P) \underset{\wUg{H\cap P}}{\w\otimes} N .\] The natural map
\[ N\longrightarrow \w\cD(\bU,H) \underset{\wUg{H}}{\w\otimes} M, x \to 1\w\otimes f(x)\]
is $\wUg{H\cap P}$-linear and extends to a $\w\cD(\bU,H\cap P)$-linear map
\[\w\cD(\bU,H\cap P) \underset{\wUg{H\cap P}}{\w\otimes} N\longrightarrow  \w\cD(\bU,H) \underset{\wUg{H}}{\w\otimes} M.\] This defines a morphism of sheaves \[ \cP^{\wUg{P}}_\bX(N)\longrightarrow \cP^{\wUg{G}}_\bX(M) \]
on $\bX_w(\cT)$ which extends then to a morphism
\begin{equation}\label{Loc(f)}
\Loc(f): \Loc^{\wUg{P}}_\bX(N)\longrightarrow \Loc^{\wUg{G}}_\bX(M)
\end{equation}
in $\Frech(P - \cD_\bX)$.
\section{Irreducibility of certain induced equivariant $\cD$-modules}
Let $K$ be a non-Archimedean complete field of mixed characteristic $(0,p)$.

\subsection{Some general results from rigid analytic geometry}
Our first Lemma is presumably well-known, but we could not find a reference.
\begin{lem} \label{kerval} Let $\bU$ be an affinoid variety over $K$ and let $\bY = V(I)$ be a Zariski closed subset of $\bU$, cut out by an ideal $I$ of $\cO(U)$. Then $p \in \sP(\bU)$ lies in the closure $\overline{\bY}$ of $\bY$ in $\sP(\bU)$ if and only if $\ker(p)$ contains $I$.
\end{lem}
\begin{proof} We recall from \cite[Theorem 4]{SchVdPut} that $p$ corresponds uniquely to a valuation $\val(p) = (\fp, V)$ on $A := \cO(U)$. We write here $\ker(p) =: \fp$; it is a prime ideal of $A$, and $V$ is a certain valuation ring in the field of fractions $k_p$ of $A/\fp$. Using \cite[Lemma 2.1.11(c)]{EqDCapTwo}, we see that we have to show that $\bU \backslash \bY \notin p$ if and only if $\ker(p) \supseteq I$.

Suppose that $\bU \backslash \bY \notin p$. By definition of $\ker(p) = \fp$, we have to show that $I$ maps to zero under the restriction map $A \to \cO(\bV)$ whenever $\bV \in p$. Choose a finite set of generators $f_1,\ldots,f_r$ such that $I = Af_1 + \cdots + Af_r$ and define, for each $n \geq 0$,
\[ \bY_n := \bigcap\limits_{j=1}^r \bU(f_j/\pi^n) \qmb{and} \bZ_n := \bigcup\limits_{j=1}^r \bU((f_j/\pi^n)^{-1}).\]
Then for each $n \geq 0$, $\{\bY_n, \bZ_n\}$ is an admissible covering of $\bU$ by special subsets. Since $p$ is a filter on the admissible open subsets of $\bU$, we see that $\bY_n\in p$ or $ \bZ_n\in p$. However if $\bZ_n\in p$ then $\bU \backslash \bY\in p$ because $\bZ_n \subseteq \bU \backslash \bY$ for any $n \geq 0$. Since we're assuming that $\bU \backslash \bY \notin p$, we see that $\bY_n \in p$ for all $n \geq 0$. Now fix $i=1,\cdots, r$ and consider the norm $||f_i||_p$ of $f_i$ in the local ring $\cO_p$, \cite[p.6]{SchVdPut}. By definition, we have $||f_i||_p = \inf\limits_{\bV \in p} ||f_i||_{\bV}$, where the infimum runs over all affinoid subdomains $\bV$ of $\bU$ contained in $p$. Since $\bY_n \in p$ for all $n \geq 0$, we see that $||f_i||_p \leq ||f_i||_{\bY_n} = |\pi^n| \qmb{for all} n \geq 0$. This shows that $||f_i||_p = 0$ for all $i=1,\cdots, r$, and therefore $||f||_p = 0$ for all $f \in I$. Since $\ker(p) = \{f \in A : ||f||_p = 0\}$, we conclude that $I \subseteq \ker p$.

Conversely, suppose that $I \subseteq \ker p$. Then as we saw above, $||f||_p = 0$ for all $f \in I$. Suppose for a contradiction that $\bU \backslash \bY \in p$. By \cite[Folgerung 1.3]{Kiehl_dR}, the covering $\{\bU((f_j/\pi^n)^{-1}) : n \geq 0, j=1,\cdots, r\}$ of $\bU \backslash \bY$ is admissible. Hence, using \cite[page 4, (p4')]{SchVdPut}, we see that $p$ must contain $\bU((f_j/\pi^n)^{-1})$ for some $n \geq 0$ and some $j=1,\cdots, r$. But then $f_j$ maps to a unit in $\cO(\bU((f_j/\pi^n)^{-1}))$. Since the map $A \to \cO_p$ factors through this algebra, we see that $f_j$ maps to a unit in $\cO_p$. But then $||f_j||_p$ cannot be zero, which is the required contradiction.
\end{proof}

\begin{lem} \label{SHY} Let $H$ be a compact $p$-adic Lie group acting continuously on the affinoid variety $\bU$, and let $\bS$ and $\bT$ be two Zariski closed subsets of $\bU$ such that $\bS \subseteq H\bT$. Then for every irreducible component $\bS'$ of $\bS$ there exists $h \in H$ such that $\bS' \subseteq h \bT$, and hence $\dim \bS \leq \dim \bT$.
\end{lem}
\begin{proof} Without loss of generality we may assume that $\bS$ is already irreducible. 
Hence by \cite[Lemma 2.5.13(b)]{EqDCapTwo},
 $\cO(\bS) = \cO(\bU) / \fp$ is an integral domain, and the ideal $\fp$ of $\cO(\bU)$ of functions vanishing on $\bS$ is prime. Pick any valuation $(\fp, V)$ on the field of fractions of $\cO(\bS)$ in the sense of \cite[page 4]{SchVdPut} and let $p \in \sP(\bU)$ be the corresponding prime filter given by \cite[Theorem 4]{SchVdPut}. Since $\bS$ is cut out in $\bU$ by $\fp = \ker(p)$ by construction, Lemma \ref{kerval} tells us that $p \in \overline{\bS}$. 

Since $\bS \subseteq H\bT$ by assumption and since $H$ is compact, we can now use \cite[Corollary 2.1.16 and Lemma 2.1.11(c)]{EqDCapTwo} to see that $p \in \overline{\bS} \subseteq \overline{H\bT} = H \overline{\bT}$. Hence we can find $h \in H$ such that $p \in h \overline{\bT} = \overline{h \bT}$. Using Lemma \ref{kerval} again, we conclude that $\ker(p) \supseteq h \cdot J$ where $J$ is the ideal of functions vanishing on $\bT$. Hence $\bS = V(\ker(p)) \subseteq V(h\cdot J) = h V(J) = h \bT$, and therefore $\dim \bS \leq \dim h \bT = \dim \bT$ as required.
\end{proof}
\begin{lem}\label{Lem_connected} Let $f: \bX\rightarrow \bY$ be a surjective morphism of rigid $K$-analytic spaces. Assume that $\bY$ admits an admissible open covering by connected affinoids $\bY_i, i\in I$ such that $f^{-1}(\bY_i)$ is connected for all $i$. If $\bY$ is connected, then $\bX$ is connected.
\end{lem}
\begin{proof} Assume that $\bY$ is connected. Then any $\bY_{i}$ can be linked up to any $\bY_{i'}$ by some finite chain of $\bY_j's$, i.e. there are $\bY_{i_1}=\bY_i,\bY_{i_2},...,\bY_{i_n}=\bY_{i'}$ with 
$\bY_{i_k}\cap\bY_{i_{k+1}}\neq\emptyset$. Let now $x,x'\in \bX$. Put $y=f(x),y'=f(x')$
and choose $\bY_{i}$ containing $y$ and $\bY_{i'}$ containing $y'$. Then pick a chain 
$\bY_{i_1},...,\bY_{i_n}$ as above. Let $\bX_{k}=f^{-1}(\bY_{i_k})$, so that $x\in\bX_1$ and $x'\in\bX_n$.
Choose a sequence of points $z_0=x, z_1,z_2,...,z_{n-1},z_n=x'$ in $\bX$ such that $z_k\in \bX_{k}\cap\bX_{{k+1}}$ for $1\leq k\leq n-1$. Since each $\bX_{k}$ is connected, one may link up $z_{k-1}$ and $z_{k}$ by a sequence of connected affinoid opens in $\bX_k$. Varying $k$, this  produces a link between $x$ and $x'$ by connected affinoid opens in $\bX$. Hence $\bX$ is connected.
\end{proof}

\begin{prop} \label{prop3} Let $H$ be a compact $p$-adic Lie group acting continuously on the affinoid variety $\bU$, and let $\bZ$ be a Zariski closed subset of $\bU$. Then
\[ \bV := \bU \hspace{0.2cm} \backslash \hspace{0.2cm} H \cdot \bZ\]
is an admissible open subspace of $\bU$: we can find a countable increasing admissible covering $\{\bV_n : n \geq 0\}$ of $\bV$ by $H$-stable affinoid subdomains of $\bU$.\end{prop}
\begin{proof} Let $A:=\cO(\bU)$ and let $I\subseteq A$ be an ideal such that $\bZ=V(I)$. 
Choose a finite set of generators $f_1,\ldots,f_r$ such that $I=Af_1+\cdots+Af_r$. 
For any $n\geq 0$ and $j = 1,\ldots,r$, let $$\bZ_n^{(j)}:=\{x\in\bU: | f_j(x) | < |\pi^n|\}, \quad \bZ_n:=\bigcup\limits_{j=1}^r \bZ^{(j)}_n \qmb{and} \bV_n:=\bU\hspace{0.2cm}\backslash \hspace{0.2cm}H\cdot \bZ_n.$$
We claim that the $H$-stable subset $\bV_n$ of $\bU$ is an affinoid subdomain of $\bU$. To see this, note that complement of $\bZ_n^{(j)}$ in $\bU$ is the Laurent domain $\bU((f_j/\pi^n)^{-1})$, defined by $| f_j/\pi^n | \geq 1$. Its stabilizer $H_{j,n}$ in $H$ (which coincides with the stabiliser in $H$ of 
$\bZ_n^{(j)}$) is therefore an open subgroup of $H$, by continuity of the $H$-action on $\bU$.
Let $$H_n:=\bigcap_{j=1}^r \;H_{j,n},$$ an open subgroup of $H$, stabilizing each $\bZ^j_n$ and therefore also $\bZ_n$. Since $H$ is compact, we can find finitely many elements $h_{n,1},\ldots,h_{n,m} \in H$ 
such that 
$$H= h_{n,1}H_n\cup \cdot\cdot\cdot \cup h_{n,m}H_n.$$ It follows that
$$H\cdot \bZ_n = \bigcup\limits_{j=1}^r \;H\cdot \bZ_n^{(j)}= \bigcup\limits_{j=1}^r\bigcup\limits_{i=1}^m h_{n,i}H_n\cdot \bZ_n^{(j)}=\bigcup\limits_{j=1}^r\bigcup\limits_{i=1}^m h_{n,i}\bZ_n^{(j)}.$$ Passing to complements gives 
$$\bV_n= \bU\hspace{0.1cm} \backslash \hspace{0.1cm}H\cdot \bZ_n = \bigcap\limits_{j=1}^r\bigcap\limits_{i=1}^m h_{n,i}\bU((f_j/\pi^n)^{-1}).$$
Since
$h_{n,i}\bU((f_j/\pi^n)^{-1})$ equals the Laurent domain 
$\{ x\in\bU: |f_j( h_{n,i}^{-1}x) |\geq | \pi^n | \}$ and since Laurent domains are stable under finite intersections \cite[Prop. 1.6.14]{BoschFRG}, we recognize $\bV_n$ as a Laurent domain in $\bU$.

We claim further that given a morphism of affinoids $f:\bW\rightarrow \bU$ with $f(\bW)\subset \bV$, there exists $n\geq 0$ such that $f(\bW)\subseteq \bV_n$. Indeed, for any fixed $h\in H$, and any $n\geq 0$,
the same argument as above shows that $\bU \backslash h\bZ_n$ is a Laurent domain in $\bU$. Moreover, 
$$\bU\backslash  h\bZ=\bigcup\limits_{n=0}^\infty \;\bU \backslash h\bZ_n$$ 
is an increasing admissible open covering of the Zariski open subspace $\bU \backslash h\bZ$ of $\bU$, cf.
\cite[Folgerung 1.3]{Kiehl_dR}. By admissibility \cite[Def. 1.10.4]{BoschFRG}, for each $h\in H$ there is $n=n(h)$ such that 
$$ f(\bW) \subseteq \;\bU \backslash h\bZ_{n(h)}.$$
Now $hH_{n(h)}$ is an open subset of $H$ containing $h$. By the compactness of $H$, there are therefore $h_1,\ldots,h_m \in H$ such that 
$$ H=h_1H_{n(h_1)}\cup \cdot\cdot\cdot \cup h_mH_{n(h_m)}.$$ 
We set $n:=\max\limits_{i=1}^m n(h_i)$. Now let $h\in H$ and choose $i$ such that $h\in h_i H_{n(h_i)}$. Then 
$$ h\bZ_n \subseteq h\bZ_{n(h_i)}\subseteq h_iH_{n(h_i)}\cdot \bZ_{n(h_i)}=h_i \bZ_{n(h_i)}$$
which means that 
$$ H\cdot \bZ_n \subseteq h_1\bZ_{n(h_1)}\cup \cdot\cdot\cdot \cup h_m\bZ_{n(h_m)}.$$ 
Passing to complements yields 
$$f(\bW)\subseteq  \bU \backslash \big( h_1\bZ_{n(h_1)}\cup \cdot\cdot\cdot \cup h_m\bZ_{n(h_m)}\big) \subseteq \bU\hspace{0.1cm} \backslash \hspace{0.1cm}H\cdot \bZ_n = \bV_n,$$
as claimed. 
We conclude \cite[Def. 1.10.4]{BoschFRG} that the covering $\{\bV_n : n \geq 0\}$ of $\bV$ is admissible and that $\bV$ is an admissible open in $\bU$. 
\end{proof}

\subsection{Local cohomology of the induced module}\label{LocCohIndMod}
We keep all notation from the preceding subsection. 
We assume here that $(\bU,H)$ is small and $H$ is a uniform pro-$p$ group, with a closed, isolated subgroup $J$ of $H$. Because $(\bU,H)$ is small, we can find an affine $H$-stable formal model $\cA$ in $\cO(\bU)$, and an $H$-stable free $\cA$-Lie lattice $\cL$ in $\cT(\bU)$. We fix a good chain $(H_\bullet)$ for $\cL$ in the sense of \cite[Definition 3.3.3]{EqDCap}.  By shrinking it further, we may assume that there is an increasing sequence of integers $(e_m)$ such that
\[ H_m = H^{p^{e_m}} \qmb{for all} m \geq 0.\]
For each $m \geq 0$, we define $J_m := J \cap H_m$ and introduce the $K$-Banach algebras 
\[ D_m := \hK{U(\pi^m \cL)} \rtimes_{J_m} J\qmb{and} R_m := \hK{U(\pi^m \cL)} \rtimes_{H_m} H.\]
Using \cite[Lemma 3.3.4]{EqDCap} we then have
\[\w\cD(\bU,J) = \invlim D_m \qmb{and}\w\cD(\bU,H) = \invlim R_m.\]
There is a natural map of $K$-Banach algebras $D_m \to R_m$. Because the group $J/J_m$ is canonically isomorphic to $JH_m / H_m$, we will identify $D_m$ with its image in $R_m$, which is equal to the sub-crossed product
\[ D_m = \hK{U(\pi^m \cL)} \rtimes_{H_m} JH_m \hookrightarrow R_m = \hK{U(\pi^m \cL)} \rtimes_{H_m} H.\]
Because $H$ is a uniform pro-$p$ group and $J$ is a closed isolated subgroup, we can find a minimal topological generating set $\{g_1,\cdots,g_d\}$ for $H$ such that $\{g_{r+1},\cdots,g_d\}$ is a minimal topological generating set for $J$. Recall \cite[\S 2.2, equation (3)]{EqDCap} that $\gamma(g)$ denotes the image of any $g \in H$ in any crossed product that we consider here, such as $D_m$ or $R_m$. We define $b_i := \gamma(g_i)-1 \in R_m$ if $1 \leq i \leq d$; then for each $\alpha \in \mathbb{N}^d$ we have the element
\[ \mathbf{b}^\alpha := b_1^{\alpha_1} b_2^{\alpha_2} \cdots b_d^{\alpha_d} \in R_m.\]
Let $\mathbb{N}_m := \{n \in \mathbb{N} : n < p^{e_m}\}$.
\begin{lem} \label{RmDm} Let $m \geq 0$. Then $\{\mathbf{b}^\alpha : \alpha \in \mathbb{N}_m^r\}$ is a basis for $R_m$ as a right $D_m$-module.
\end{lem}
\begin{proof} Note that $\{g_1^{k_1}\cdots g_r^{k_r} : k_1,\cdots,k_r \in \mathbb{N}_m\} \subset H$ maps to a complete set of coset representatives for $JH_m/H_m$ in $H/H_m$. Use \cite[Lemma 7.8]{DDMS}.
\end{proof}
We fix $\cN$ in $\cC_{\bU/J}$ so that $\cN(\bU)$ is a coadmissible $\w\cD(\bU,J)$-module. Writing
\[N_m := D_m \utimes{\w\cD(\bU,J)} \cN(\bU) \qmb{for all} m \geq 0,\]
we then see that 
\[\cN(\bU) \cong \invlim N_m. \]
If $\cM = \ind_J^H \cN$ denotes the induced module, then by \cite[Lemma 2.3.6]{EqDCapTwo}, we have
\[\cM(\bU) = \w\cD(\bU,H) \wotimes{\w\cD(\bU,J)} \cN(\bU)  = \invlim R_m \utimes{D_m} N_m.\]
We now introduce a Zariski closed subset $\bZ$ of $\bU$ and define
\[ \bV := \bU \hspace{0.2cm} \backslash \hspace{0.2cm} H \cdot \bZ.\]
We fix the countable increasing admissible covering $\{\bV_n : n \geq 0\}$ given by Proposition \ref{prop3} until the end of $\S \ref{LocCohIndMod}$. For each $n \geq 0$, using \cite[Lemma 7.6(b)]{DCapOne}, we can choose a non-negative integer $k_n$ such that $\bV_n$ is $\pi^{k_n} \cL$-admissible. We may also assume that the $k_n$'s form an increasing sequence. Using \cite[Definition 4.3.8]{EqDCap}, we form the following $K$-Banach algebras for each $m \geq k_n$:
\[ D_{m,n} := ( \hK{\sU(\pi^m \cL)} \rtimes_{J_m} J )(\bV_n) \hookrightarrow R_{m,n} :=( \hK{\sU(\pi^m \cL)}  \rtimes_{H_m} H )(\bV_n).\]
For fixed $n$ and varying $m \geq k_n$, these give Fr\'echet-Stein presentations for $\w\cD(\bV_n, J)$ and $\w\cD(\bV_n,H)$-respectively, by \cite[Proposition 4.4.2(a)]{EqDCap}:
\[ \w\cD(\bV_n,J) \cong \lim\limits_{\stackrel{\longleftarrow}{m \geq k_n}}D_{m,n} \qmb{and} \w\cD(\bV_n,J) \cong \lim\limits_{\stackrel{\longleftarrow}{m \geq k_n}} R_{m,n}.\]
\begin{lem} \label{RRDD} $R_{m,n} \cong R_m \utimes{D_m} D_{m,n}$ as $(R_m, D_{m,n})$-bimodules if $n \geq 0$ and $m \geq k_n$. 
\end{lem}
\begin{proof} Same idea as in the proof of Lemma \ref{RmDm}.
\end{proof}

Let $N_{m,n} := D_{m,n} \utimes{D_m} N_m$ for each $m \geq k_n$; using Lemma \ref{RRDD} we then have
\begin{equation}\label{eq: NVnMVn}\cN(\bV_n) = \lim\limits_{\stackrel{\longleftarrow}{m \geq k_n}} N_{m,n}  \qmb{and} \cM(\bV_n) = \lim\limits_{\stackrel{\longleftarrow}{m \geq k_n}} R_m \utimes{D_m} N_{m,n}.\end{equation}
\begin{defn} Let $m \geq 0$.
\be \item Define $n(m) := \max\{ n \geq 0 : k_n \leq m \}$. 
\item Define $N'_m := N_{m,n(m)}$.\ee
\end{defn}
Note that for each $m \geq 0$, $k_{n(m)} \leq m \leq m+1$ implies that $n(m) \leq n(m+1)$. Hence the $N'_m$ form a projective system.

\begin{lem}\label{resNM} The restriction maps $\cN(\bU) \to \cN(\bV)$ and $\cM(\bU) \to \cM(\bV)$ fit into the following commutative diagrams:
\[ \xymatrix{ \cN(\bU) \ar[rr] \ar[d]_{\cong} && \cN(\bV) \ar[d]^{\cong} \\ \invlim N_m \ar[rr] && \invlim N'_m }\qmb{and} \xymatrix{ \cM(\bU) \ar[r] \ar[d]_{\cong} & \cM(\bV) \ar[d]^{\cong} \\ \invlim \bigoplus\limits_{\alpha \in \mathbb{N}^r_m} \mathbf{b}^\alpha \otimes N_m \ar[r] & \invlim \bigoplus\limits_{\alpha \in \mathbb{N}^r_m} \mathbf{b}^\alpha \otimes N'_m }\]
\end{lem}
\begin{proof} Since $\{\bV_n : n \geq 0\}$ is an admissible covering of $\bV$ and since $\cN$ is a sheaf, the restriction maps induce an isomorphism
\[ \cN(\bV) \tocong \lim\limits_{\stackrel{\longleftarrow}{n \geq 0}} \cN(\bV_n).\]
Applying the first formula in (\ref{eq: NVnMVn}) and swapping the order of limits, we have
\[\cN(\bV) \tocong \lim\limits_{\stackrel{\longleftarrow}{n \geq 0}} \cN(\bV_n) = \lim\limits_{\stackrel{\longleftarrow}{n \geq 0}} \lim\limits_{\stackrel{\longleftarrow}{m \geq k_n}} N_{m,n} = \lim\limits_{\stackrel{\longleftarrow}{m \geq 0}} \lim\limits_{\stackrel{\longleftarrow}{n : k_n \leq m}} N_{m,n} = \lim\limits_{\stackrel{\longleftarrow}{m \geq 0}} N'_m. \]
Similarly, using the second formula in (\ref{eq: NVnMVn}) together with Lemma \ref{RmDm}, we have
\[\cM(\bV) \quad \cong \quad \lim\limits_{\stackrel{\longleftarrow}{m \geq 0}} R_m \utimes{D_m} N'_m = \lim\limits_{\stackrel{\longleftarrow}{m \geq 0}} \bigoplus\limits_{\alpha \in \mathbb{N}^r_m} \mathbf{b}^\alpha \otimes N'_m.\]
The result follows. \end{proof}
\begin{thm} \label{LCIndM} Let $H$ be a uniform pro-$p$ group with closed subgroup $J$. Suppose that $(\bU,H)$ is small. Let $\bZ$ be a Zariski closed subset of $\bU$ and let $\bV = \bU \backslash H \bZ$.  Suppose that $\cN \in \cC_{\bU/J}$ is such that the restriction map $\cN(\bU) \to \cN(\bV)$ is injective, and let $\cM := \ind_J^H \cN$. Then the restriction map $\cM(\bU) \to \cM(\bV)$ is injective as well.
\end{thm}
\begin{proof}  We first deal with the special case where the closed subgroup $J$ is \emph{isolated}, and we use Lemma \ref{resNM} and its notation. So let $\cK_m$ be the kernel of the map $N_m\rightarrow N'_m$ for any $m\geq 0$. The first diagram in the lemma, together with left-exactness of the projective limit, implies $\invlim \cK_m=0$, since the restriction map $\cN(\bU) \to \cN(\bV)$ is injective by hypothesis. 
Define \[\cW_{m}:=\bigoplus\limits_{\alpha \in \mathbb{N}^r_m} \mathbf{b}^\alpha \otimes \cK_m.\] By the second diagram in the lemma, the projective limit 
$\invlim \cW_m$ computes the kernel of the restriction map $\cM(\bU) \to \cM(\bV)$. Hence, it suffices to show that this latter projective limit vanishes. 

To start with, let $f_m$ be the transition map $\cK_{m+1}\rightarrow \cK_m$. For each $\alpha\in\N^r_m\subset \N^r_{m+1}$, the transition map $\cW_{m+1}\rightarrow \cW_m$ sends the direct summand ${\bf b}^{\alpha}\otimes \cK_{m+1}$ to ${\bf b}^{\alpha}\otimes \cK_m$ via the map ${\bf b}^{\alpha}\otimes f_m$.
Now let $(v(m))_{m\geq 0}\in \invlim \cW_m$. Write 
$$v(m)=\sum_{\alpha\in\N^r_m} {\bf b}^{\alpha}\otimes v(m)_\alpha$$
with some $v(m)_\alpha\in\cK_m$.
Now fix $m \geq 0$ and take $\alpha\in \N^r_m$. By the observation above, $$ f_{m+1}(v(m+1)_\alpha)=v(m)_\alpha$$ and, more generally, 

$$ f_{m+k}(v(m+k)_\alpha)=v(m+k-1)_\alpha$$
for any $k\geq 1$. So 
$$ ( v(m)_\alpha, v(m+1)_\alpha, v(m+2)_\alpha,\ldots) \in \lim\limits_{\stackrel{\longleftarrow}{s \geq m}} \cK_s=0.$$
In particular, we have $v(m)_\alpha=0$. This holds for any $\alpha\in \N^r_m$, whence
$v(m)=0$. This shows $(v(m))_{m\geq 0}=0$ and $\invlim\cW_m=0$.

Returning to the general case, we let $\widetilde{J} := \{g \in H : g^{p^n} \in J$ for some $n \geq 0\}$. Using \cite[Proposition 7.15(i)]{DDMS}, we see that $\widetilde{J}$ is a closed isolated subgroup of $H$ containing $J$ as an open subgroup. Let $\widetilde{\cN} := \ind_J^{\widetilde{J}} \cN \in \cC_{\bU/\widetilde{J}}$; then using \cite[Theorem B(c)]{EqDCap}, \cite[Lemma 2.3.6]{EqDCapTwo} and \cite[Corollary 7.4]{DCapOne}, we see that $\cM$ is isomorphic to $\ind_{\widetilde{J}}^H \widetilde{\cN}$ in $\cC_{\bU/H}$. Since $\widetilde{J}$ is isolated in $H$ by construction, the special case handled above shows that it is enough to show that the restriction map $\widetilde{\cN}(\bU) \to \widetilde{\cN}(\bV)$ is injective. Since $J$ has \emph{finite index} in $\widetilde{J}$, the left-handed version of \cite[Proposition 3.4.10(a)]{EqDCap} shows that the natural map
\[ K[\widetilde{J}] \underset{K[J]}{\otimes}{} \w\cD(\bU,J) \to \w\cD(\bU,\widetilde{J})\]
is an isomorphism. It follows that the maps
\[ K[\widetilde{J}] \underset{K[J]}{\otimes}{} \cN(\bU) \to \widetilde{\cN}(\bU) \qmb{and} K[\widetilde{J}] \underset{K[J]}{\otimes}{} \cN(\bV) \to \widetilde{\cN}(\bV)\]
are isomorphisms as well. Since $K[\widetilde{J}]$ is a free right $K[J]$-module, the injectivity of $\widetilde{\cN}(\bU) \to \widetilde{\cN}(\bV)$ now follows easily from the given injectivity of $\cN(\bU) \to \cN(\bV)$.
\end{proof}

\subsection{The proof of irreducibility}\label{proofirred}
We will work in the following axiomatic setting.
\begin{setup} \label{AbsHyps}\,
\begin{itemize} 
\item $\bX$ is a connected, smooth, rigid $K$-analytic variety,
\item $G$ is a {\it compact} $p$-adic Lie group acting continuously on $\bX$,
\item $\bY$ is a connected Zariski closed subset of $\bX$,
\item $\bZ \subset \bY$ is a Zariski closed subset of $\bY$ with $\dim \bZ < \dim \bY$,
\item $\cN \in \cC_{\bX/P}^{\bY}$ where $P := G_{\bY}$ is the stabilizer of $\bY$ in $G$,
\item $\cM = \ind_P^G \cN \in \cC_{\bX/G}$.
\end{itemize}
\end{setup}

\begin{defn} \label{DefSigma} We define $\Sigma := \bX \backslash G\bZ$.
\end{defn}

\begin{lem} \label{SigmaAdm} $\Sigma$ is an admissible open subset of $\bX$.
\end{lem}
\begin{proof} Let $\{\bX_i\}_{i\in I}$ be an admissible affinoid covering of $\bX$. For the admissibility of $\Sigma$, it suffices, according to property $(G_1)$ of the $G$-topology on $\bX$, (in the terminology of \cite[Definition 9.3.1/4(i)]{BGR}) to show that each 
$\bX_i\cap\Sigma$ is admissible open in $\bX$. 

Let $H_i$ be an $\bX_i$-small open subgroup of $G_{\bX_i}$. According to Lemma \ref{lemma1} below, one has $\bX_i\cap\Sigma= \bX_i \backslash H_i\bZ_{\bX_i,H_i}$ for some Zariski-closed $\bZ_{\bX_i,H_i}$ in $\bX_i$. 
By Proposition \ref{prop3},  $\bX_i \backslash H_i\bZ_{\bX_i,H_i}$ is admissible open in $\bX_i$, hence in $\bX$.
\end{proof}

\begin{lem} \label{lemma1} Let $\bU$ be an affinoid subdomain of $\bX$ and let $H$ be an open subgroup of $G_{\bU}$. Then there exists a Zariski closed subset $\bZ_{\bU,H}$ of $\bU$ such that $\bU \cap G\bZ = H \bZ_{\bU,H}$ and $\dim \bZ_{\bU,H} \leq \dim \bZ$, and therefore $\bU \cap \Sigma = \bU \backslash H \bZ_{H,\bU}$.
\end{lem}
\begin{proof} Let $s_1,...,s_n$ be the representatives for the $(H,G_\bZ)$-double cosets in $G$.
Then $G=\coprod\limits_{i=1}^n Hs_i G_\bZ$, so $G\bZ=\bigcup\limits_{i=1}^n Hs_i \bZ$. Since $\bU$ is $H$-stable, one finds
\[ \bU\cap G\bZ=\bigcup\limits_{i=1}^n \bU\cap Hs_i\bZ=\bigcup\limits_{i=1}^n H(\bU\cap s_i\bZ)=H\bZ_{\bU,H} ,\]
where $\bZ_{\bU,H}:=\bigcup\limits_{i=1}^n \bU\cap s_i\bZ$ is Zariski-closed in $\bU$ with  $\dim \bZ_{\bU,H} \leq \dim \bZ$. 
\end{proof}

Next, we introduce the following conditions on our data $(\bX,G,\bY,\bZ,\cN)$:
\begin{setup}\label{axioms}\,
\begin{enumerate}[{(}A{)}]
\item $(\bX, \bY, G)$ satisfies the LSC from \cite[Definition 2.5.6]{EqDCapTwo},
\item $\bigcup\limits_{\stackrel{g,h \in G}{g\bY \neq h\bY}} g\bY \cap h\bY \subseteq G \bZ$,
\item $\bY \cap \Sigma$ is connected,
\item $\cN$ is {\it locally simple}, i.e. $\cN_{|\bU}$ is a simple object in $\cC_{\bU/P_{\bU}}$whenever $\bU \in \bX_w(\cT)$ is connected and $\bU \cap \bY$ is connected and non-empty,
\item $\bY \subseteq \Supp(\tilde{\cN})$, and
\item $\cN$ is \emph{weakly holonomic} in the sense of \cite[Def. 5.9]{VuThesis}: $\cN \in \cC_{\bX/P}^{\wh}$.
\end{enumerate}
\end{setup}
Recall from \cite[Definition 3.4.6(a)]{EqDCap} that $\bX_w(\cT)$ denotes the set of affinoid subdomains $\bU$ of $\bX$ such that $\cT(\bU)$ admits a free $\cA$-Lie lattice for some affine formal model $\cA$ in $\cO(\bU)$. 

Our goal will be to prove Theorem \ref{MainResult} below. We assume, until the end of $\S \ref{proofirred}$, that $(\bX,G,\bY,\bZ)$  satisfy Conditions (A),(B) and (C), and that $\cN$ satisfies Conditions (D,E,F).
\begin{lem} \label{regGorbit} The $G$-orbit of $\Sigma \cap \bY$ in $\Sigma$ is regular.
\end{lem}
\begin{proof} According to \cite[Definition 1.2.2]{EqDCapTwo}, we have to show that any two distinct $G$-translates of $\Sigma \cap \bY$ in $\Sigma$ have empty intersection. Because $\Sigma = \bX \backslash G \bZ$ is $G$-stable, we have $g(\Sigma \cap \bY) = \Sigma \cap g \bY$ for any $g \in G$. Suppose that $g(\Sigma \cap \bY) \neq h(\Sigma \cap \bY)$ for some $g,h \in G$. Then $\Sigma \cap g\bY \neq \Sigma \cap h \bY$ and hence $g \bY \neq h \bY$. Condition (B) now implies that $g\bY \cap h \bY \subseteq G \bZ$ and hence $g(\Sigma \cap \bY) \cap h(\Sigma \cap \bY) = \Sigma \cap (g\bY \cap h \bY) \subseteq \Sigma \cap G\bZ = \emptyset$.  
\end{proof}

\begin{lem} \label{simpleNonSigma} $\cN_{|\Sigma}$ is a simple object in $\cC_{\Sigma/P}$.
\end{lem}
\begin{proof} Suppose that $\cN'$ is a subobject of $\cN_{|\Sigma}$ in $\cC_{\Sigma/P}$. Fix a $\Sigma_w(\cT)$-covering $\cU$ of $\Sigma$ consisting of connected affinoid subdomains. By applying \cite[Lemma 2.5.16]{EqDCapTwo}, we may refine the covering to assume that each $\bU \cap \bY$ is either connected or empty for $\bU\in\cU$. Now we define
\[ \cU_1 := \{\bU \in \cU : \cN'_{|\bU} = 0\} \qmb{and} \cU_2 := \{\bU \in \cU : \cN'_{|\bU} = \cN_{|\bU}\}.\]
We restrict the admissible covering $\cU$ to $\bY \cap \Sigma$. This gives us an admissible covering $\cU' := \{\bU \cap \bY : \bU \in \cU\}$ of $\bY \cap \Sigma$ (possibly containing the empty set), together with its subsets $\cU'_1$ and $\cU'_2$ which are defined analogously. We will now show that 
\[ \cU' = \cU'_1 \cup \cU'_2, \qmb{and} \bigcup \cU'_1 \cap \bigcup \cU'_2 = \emptyset.\]

Suppose that $\bU \in \cU \backslash \cU_1$. Then $0 \neq \cN'_{|\bU} \leq \cN_{|\bU}$ implies that $\bU \cap \bY \neq \emptyset$ because $\cN \in \cC^{\bY}_{\bX/P}$ by assumption. Hence $\bU \cap \bY$ is connected by the first paragraph of the proof. By Lemma \ref{SigmaAdm}, $\bU$ lies in $\Sigma_w(\cT) \subset \bX_w(\cT)$. By condition (D), $\cN_{|\bU}$ is a simple object in $\cC_{\bU/P_{\bU}}$, so $\cN'_{|\bU} = \cN_{|\bU}$ and hence $\bU \in \cU_2$. We have shown that $\cU = \cU_1 \cup \cU_2$, and hence \emph{a fortiori}, $\cU' = \cU'_1 \cup \cU'_2$.

Now take $\bU_1 \in \cU_1$ and $\bU_2 \in \cU_2$ and suppose for a contradiction that $\bU_1 \cap \bU_2 \cap \bY \neq \emptyset$. By applying \cite[Lemma 2.5.16]{EqDCapTwo} again, we can choose a non-empty connected affinoid subdomain $\bU_3$ of $\bU_1 \cap \bU_2$ such that $\bU_3 \cap \bY$ is also connected and non-empty.  Then $\cN_{|\bU_3}$ is simple and hence non-zero by condition (D). Hence $\cN|_{\bU_1\cap\bU_2}$ is also non-zero. However $\cN'_{|\bU_1} = 0$ because $\bU_1 \in \cU_1$ and $\cN'_{|\bU_2} = \cN_{|\bU_2}$ because $\bU_2 \in \cU_2$, so 
\[ 0 = (\cN'_{|\bU_1})|_{\bU_1 \cap \bU_2} = \cN'|_{\bU_1 \cap \bU_2} = (\cN'_{|\bU_2})|_{\bU_1 \cap \bU_2} = (\cN_{|\bU_2})|_{\bU_1 \cap \bU_2} = \cN|_{\bU_1 \cap \bU_2} \]
and we have a contradiction. Hence $\bigcup \cU'_1 \cap \bigcup \cU'_2 = \emptyset$ as required. Note that this also implies that $\cU'_1 \cap \cU'_2$ is either empty, or is equal to $\{\emptyset\}$. 

Now $\cU' \backslash \{\emptyset\}$ is still an admissible covering of $\bY \cap \Sigma$, and $\cU' \backslash \{\emptyset\}$ is the disjoint union of $\cU'_1 \backslash \{\emptyset\}$ and $\cU'_2 \backslash \{\emptyset\}$. Since $\bY \cap \Sigma$ is connected by condition (C), we deduce that $\cU'_1 = \{\emptyset\}$ or $\cU'_2 = \{\emptyset\}$ by \cite[p. 108, equation ($\ast$)]{BoschFRG}.

Suppose that $\cU'_1 = \{\emptyset\}$. This means that for every $\bU \in \cU$,  $\cN'_{|\bU} = 0$ implies that $\bU \cap \bY = \emptyset$. In other words, whenever $\bU \cap \bY \neq \emptyset$ with $\bU \in \cU$, we have $\cN'_{|\bU} \neq 0$. Condition (D) then shows that $\cN'_{|\bU} = \cN_{|\bU}$ for every such $\bU$, since then $\bU \cap \bY$ is also connected by the first paragraph of the proof. On the other hand, if $\bU \cap \bY = \emptyset$ with $\bU \in \cU$, then $\cN'_{|\bU} \leq \cN_{|\bU} = 0$ so $\cN'_{|\bU} = \cN_{|\bU}$. Hence $\cN' = \cN$, as $\cU$ is an admissible covering of $\Sigma$.

Suppose that $\cU'_2 = \{\emptyset\}$. This means that for every $\bU \in \cU$,  $\cN'_{|\bU} = \cN_{|\bU}$ implies that $\bU \cap \bY = \emptyset$. In other words, whenever $\bU \cap \bY \neq \emptyset$ with $\bU \in \cU$, we have $\cN'_{|\bU} \neq \cN_{|\bU}$. Condition (D) then shows that $\cN'_{|\bU} = 0$ for every such $\bU$, since then $\bU \cap \bY$ is also connected by the first paragraph of the proof. Since $\cN'_{|\bU} = 0$ whenever $\bU \cap \bY = \emptyset$, we see that $\cN' = 0$ as $\cU$ is an admissible covering of $\Sigma$.
\end{proof}

\begin{lem} \label{YcapSigmaDense} $\bY \cap \Sigma$ is dense in $\bY$ in the classical topology on $\bX$.
\end{lem}
\begin{proof} Let $\bU$ be an affinoid subdomain of $\bX$ containing $y \in \bY \backslash (\bY \cap \Sigma) = \bY \cap G\bZ$. It will be enough to show that $\bU \cap (\bY \cap \Sigma) \neq \emptyset$, so suppose for a contradiction that $\bU \cap \bY \subseteq \bY \cap G \bZ$. Choose an open subgroup $H$ of $G_{\bU}$. Then by Lemma \ref{lemma1}, $\bU \cap G\bZ \subseteq H \bZ_{\bU,H}$ for some Zariski closed subset $\bZ_{\bU,H}$ of $\bU$ with $\dim \bZ_{\bU,H} \leq \dim \bZ$, and hence $\bU \cap \bY \subseteq \bU \cap G\bZ \subseteq H\bZ_{\bU,H}$. Now Lemma \ref{SHY} implies that $\dim \bU \cap \bY \leq \dim \bZ_{\bU,H}$. 

By Hypothesis \ref{AbsHyps}, $\bY$ is connected, so $\dim \bY = \dim \bU \cap \bY$ and therefore $\dim \bY \leq \dim \bZ$. This contradicts Hypothesis \ref{AbsHyps}.
\end{proof}

\begin{cor} \label{StabYcapSigma} We have $G_{\bY \cap \Sigma} = P$. 
\end{cor}
\begin{proof} $P$ stabilises $\bY$ by Hypothesis \ref{AbsHyps}, so it also stabilises $\bY \cap \Sigma$ because $\Sigma$ is $G$-stable. Hence $P \leq G_{\bY \cap \Sigma}$. On the other hand, if $g \in G$ preserves $\bY \cap \Sigma$, then since $G$ acts continuously on $\bX$, $g$ must preserve the closure of $\bY \cap \Sigma$ in $\bX$ in the classical topology. Hence $g \bY \subseteq \bY$ by Lemma \ref{YcapSigmaDense} and $g \in P = G_{\bY}$.
\end{proof}

\begin{prop} \label{simpleMSigma} $\cM_{|\Sigma}$ is a simple object in $\cC_{\Sigma/G}$.
\end{prop}
\begin{proof} We will first verify the conditions of the Induction Equivalence \cite[Corollary 2.5.11]{EqDCapTwo}, applied to the action of $G$ on $\Sigma$ and to its Zariski closed subspace $\bY \cap \Sigma$.

(a) By Corollary \ref{StabYcapSigma} we have $G_{\bY \cap \Sigma} = P$, which is co-compact in $G$ by Hypothesis \ref{AbsHyps}.

(b) This is Lemma \ref{regGorbit}.

(c) By Condition (A), $(\bX, \bY, G)$ satisfies the LSC. Since $\Sigma$ is admissible open in $\bX$ by Lemma \ref{SigmaAdm} and since it is $G$-stable, $(\Sigma, \bY \cap \Sigma, G)$ satisfies the LSC by \cite[Lemma 2.5.19]{EqDCapTwo}.

Hence by \cite[Corollary 2.5.11]{EqDCapTwo}, $\ind_P^G : \cC^{\bY \cap \Sigma}_{\Sigma/P} \to \cC^{G(\bY \cap \Sigma)}_{\Sigma/G}$ is an equivalence of categories. Since $\cN_{|\Sigma}$ is a simple object in $\cC_{\Sigma/P}$ by Lemma \ref{simpleNonSigma}, we conclude that $\cM_{|\Sigma} \cong \ind_P^G (\cN_{|\Sigma})$ is a simple object in $\cC_{\Sigma/G}$ as required.\end{proof}

\begin{lem} \label{H0NalongGZ} We have $\cH^0_{G\bZ}(\cN) = 0$.
\end{lem}
\begin{proof} It is enough to show that $\cH^0_{G\bZ}(\cN)_{|\bU} = 0$ for every affinoid subdomain $\bU$ of $\bX$. By applying \cite[Lemma 2.5.16]{EqDCapTwo}, we may assume that $\bU$ is connected and $\bU \cap \bY$ is connected. If $\bU \cap \bY$ is empty, then because $\cN \in \cC_{\bX/P}^{\bY}$ by Hypothesis \ref{AbsHyps}, $\cN_{|\bU} = 0$ and there is nothing to show. So we may assume further that $\bU \cap \bY$ is non-empty.

By Lemma \ref{lemma1}, there is a Zariski closed subset $\bZ' := \bZ_{\bU,G_{\bU}}$ of $\bU$ such that $\bU \cap G \bZ = G_{\bU} \bZ'$ and $\dim \bZ' \leq \dim \bZ$.  Then 
\[\cH^0_{G\bZ}(\cN)(\bU) = H^0_{\bU \cap G\bZ}(\cN_{|\bU}) = H^0_{G_U\bZ'}(\cN_{|\bU}) = \ker\left(\cN(\bU) \to \cN(\bU \backslash G_{\bU}\bZ')\right)\]
and it will be enough to show that this is zero. By Lemma \ref{prop3}, there is an increasing admissible affinoid covering $(\bV_n)_{n=0}^\infty$ of $\bU \backslash G_{\bU}\bZ'$ such that each $\bV_n$ is $G_{\bU}$-stable. Let $\rho_n : \cN(\bU) \to \cN(\bV_n)$ denote the restriction map and suppose for a contradiction that $\ker \rho_n = \cN(\bU)$ for all $n \geq 0$. 

Fix a $\bU$-small subgroup $H$ of $G_{\bU}$. Using \cite[Theorem B(c)]{EqDCap}, the natural map
\[ \w\cD(\bV_n, H) \underset{\w\cD(\bU,H)}{\w\otimes}{} \cN(\bU) \to \cN(\bV_n) \]
is an isomorphism; then since the image of $\cN(\bU)$ in $\cN(\bV_n)$ is zero by assumption, this means that $\cN(\bV_n) = 0$ for all $n \geq 0$. In this case, $\cN_{|\bU \backslash G_{\bU} \bZ'} = 0$. Now, by Lemma \ref{SHY} applied with $H = \bG_{\bU}$, $\bS = \bU \cap \bY$ and $\bT = \bZ'$, we have $\bU \cap \bY \nsubseteq G_{\bU} \bZ'$ because as $\bY$ is connected and $\bU \cap \bY \neq \emptyset$, we have $\dim (\bU \cap \bY) = \dim \bY > \dim \bZ \geq \dim \bZ'$. Choose $y \in (\bU \cap \bY) \backslash G_{\bU}\bZ'$; then on the one hand, the stalk $\tilde{\cN}_y$ is non-zero by Condition (E), and on the other hand it must be zero because $y \in \bU \backslash G_{\bU}\bZ'$ and $\cN_{|\bU \backslash G_{\bU} \bZ'} = 0$. This contradiction shows that $\ker \rho_n \neq \cN(\bU)$ for some $n \geq 0$.

Now, since $\bV_n$ is $G_{\bU}$-stable, the restriction map $\rho_n : \cN(\bU) \to \cN(\bV_n)$ is $\w\cD(\bU, H)$-linear, and hence \emph{a fortiori} $\w\cD(\bU, H \cap P)$-linear. Since $\cN_{|\bU}$ is a simple object in $\cC_{\bU/P_{\bU}}$ by Condition (D), we deduce from \cite[Theorem B(c)]{EqDCap} that the coadmissible $\w\cD(\bU,H\cap P)$-module $\cN(\bU)$ has no non-zero, proper, closed $\w\cD(\bU,H \cap P)$-submodules that are also $P_{\bU}$-stable. Since $\ker \rho_n$ a proper, closed $\w\cD(\bU, H \cap P)$-submodule of $\cN(\bU)$ which is also $P_{\bU}$-stable, we deduce that $\ker \rho_n = 0$. Then for any $m \geq n$, we also have $\ker \rho_m \subseteq \ker \rho_n = 0$ since $\bV_n \subseteq \bV_m$. Finally, as $\bU \backslash G_U\bZ' = \bigcup\limits_{m=n}^\infty \bV_m$, using \cite[Corollary 2.1.5]{EqDCapTwo} we see that
\[\cH^0_{G\bZ}(\cN)(\bU) = \ker\left(\cN(\bU) \to \cN(\bU \backslash G_U\bZ')\right) = \bigcap\limits_{m = n}^\infty \ker(\rho_m : \cN(\bU) \to \cN(\bV_m)) = 0. \qedhere\]
\end{proof}

\begin{lem} \label{H0MalongGZ} We have $\cH^0_{G\bZ}(\cM) = 0$.
\end{lem}
\begin{proof} This is again a local statement, so we fix an affinoid subdomain $\bU$ of $\bX$ and a $\bU$-small open subgroup $H$ of $G_{\bU}$. By shrinking $H$ further, we assume that $H$ is uniform pro-$p$. By Lemma \ref{lemma1}, there is a Zariski closed subset $\bZ' := \bZ_{\bU,H}$ of $\bU$ such that $\bU \cap G \bZ = H \bZ'$ and $\dim \bZ' \leq \dim \bZ$. It will be enough to show that $\cH^0_{G\bZ}(\cM)(\bU) = H^0_{H\bZ'}(\cM_{|\bU}) = 0$. 

Recall the Mackey decomposition for $\cM_{|\bU} = (\ind_P^G \cN)_{|\bU}$ from \cite[Lemma 2.3.7]{EqDCapTwo}: choose a set $\{s_1,\cdots,s_m\}$ of representatives for the $(H,P)$-double cosets in $G$, for each $i=1,\cdots,m$ write $H_i := H \cap s_i P s_i^{-1}$ and set $\cN_i = \Res^{s_i P s_i^{-1}}_{H_i} [s_i] s_{i,\ast}\cN$ to be the restriction to $H_i$ of the $s_i$-twist $[s_i]s_{i,\ast}\cN \in \cC_{\bX/s_iPs_i^{-1}}$ of $\cN$; then there is a natural isomorphism in $\cC_{\bU/H}$
\[ \cM_{|\bU} \cong \bigoplus\limits_{i=1}^m \ind_{H_i}^H (\cN_{i|\bU}).\]
Fixing $i=1,\cdots, m$, it will therefore be enough to show that $H^0_{H\bZ'}(\ind_{H_i}^H(\cN_{i|\bU})) = 0$. Using \cite[Corollary 2.1.5]{EqDCapTwo}, we see that
\[ H^0_{H\bZ'}(\cN_{i|\bU}) = \ker\left( \cN_i(\bU) \to \cN_i(\bU \backslash H \bZ')\right) = \ker\left( \cN(s_i^{-1}\bU) \to \cN(s_i^{-1}(\bU \backslash H \bZ'))\right).\]
Since $s_i^{-1}(\bU \backslash H \bZ') = s_i^{-1}(\bU \backslash G \bZ) = s_i^{-1}\bU \backslash G \bZ$, applying \cite[Corollary 2.1.5]{EqDCapTwo} again shows that $H^0_{H\bZ'}(\cN_{i|\bU}) = \cH^0_{G\bZ}(\cN)(s_i^{-1}\bU)$. This last group is zero by Lemma \ref{H0NalongGZ}. Finally we can apply Theorem \ref{LCIndM} to see that $H^0_{H\bZ'}(\ind_{H_i}^H(\cN_{i|\bU})) = 0$ as required.
\end{proof}

\begin{cor} \label{KeyLemma} Assume that $\cN$ satisfies Conditions (D,E), and suppose that $\cM'$ is a subobject of $\cM = \ind_P^G\cN$ in $\cC_{\bX/G}$ such that $\cM'_{|\Sigma} = 0$. Then $\cM' = 0$.
\end{cor}
\begin{proof} Lemma \ref{H0MalongGZ} implies that $\cH^0_{G\bZ}(\cM') \leq \cH^0_{G\bZ}(\cM) = 0$. Hence for any affinoid subdomain $\bU$ of $\bX$, using Lemma \ref{SigmaAdm} and \cite[Corollary 2.1.5]{EqDCapTwo} we have
\[ 0 = \cH^0_{G\bZ}(\cM')(\bU) = \ker\left(\cM'(\bU) \to \cM'(\bU \backslash G\bZ)\right).\]
However $\cM'(\bU \backslash G\bZ) = \cM'(\bU \cap \Sigma) = 0$ by assumption, so $\cM'(\bU) = 0$ as required. \end{proof}

\begin{thm} \label{MainResult} Suppose that $(\bX,G,\bY,\bZ)$ satisfy Conditions (A,B,C), that Bernstein's inequality holds in $\cC_{\bX/G}$ and $\cC_{\bX/P}$ and that $\cN$ and $\mathbb{D}(\cN)$ satisfy Conditions (D,E,F). Then $\cM=\ind_P^G \cN$ is a simple object in $\cC_{\bX/G}$.
\end{thm}
Remark: If $\cN$ satisfies (E,F), then $\mathbb{D}(\cN)$ automatically satisfies (E,F),
since duality preserves the support and weak holonomicity, so the point is condition (D). 
\begin{proof}  Let $\cM'$ be a non-zero subobject of $\cM$ in $\cC_{\bX/G}$; we have to show that $\cM' = \cM$. 
Since $\cN$ is weakly holonomic by Condition (F), we know by \cite[Prop. 6.20]{VuThesis} that $\cM = \ind_P^G \cN$ is also weakly holonomic, and by Theorem \ref{thm_IndDual},
that we have a natural isomorphism
\[\mathbb{D}(\cM) \cong \ind_P^G \mathbb{D}(\cN)\]
in $\cC_{\bX/G}$, if Bernstein's inequality holds in $\cC_{\bX/G}$ and $\cC_{\bX/P}.$
Therefore by \cite[Prop. 5.11]{VuThesis} all terms in the short exact sequence
\[ 0 \to \cM' \to \cM \to \cM'' \to 0\]
are weakly holonomic as well. Applying the exact and contravariant duality functor $\mathbb{D}$ from \cite[Def. 5.14]{VuThesis} gives us another short exact sequence in $\cC_{\bX/G}$:
\[0 \to \mathbb{D}(\cM'') \to \mathbb{D}(\cM) \to \mathbb{D}(\cM') \to 0.\]
Restricting this sequence to $\Sigma$ gives the short exact sequence in $\cC_{\Sigma/G}$
\[0 \to \mathbb{D}(\cM''_{|\Sigma}) \to \mathbb{D}(\cM_{|\Sigma}) \to \mathbb{D}(\cM'_{|\Sigma}) \to 0.\]
Now $\cM_{|\Sigma}$ is simple in $\cC_{\Sigma/G}$ by Proposition \ref{simpleMSigma}; hence $\mathbb{D}(\cM_{|\Sigma})$ is also simple. Also, $\mathbb{D}(\cM'_{|\Sigma}) \neq 0$ because $\cM'_{|\Sigma} \neq 0$ by Corollary \ref{KeyLemma} and $\mathbb{D} \circ \mathbb{D} \cong 1_{\cC_{\bX/G}^{\wh}}$ by \cite[Prop. 5.15]{VuThesis}. Hence $\mathbb{D}(\cM'')_{|\Sigma}=\mathbb{D}(\cM''_{|\Sigma}) = 0$. However $\mathbb{D}(\cN)$ also satisfies Conditions (D,E) by assumption, so Corollary \ref{KeyLemma} applied to the subobject $\mathbb{D}(\cM'')$ of $\mathbb{D}(\cM) \cong \ind_P^G \mathbb{D}(\cN)$ shows that $\mathbb{D}(\cM'') = 0$. Hence $\cM'' = 0$ and $\cM = \cM'$ as required.\end{proof}

 \subsection{The set of self-intersections}\label{sec_self_inter}
 In this subsection, we give a criterion to verify the hypothesis (B) appearing in the list of conditions of the preceding subsection. 
 
 Let $\bX$ be a rigid analytic variety and $G$ a $p$-adic Lie group (possibly non-compact) acting continuously on $\bX$. Let $\bY$ a Zariski closed subset of $\bX$. Recall from $\S$\ref{sec-IntObstr} the stabilizer $G_\bY = \{ g \in G : g \bY \subseteq \bY\}$ of $\bY$ in $G$.\begin{lem}\label{lem-gYstab} Suppose that $\bX$ is quasi-compact. Then $g \bY = \bY$  for every $g \in G_\bY$.\end{lem}
\begin{proof} Suppose that $g \bY \subseteq \bY$ for some $g \in G$. Then the descending chain $\bY \supseteq g \bY \supseteq g^2 \bY \supseteq \cdots $must terminate: this is clear when $\bX$ is affinoid as $\cO(\bX)$ is then a Noetherian ring, and in general it follows from the quasi-compactness of $\bX$. Hence $g^{n+1} \bY = g^n \bY$ for sufficiently large $n$; applying $g^{-n}$ then gives $g \bY = \bY$ as claimed. \end{proof}
Let $S$ be a set of representatives for the double cosets $G_\bY \setminus G \;/ G_\bY$ containing $1\in G$ and define $S^\ast := S \backslash \{1\}$. We write
\[ \bR_v := \bY \cap v \bY \qmb{for every} v \in S, \qmb{and} \bZ :=  \bigcup\limits_{v\in S^\ast} \bR_v .\]\begin{rmk} $\bZ$ is a Zariski closed subset of $\bX$ whenever $G_\bY \setminus G \;/ \;G_\bY$ is finite.
\end{rmk}
\begin{prop}\label{prop-richardson} Suppose that $\bX$ is quasi-compact. Then
\[ \bigcup\limits_{g\in G \setminus G_\bY} \;\bY\cap g\bY=G_\bY\bZ.\]
In particular, $\bY$ has a regular $G$-orbit in $\bX$ if and only if $\bZ=\emptyset$.
\end{prop}\begin{proof} Given $g\in G\setminus G_\bY$, there is a unique $v\in S^*$ such that $g\in G_\bY v G_\bY$. Write $g = g_1 v g_2$ for some $g_1, g_2 \in G_{\bY}$. The quasi-compactness assumption on $\bX$ implies that $g_1\bY = \bY = g_2 \bY$ by Lemma \ref{lem-gYstab}. Hence for every $g \in G \setminus G_\bY$ we have
\[ \bY \cap g \bY = \bY \cap g_1 v g_2 \bY = g_1 (\bY \cap v \bY)  = g_1 \bR_v \subseteq G_\bY \bZ\]
which gives the forward inclusion. Using Lemma \ref{lem-gYstab} again, if $h \in G_\bY$ and $w \in S^\ast$ then
\[h \bR_w = h(\bY \cap w \bY) = \bY \cap h w \bY \subseteq \bigcup\limits_{g\in G \setminus G_\bY} \;\bY\cap g\bY,\]
and the reverse inclusion follows. The last sentence is clear.\end{proof}
\begin{cor}\label{cor-richardson} Suppose that $\bX$ is quasi-compact. Then
\[\bigcup\limits_{\stackrel{g,h \in G}{g\bY \neq h\bY}} g\bY \cap h\bY = G \bZ.\]
\end{cor}
\begin{proof}
For the forward inclusion, let $g,h \in G$ such that $g\bY \neq h\bY$. Then $g^{-1}h\notin G_\bY$ by Lemma \ref{lem-gYstab}. By Proposition \ref{prop-richardson}, we have
\[g\bY \cap h\bY=g( \bY\cap g^{-1}h\bY)\subseteq G. \bigcup_{g'\in G \setminus G_\bY} \;\bY\cap g'\bY=G.(G_\bY.\bZ_w)=(G.G_\bY).\bZ=G\bZ.\]
Since $\bZ$ is clearly contained in the left-hand side, which is moreover $G$-stable, the reverse inclusion quickly follows.
\end{proof}
For future applications, we single out the following observation, which is a direct consequence of Corollary \ref{cor-richardson}.
\begin{cor}\label{cor_compact_stab} Suppose that $\bX$ is quasi-compact, and that there is a compact subgroup $G_0\subseteq G$ such that 
$G\bZ=G_0\bZ$. Then
$$\bigcup\limits_{\stackrel{g,h \in G_0}{g\bY \neq h\bY}} g\bY \cap h\bY \subseteq G_0 \bZ.$$
\end{cor}
To conclude, we observe a certain stability of these constructions under inverse images with respect to equivariant surjections. To make this precise, let  $\tilde{\bX}$ be another rigid analytic space with continuous $G$-action and let
$$ f: \tilde{\bX}\longrightarrow\bX$$
be a $G$-equivariant morphism. Let $\tilde{\bY}$ be a Zariski closed subset of $\tilde{\bX}$. Denote by $\tilde{S}, \tilde{\bR}_v$ (for $v\in\tilde{S}$) and $\tilde{\bZ}$ the above notions for the pair $\tilde{\bY}\subset\tilde{\bX}$.

\begin{prop}\label{prop-mor} Let $\tilde{\bY}=f^{-1}(\bY)$. Suppose that 
$f$ is surjective. Then $G_{\tilde{\bY}}=G_\bY$. If 
$\tilde{S}=S$, then $\tilde{\bZ}=f^{-1}(\bZ)$.
If, additionally, $G\bZ=G_0\bZ$ in $\bX$ for some compact subgroup $G_0\subseteq G$, then also $G\tilde{\bZ}=G_0\tilde{\bZ}$ in $\tilde{\bX}$.
\end{prop}
\begin{proof} The equality $G_{\tilde{\bY}}=G_\bY$ follows from Lemma \ref{lem-elementary}(b).
Taking $\tilde{S}=S$, one computes for $v\in S$ that 
$$ \tilde{\bR}_v=\tilde{\bY}\cap v\tilde{\bY}=f^{-1}(\bY)\cap v f^{-1}(\bY)=f^{-1}(\bY\cap v \bY)=f^{-1}(\bR_v).$$ Since $f^{-1}$ commutes with arbitrary unions, this implies $\tilde{\bZ}=f^{-1}(\bZ)$.
If $G\bZ=G_0\bZ$, then 
\[G\tilde{\bZ}=Gf^{-1}(\bZ)=f^{-1}(G\bZ)=f^{-1}(G_0\bZ)=G_0f^{-1}(\bZ)=G_0\tilde{\bZ}.\qedhere\]
\end{proof}

\section{Irreducible equivariant $\cD$-modules for Schubert varieties}

Let $K$ be a non-Archimedean complete field of mixed characteristic $(0,p)$.
We give examples related to classical Schubert varieties, where the axiomatic approach for irreducible induced modules of the previous section applies. 

\subsection{Schubert varieties and their $G$-orbits}
Our basic reference for the following is \cite[chap. 13]{Jantzen}. In this subsection, $K$ could be any field. Let $\G$ be a split connected reductive $K$-group $\G$, with its natural $\G$-action given by conjugating the Borel subgroups of $\G$. Let $G$ be a $p$-adic Lie group with a continuous homomorphism $G\rightarrow\G(K)$. Let $\T\subseteq\B$ be a Borel subgroup in $\G$ containing a split maximal torus $\T$. Let $W$ be the Weyl group of the pair $(\G,\T)$. 
The choice of $\B$ determines a set of simple reflections $s_i$ and a corresponding length function for $W$.
The $\B$-orbits $C_w$ in the full flag variety $\G/\B$ can be indexed by the Weyl elements $w\in W$ and their Zariski closures $\X_w$ give rise to the well-known Schubert varieties. Each $\X_w$ has the structure of a normal projective $K$-variety, usually with singularities. \vskip5pt 

Let $\B\subseteq \P$ be a parabolic subgroup of $\G$. 
Let $W_{\P}\subseteq W$ be the parabolic subgroup of $W$ associated to $\P$ and let $W^{\P}\subseteq W$ be the system of minimal representatives (i.e representatives of  minimal length) for the cosets in $W/W_{\P}$. Denote by $w_{o,\P}\in W_{\P}$ the longest element in $W_{\P}$. 
The products of the form $ww_{o,\P}$ with $w\in W^{\P}$ are then the maximal representatives (i.e. representatives of maximal length) of the cosets in $W/W_{\P}$. The Schubert varieties in the partial flag variety $\G/\P$ are the Zariski closures $\X_{w\P}$  of the $\B$-orbits $\B w\P/\P$ for $w\in W^{\P}$.
There is the surjective $\G$-equivariant morphism $$f: \G/\B\rightarrow \G/\P, \;\; g\B\mapsto g\P.$$
\vskip5pt\begin{prop}\label{prop_schubert} \be
\item One has $f^{-1}(\X_{w\P})=\X_{ww_{o,\P}}$ for any $w\in W^{\P}$.
\item $\X_{ww_{o,\P}}$ is smooth if and only if $\X_{w\P}$ is smooth.
\item $\X_{ww_{o,\P}}$ has a regular $G$-orbit in $\G/\P$ if and only if $\X_{w\P}$ has a regular $G$-orbit in $\G/\P$.
\ee
\end{prop}
\begin{proof} Part (a) is \cite[13.8(2)]{Jantzen} which moreover says that the induced map $\X_{ww_{o,\P}}\rightarrow \X_{w\P}$ is a locally trivial fibration with fiber $\P/\B$. Since $\P/\B$ is smooth, this implies (b). Finally, (c) follows from (a) using Lemma \ref{lem-elementary}(c).
\end{proof}

\begin{cor}\label{cor_schubert} 

 \be 
 \item $\X_{w_{o,\P}}=\P/\B$ is smooth.
 \item The $G$-orbit of $\X_{w_{o,\P}}$ is regular in $\G/\B$.
 \item  the Schubert curves $\X_{s_i}$ have regular $G$-orbits in $\G/\B$.
 \ee
\end{cor}
\begin{proof}
The above proposition applied to $w=1$ gives (a) and (b). Part (c) is the special case where $\P$ is the minimal parabolic associated with $s_i$.
\end{proof}

\begin{rmk} We do not know whether the varieties $\P/\B$ for the standard parabolics $\B\subseteq \P$ appearing in the corollary exhaust all smooth Schubert varieties with regular $G$-orbit in $\X$. 
\end{rmk}

\vskip5pt

\begin{ex}\label{example} We discuss some cases of low dimension for the group $\G={\rm GL}_n$. Let $G={\rm GL}_n(K)$.
Identify $W$ with the symmetric group $S_n$.
Let ${\rm Gr}(d,n)$ be the Grassmannian of $d$-dimensional linear subspaces in affine $n$-space. Let $e_1,\ldots,e_n$ denote the standard basis of the latter space and denote by $\P=\P(d,n-d)\subseteq\G$ the parabolic subgroup equal to the stabilizer of the subspace spanned by $e_1,\ldots,e_d$. 
Then $W_{\P}$ identifies to the subgroup $S_d \times S_{n-d}$ of $S_n$. Moreover, ${\rm Gr}(d,n)=\G/\P$ and we have the surjective morphism $f: \X\rightarrow {\rm Gr}(d,n)$.

\vskip5pt

In the special case $d=1$ the Grassmannian  ${\rm Gr}(1,n)$ is the projective space $\P_K^{n-1}$ of dimension $n-1$. The system $W^{\P}$ identifies with the set $\{1,\ldots,n\}$ via $w\mapsto w(1)$ and the Schubert varieties of  ${\rm Gr}(1,n)$ are all smooth and form a flag of linear subspaces 
\[\X_{1\P} \subset \X_{2\P} \subset \cdot\cdot\cdot \subset \X_{n\P}\]
with $\X_{j\P} \simeq \P_K^{j-1}$ for $j\in\{1,\ldots,n\}$.

\vskip5pt

The first case of a non-regular $G$-orbit for these 
$\X_{j\P}$ appears in the case $n=3$ and $j=2$, i.e. the Schubert divisor $\X_{2\P}$ does not have a regular $G$-orbit in ${\rm Gr}(1,3)$. In fact,
$$\X_{2\P}=\{ [x_1:x_2:x_3]\in\P^2_K \; | \; x_3=0\} \subset \P^2_K $$ 
and its stabilizer $G_{\X_{2\P}}$ in ${\rm GL}_3$ equals therefore the minimal standard parabolic $\P(2,1)$ of all matrices of the form

$$
 \begin{pmatrix}
  a_{1,1} & a_{1,2} &  a_{1,3} \\
  a_{2,1} & a_{2,2} & a_{2,3} \\
 
  0 & 0 & a_{3,3} 
 \end{pmatrix}$$ 
 
in ${\rm GL}_3$. Note that our parabolic $\P=\P(1,2)$ from the beginning equals the "other" minimal standard parabolic of ${\rm GL}_3$. To compute the intersection obstruction we let 
$g\in {\rm GL}_3\setminus \P(2,1)$. Assume $g_{32}\neq 0$
and choose a point $x=[x_1:x_2:0]\in \X_{2\P}$
with $-x_2/x_1=g_{31}/g_{32}$. Then $g_{31}x_1+g_{32}x_2=0$ which means $gx\in \X_{2\P}\cap g\X_{2\P}$. In other words,  $\X_{2\P}\cap g\X_{2\P}\neq \emptyset$. The case where $g_{31}\neq 0$ works similarly. Denoting by 
$P(2,1)$ the $L$-rational points of $\P(2,1)$, we therefore have $$Z_{\X_{2\P}}/ G_{\X_{2\P}}=G/P(2,1).$$
In particular, $\X_{2\P}$ does not have a regular $G$-orbit in ${\rm Gr}(1,3)$. 

\vskip5pt 

Still in the case of ${\rm GL}_3$, let $s_i=(i,i+1)$ and $c=s_1s_2=(123)\in S_3$. The Schubert divisor $\X_c$ of $\X$ equals the inverse image $f^{-1}(\X_{2\P})$ under the map  $f: \X\rightarrow {\rm Gr}(1,3)$. In particular,  $$Z_{\X_c}/ G_{\X_c}=Z_{\X_{2\P}}/ G_{\X_{2\P}}=G/P(2,1)$$
according to \ref{lem-elementary} and hence $\X_c$ does not have a regular $G$-orbit in $\X$. 

\end{ex}

\subsection{Schubert varieties in projective space}
In this subsection, we explain how Schubert varieties in projective space gives rise to 
irreducible equivariant $\cD$-modules. We let $G={\rm GL}_n(L)$, where $L$ is a finite extension of $\Qp$ contained in $K$, and consider $\P_K^{n-1,\rm an}$ with its induced $G$-action. We consider the analytic Schubert varieties 
\[\bX_{1} \subset \bX_{2} \subset \cdot\cdot\cdot \subset \bX_{n}\]
with $\bX_{j}= (\X_{j\P})^{\rm an}$. Let $P_j={\rm Stab}_G(\bX_j)$.
 \begin{thm}\label{thm-SchubProjIrred} Fix $j$ and write $P:=P_j$. Let $i: \bX_j\hookrightarrow \P_K^{n-1,{\rm an}}$ denote the closed embedding. 
 Let $\cN:=i^{P}_{+}\cO_{\bX_j}\in \cC_{\P_K^{n-1,{\rm an}}/P}^{\bX_j}$ be the $P$-equivariant pushforward of the structure sheaf $\cO_{\bX_j}$. Then the induced module
 $\cM:=\ind_P^G \cN$ is an irreducible object in $\cC_{\P_K^{n-1,{\rm an}}/G}$.
  \end{thm}
  By the equivariant Kashiwara Theorem, \cite[Theorem B]{EqDCapTwo}, we know that $\cN_j$ is an irreducible object in $\cC_{\P_K^{n-1,{\rm an}}/P}^{\bX_j}$. If the $G$-orbit of $\bX_j$ is regular 
 in $\P_K^{n-1,{\rm an}}$, then the induction equivalence, Theorem \ref{inductionequiv}, applies directly, and shows that
 $ \cM$ is an irreducible object in $\cC_{\P_K^{n-1,{\rm an}}/G}$. This is for example the case when $j=1$ or $j=n$.  However, our above example in the case $n=3$ and $j=2$ shows that this is not always the case. Instead, we will apply Theorem \ref{MainResult}. Because of these remarks, we may and will suppose in the following that $2\leq j \leq n-1$. 
 
 We denote by $s_j=(j,j+1)\in W$ the $j$-th elementary transposition. Let $G_0={\rm GL}_n(o_L)$ and $P_0=P\cap G_0$. Theorem \ref{MainResult} will in fact give a stronger result, namely the irreducibilty of $\Res^G_{G_0}\cM$ in the category $\cC_{\P_K^{n-1,{\rm an}}/G_0}$. According to Proposition \ref{prop_iwasawa} we have $\Res^G_{G_0}\cM\simeq \ind_{P_0}^{G_0} (\Res^P_{P_0}\cN)$. 
 
 For simplicity, we will continue to write $\cN$ instead of $\Res^P_{P_0}\cN$ in the following. 
 Let $S\subset W$ be a finite set of representatives containing $1$ for the double cosets 
  $P \setminus G \;/ \; P$. Note that the transposition $s_j$ does not stabilize $\bX_j$. Thus we
may and will suppose that $s_j\in S$. Write $S^*=S\setminus\{ 1\}$. Recall from $\S$\ref{sec_self_inter} the Zariski closed subsets $\bR_v:=\bX_{j}\cap v\bX_{j}$ and 
\[\bZ_{j}:= \bigcup\limits_{v\in S^*} \bR_v\]
of $\P_K^{n-1,{\rm an}}$. We will apply Theorem \ref{MainResult}, to the data $$(\bX := \P_K^{n-1,{\rm an}},G_0,\bX_j,\bZ_j,\cN).$$
 
  \vskip5pt
 In the following, we will verify the axioms (A,B,C,D,E,F) appearing in Theorem \ref{MainResult}.
 
 \vskip5pt
 Since $\P_K^{n-1,{\rm an}}$ is separated and $\bX_j$ is irreducible and quasi-compact, 
 \cite[Corollary 2.5.18]{EqDCapTwo} implies that the triple $(\P_K^{n-1,{\rm an}},\bX_j,G_0)$ satisfies the LSC, whence axiom (A). The verification of (B) and (C) relies on the following lemma. 
\begin{lem}\label{lem-linearlydefined} \,
 \be
 \item For every $v\in S^*$ there is $w=w(v)\in W$ and $m=m(v)$ with $1\leq m < j$, such that ${^wP_m}=wP_mw^{-1}$ equals the stabilizer of $\bR_v$. 
 \item We have $G\bZ_j=G_0\bZ_j$.
 \ee
 \end{lem}
\begin{proof}  
(a) For each $r = 1,\cdots,n$ let $p_r$ be the coordinate function on $\P_K^{n-1,{\rm an}}$ vanishing on the $r$-th homogeneous coordinate. Then $\bX_j=V(p_n,p_{n-1},\ldots,p_{j+1})$ and therefore $v\bX_j=V(p_{v(n)},\ldots,p_{v(j+1)})$. Let 
 $I_v:=\{n,\ldots,j+1\} \cup \{v(n),\ldots,v(j+1) \}$, so that
 $$ \bR_v=\bX_j\cap v\bX_j=\bigcap\limits_{i\in I_v} V(p_i).$$
Since $v\in S^*$, the set $\bR_v$ is properly contained in $\bX_j$. Now $I_v$ is just some subset of $\{1,\cdots,n\}$ with $n - j <  |I_v| \leq 2(n-j)$, so we can find some $m(v) <  j$ with $|I_v| = n - m(v)$. Then we can find some $w\in W$ such that 
\[w(I_v)=\{n,\ldots,m(v)+1\}.\]
Hence $w\bR_v=\bX_{m(v)}$, and therefore $wP_{m(v)}w^{-1}$ equals the stabilizer of $\bR_v$ in $G$.
 
(b) Using (a), we see that $G_0\bR_v \subseteq G_0\bX_{j-1}$ for any $v\in S^*$, so that $G_0\bZ_j\subseteq G_0\bX_{j-1}$. But if $v=s_j$, then $I_v=\{n,\cdots,j\}$, and hence $\bR_v=\bX_{j-1}$ and  $G_0\bZ_j=G_0 \bX_{j-1}$. By the Iwasawa decomposition  $G=G_0P_{j-1}$, whence 
 $G\bZ_j=G\bX_{j-1}=G_0\bX_{j-1}=G_0\bZ_j$.
 \end{proof}
 
  Now (B) follows from Lemma \ref{lem-linearlydefined}(b) together with and Corollary \ref{cor_compact_stab}. Furthermore, 
  $$\Sigma=\P_K^{n-1,{\rm an}}\setminus G_0\bZ_j=\P_K^{n-1,{\rm an}}\setminus G\bX_{j-1}.$$
  It is rather clear that  $\bX_{j}\cap G\bX_{j-1}$ equals the set $\mathcal{H}$ of all $L$-rational hyperplanes in $\bX_j\simeq \P_K^{j-1,{\rm an}}$, whence
  $$\Sigma\cap \bX_{j}\simeq \P_K^{j-1,{\rm an}}\setminus \bigcup_{H\in \mathcal{H}} H.$$
   This Drinfeld space is well-known to be connected, whence (C). 
 
 Note that Bernstein's inequality holds in $\cC_{\P_K^{n-1,{\rm an}}/G_0}$ and $\cC_{\P_K^{n-1,{\rm an}}/P_0}$, since $\P_K^{n-1,{\rm an}}$ has good reduction, cf. \cite[Cor. 5.8]{VuThesis}.

We start the verification of the remaining axioms (D,E,F) from \ref{axioms} with the observation, cf. Theorem \ref{prop_SelfGDual}, that $\cN\simeq \mathbb{D}(\cN)$. 
Now let $\bU \in \bX_w(\cT)$ be connected with $\bU \cap \bX_j$ connected and non-empty. Let $P_\bU$ be the stabilizer of $\bU \cap \bX_j $ in $P$.
By the local nature of the equivariant pushforward, we have 
$$\cN_{|\bU}\simeq i^{P_\bU}_{\bU,+}\cO_{\bU\cap \bX_j},$$
where $i_{\bU}$ denotes the closed immersion of $\bU\cap\bX_j$ into $\bU$. 
Now $\bU \cap \bX_j$ is connected, so $\cO_{\bU\cap \bX_j}$ is a simple object in $\cC_{\bU\cap\bX_j/P_{\bU}}$ by \cite[Proposition 7.5.1(2)]{ArdWad2023b}. Hence $\cN_{|\bU}$ is a simple object in $\cC_{\bU/P_{\bU}}$ by \cite[Theorem B]{EqDCapTwo}, giving axiom (D) for $\cN$.

 It is clear that $\bX_j = \Supp(\tilde{\cN})$, whence (E).
Finally, again by the compatibilty of equivariant push-forward with duality, $\cN \in \cC_{\P_K^{n-1,{\rm an}}/P_0}$ is weakly holonomic, whence (F). This completes the proof of Theorem \ref{thm-SchubProjIrred}.
\vskip5pt

  
 \subsection{Some cases for the full flag variety}
 Let $L$ be a finite extension of $\Q_p$ contained in $K$ and let $G=\G_L(L)$ for a connected 
 reductive algebraic group $\G_L$ defined over $L$. We suppose that $\G:=\G_L \times K$ is $K$-split and adopt all the notation from 5.1 for the $K$-group $\G$. In particular, $\T\subseteq\B$ denotes a Borel subgroup in $\G$ containing a split maximal torus $\T$ and $W$ denotes the Weyl group of the pair $(\G,\T)$.

\vskip5pt 
Let $\bX=(\G/\B)^{\rm an}$. For a Schubert variety $i: \bX_w\subseteq \bX$ denote by $\bZ_w$ the Zariski closed subset of $\bX$
 from $\S$ \ref{sec_self_inter} corresponding to a finite set of representatives 
 for $G_{\bX_w}\setminus G / \;G_{\bX_w}$. We abbreviate $P_w:=G_{\bX_w}$ in the following. 
\begin{thm}\label{MainResult2} Let $w\in W$. Suppose the following three conditions. 
\be 

\item $G\bZ_{w}=G_0\bZ_{w}$ with $G_0\subset G$ some compact open subgroup such that $G=G_0P_w$.
\item $\bX_w\setminus G\bZ_w$ is connected.
\item $\bX_{w}$ is smooth.
\ee
 Let $\cN:=i^{P_w}_{+}\cO_{\bX_w}\in \cC_{\bX/P_w}^{\bX_w}$ be the $P_w$-equivariant pushforward of the structure sheaf $\cO_{\bX_w}$. Then the induced module
 $\cM:=\ind_{P_w}^G \cN$ is an irreducible object in $\cC_{\bX/G}$.
  \end{thm}
  \begin{proof}
  One applies Theorem \ref{MainResult} to the data $(\bX,G_0,\bX_w,\bZ_w,\cN)$.  
  Theorem \ref{MainResult} will in fact give a stronger result, namely the irreducibilty of $\Res^G_{G_0}\cM$ in the category $\cC_{\bX}/G_0$. According to Proposition \ref{prop_iwasawa} we have $\Res^G_{G_0}\cM\simeq \ind_{P_w\cap G_0}^{G_0} (\Res^{P_w}_{P_{w}\cap G_0}\cN)$.

  Since $\bX$ is separated and $\bX_w$ is irreducible (by \cite[Theorem 2.3.1]{ConradIrreducible}) and quasi-compact, 
 \cite[Corollary 2.5.18]{EqDCapTwo} implies that the triple $(\bX,\bX_w,G_0)$ satisfies the LSC, whence axiom (A). Axiom (B) follows from hypothesis (a) in view of Corollary \ref{cor_compact_stab}.  Axiom (C) is hypothesis (b). Since $\G$ is assumed to be $K$-split, the algebraic flag variety $\X$ has a smooth model over $o_K$, whence $\bX$ has good reduction. Therefore, Bernstein's inequality holds in $\cC_{\bX/G_0}$ and $\cC_{\bX/{P_w\cap G_0}}$. The verification of the remaining axioms (D,E,F) is now as in the case of projective space. \end{proof}
We recall from Proposition \ref{prop-richardson} that $\bX_w$ has a regular $G$-orbit if and only if $\bZ_w=\emptyset$. In this case, conditions (a) and (b) of the theorem are empty.
\vskip5pt 
For a Schubert variety $\bX_{vP}$ in some partial flag variety $(\G/\P)^{\rm an}$, denote 
by $\bZ_{vP}$ the Zariski closed subset of $(\G/\P)^{\rm an}$
from \ref{sec_self_inter} corresponding to a finite set of repesentatives 
 for $G_{\bX_{vP}}\setminus G / \;G_{\bX_{vP}}$. 
\begin{cor}\label{cor-CorToMain2} Let $w\in W$ be a maximal representative in $W$ for the cosets modulo $W_\P$
for some parabolic subgroup $\B\subseteq\P$, i.e. $w=vw_{o,\P}$ for some $v\in W^\P$. 
Suppose the following three conditions.

\be 
\item $G\bZ_{vP}=G_0\bZ_{vP}$ with $G_0\subset G$ some compact open subgroup such that $G=G_0P_w$.
\item $\bX_{vP}\setminus G\bZ_{vP}$ is connected.

 \item $\bX_{vP}$ is smooth.
 \ee
 Let $\cN:=i^{P_w}_{+}\cO_{\bX_w}\in \cC_{\bX/P_w}^{\bX_w}$. Then 
 $\cM:=\ind_{P_w}^G \cN$ is an irreducible object in $\cC_{\bX/G}$.
 \end{cor}
 \begin{proof}
This follows from Theorem \ref{MainResult2}. Indeed, if 
 $f^{\rm an}: (\G/\B)^{\rm an}\rightarrow (\G/\P)^{\rm an}$ denotes the projection, then 
 $\bX_w=(f^{\rm an})^{ -1}(\bX_{vP})$. Hence condition (a) of the Theorem follows 
 from Proposition \ref{prop-mor}, which, in particular, implies that $\bZ_w=(f^{\rm an})^{ -1}(\bZ_{vP})$.
 So $f^{\rm an}$ restricts to a surjective morphism 
 $$ h: \bX_w\setminus G\bZ_{w}\rightarrow \bX_{vP}\setminus G\bZ_{vP}.$$
Since $f^{\rm an}$ is a Zariski locally trivial fibration with fiber $(\P/\B)^{\rm an}$, there is an
admissible covering of $\bX_{vP}\setminus G\bZ_{vP}$ by connected affinoid opens 
$\bU$ trivializing the covering. In particular, 
$h^{-1}(\bU)\simeq \bU \times_K (\P/\B)^{\rm an}$
is connected, since $(\P/\B)^{\rm an}$ is geometrically connected. Since 
$ \bX_{vP}\setminus G\bZ_{vP}$ is connected by hypothesis, Lemma \ref{Lem_connected} implies that 
$\bX_w\setminus G\bZ_{w}$ is connected. This gives condition (b) of the theorem. 
\end{proof}
\begin{rmk} Theorem \ref{MainResult2} and/or Corollary \ref{cor-CorToMain2} cover, in particular,
closed and open Schubert varieties, the Schubert curves $\bX_{s}$ (for simple reflections $s\in W$) or Schubert varieties of the form $\bX_{w_{o,\P}}$ for some parabolic $\B\subseteq\P$.
In all these case, the $G$-orbit of $\bX_w$ is regular (so that $\bZ_w=\emptyset$). In the case $\G={\rm GL}_n$ all Schubert varieties arising as inverse images from Schuberts in projective space are covered. All Schubert varieties for the groups $\G={\rm GL}_2$ or ${\rm GL}_3$ are covered.\end{rmk}
\begin{rmk} 
We briefly comment on the two conditions (a), (b) and (c) of Theorem \ref{MainResult2}. Condition (a) does not hold for all Schubert varieties $\bX_w$ in $\bX=(\G/\B)^{\rm an}$. 
A first case in which it fails, appears in the case $\G={\rm GL}_4$ and $\bX_w$ equal to the inverse image of the unique Schubert divisor in the analytic Grassmannian variety ${\rm Gr}(2,4)^{\rm an}$. 
This $\bX_w$ is non-smooth, so condition (c) also fails in this case.  The latter reflects the well-known fact, that a general Schubert variety has singularities. One may imagine to eventually remove condition (c) by replacing the push-forward of $\cO_{\bX_w}$ by some
intermediate extension of $\cO_{\bC_w}$ where $\bC_w$ equals the Bruhat {\it cell} indexed by $w\in W$. However, a rigid analytic theory of intermediate extensions is currently not available. 
As for condition (b), we are not aware of any counterexamples where this condition fails. 
\end{rmk}


\section{Locally analytic representations from the BGG category $\cO$ }\label{AutLRsect}

In this and the next section, we give some applications to the locally analytic representation theory of $p$-adic groups.
\vskip5pt 
Let $L/\Qp$ be a finite field extension. Let $\G_L$ be a connected semisimple algebraic group over $L$. Let $L\subseteq K$ be a complete non-archimedean extension field, which is a splitting field for $\G_L$. Set $\G:=\G_L \times_L K$ and let $\fr{g}$ be the Lie algebra of $\G$.

\vskip5pt 
Let $\P_L\subseteq\G_L$ be a parabolic subgroup. 
Let $\T_L\subset \mathbb{L}_L\subset \P_L$ be a maximal split torus and a Levi subgroup respectively. Let $T, P, G$ be the groups of $L$-rational points
of $\T_L, \P_L, \G_L$ respectively. 

\vskip5pt 

Let  $\T,\mathbb{L},\P$ be the base change from $L$ to $K$ of the groups  $\T_L,\mathbb{L}_L,\P_L$ respectively. 
Let $\fr{t}, \fr{l}, \fr{p}$ be the $K$-Lie algebras of $\T,\mathbb{L},\P$ respectively. 

\vskip5pt 

Denote by $\X$ the algebraic flag variety of the split $K$-group $\G=\G_L \times_L K$, with its natural $\G$-action given by conjugating the Borel subgroups of $\G$. Let $\bX=\X^{\an}$ be the rigid analytification of $\X$, with its induced $G$-action.

\subsection{The Orlik-Strauch functor}
We fix a Borel subgroup $\B\subseteq \G$ of $\G$ such that 
$$ \T \subseteq\B \subseteq \P$$ and let $\fr{b}$ be the Lie algebra of $\B$. 
Let $\Phi=\Phi(\G,\T)$ be the roots of $\G$ relative to $\T$. Let $W$ be the corresponding Weyl group. Let $\rho$ be half the sum over the positive roots $\Phi^+$ with respect the Borel $\T\subseteq\B$. 
\vskip5pt

 The algebras of $K$-valued locally $L$-analytic distributions on $P$ and $G$ are denoted by $D(P,K)$ and $D(G,K)$ respectively. Since the center of $\fr{g}$ vanishes, we have the isomorphism $D(G,K)\simeq \wUg{G}$ from \cite[6.5.1]{EqDCapTwo}. 
 
 \vskip5pt 
 
 By the Iwasawa decomposition, we find a maximal compact subgroup $G_0$ of $G$ such that $G=G_0 P$, cf. \cite[3.5]{Cart79}. Let $P_0=G_0\cap P$.
 
 \vskip5pt
 
 Let $D(\fr{g},P)$ respectively $D(\fr{g},P_0)$ be the smallest $K$-subalgebras of $D(G,K)$ containing the rings $U(\fr{g})$ and $D(P,K)$ respectively the rings $U(\fr{g})$ and $D(P_0,K)$. 
 Recall the rings $\wUg{P_0}$ and $\wUg{P}$ from subsection \ref{subsec_rings}.

\begin{lem}\label{UgPDgP} The multiplication map

$$\wUg{P_0}\utimes{D(\fr{g},P_0)} D(\fr{g},P) \congs \wUg{P}$$

is an isomorphism as $(\wUg{P_0},D(\fr{g},P))$-bimodules.
\end{lem}
\begin{proof}
The map in question sits in the composite of maps

$$ \wUg{P_0}\utimes{K[P_0]} K[P] \rightarrow \wUg{P} \utimes{D(\fr{g},P_0)} D(\fr{g},P) \rightarrow \wUg{P}.$$

The composite is bijective and the first map is surjective. Hence all maps in the sequence are bijective.
\end{proof}




\vskip8pt

In the case where $\G_L$ is $L$-split, Orlik-Strauch introduce in \cite{OrlikStrauchCatO} a certain locally analytic lift $\cO^P$ of the parabolic BGG category $\cO^{\fr{p}}$ associated to $\fr{b}\subseteq \fr{p}\subseteq \fr{g}$. 
The definition of $\cO^P$ and certain basic properties which we will use, extend without difficulty to our case of a $K$-split group $\G_L$. The objects in $\cO^P$ are pairs $\underline{M}=(M,\tau)$ where $M\in\cO^{\fr{p}}$ and $\tau: P \rightarrow {\rm Aut}_K(M)$ is a locally analytic locally finite $P$-representation lifting the given $\fr{p}$-representation and compatible with the given $\fr{g}$-representation on $M$. The category $\cO^P$ is abelian and any object is of finite length. There is a forgetful functor 
$\cO^P\rightarrow \cO^{\fr{p}}, \uM\rightsquigarrow M$. 

\vskip8pt

Denote by $\cO^{\fr{p}}_{\rm alg}$ the full subcategory of $\cO^{\fr{p}}$ formed by objects $M$ such that in the weight decomposition $M=\oplus_{\lambda\in \fr{t}^{\vee}} M_{\lambda}$, all occuring weights $\lambda$ lie in the lattice $X^{\bullet}(\T)\subset \fr{t}^{\vee}$. The subcategory $\cO^{\fr{p}}_{\rm alg}$ is closed under extensions in $\cO^{\fr{p}}$. A simple object $L(\lambda)\in \cO^{\fr{p}}$ lies in 
$\cO^{\fr{p}}_{\rm alg}$ if and only if $\lambda\in X^{\bullet}(\T)$ \cite[2.7]{OrlikStrauchJH}. There is a fully faithful embedding $$\cO^{\fr{p}}_{\rm alg}\subset \cO^{P}$$ whose composition with the forgetful functor equals the inclusion $\cO^{\fr{p}}_{\rm alg}\subset \cO^{\fr{p}}$. The point is the following: let $\P=\mathbb{L}\mathbb{U}$ be the Levi decomposition induced by our choice of $\mathbb{L}$. Denote by $\fr{u}$ the Lie algebra 
of the unipotent group $\mathbb{U}$. Then the algebraic $\T$-action on an object $M\in \cO^{\fr{p}}_{\rm alg}$ lifts uniquely to an algebraic $\mathbb{L}$-action on each 
finite dimensional simple $\fr{l}$-constituent of $M$  \cite[2.8]{OrlikStrauchJH}. The $\fr{u}$-action on $M$ lifts uniquely to an algebraic $\mathbb{U}$-action via the exponential map \cite[3.2]{OrlikStrauchJH}. The two actions combine into a $\P$-action whence $M\in\cO^P$.

\vskip5pt

Any object from $\cO^P$ can be naturally regarded as a  $D(\fr{g},P)$-module. Our basic object of study is the functor 

$$\cF_P^G(-)': \cO^P\longrightarrow \cC_{D(G,K)}, \;\;\uM\rightsquigarrow   D(G,K)\utimes{D(\fr{g},P)} \uM$$

\vskip5pt 

which was introduced by Orlik-Strauch \cite{OrlikStrauchCatO,OrlikStrauchJH}. It is exact \cite[3.7]{OrlikStrauchCatO} and faithful \cite[3.7.6]{OrlikStrauchCatO} and preserves irreducibility (assuming $p>2$ or $p>3$ for certain root systems, cf. \cite[Assumption 4.1]{OrlikStrauchCatO}) in an appropriate sense
\cite[4.3]{OrlikStrauchCatO}.

\vskip8pt

We will later on restrict to subcategories, where the infinitesimal character is trivial, in the following sense. We denote by $\cO^P_0$ the subcategory of $\cO^P$ formed by objects $\uM$ satisfying $\fr{m}_0 M=0$ for the maximal ideal $$\fr{m}_0:=Z(\fr{g})\cap U(\fr{g})\fr{g}$$ of the center $Z(\fr{g})$ of $U(\fr{g})$. Similarly, we have $\cO^{\fr{p}}_{0}$ and $\cO^{\fr{p}}_{\rm alg,0}$.
The simple objects in $\cO_{0}=\cO^{\fr{b}}_{0}$ are known to be of the form $L(-w(\rho)-\rho)$ for $w \in W$ \cite[12.2]{HTT}. 

\begin{lem}\label{lem-principalblock} One has
$$\cO^{P}_0=\cO^{\fr{p}}_{\rm alg,0}=\cO^{\fr{p}}_{0}.$$
\end{lem}
\begin{proof} Let $w\in W$. It is well-known \cite[II.1.5]{Jantzen} that $w(\rho)-\rho\in \Z\Phi$. Hence $$-w(\rho)-\rho=-(w(\rho)-\rho)-2\rho\in \Z\Phi \subseteq X^{\bullet}(\T).$$ All other weights of $L(-w(\rho)-\rho)$ 
are given by $-w(\rho)-\rho-\mu$ where $\mu$ is a sum of positive roots, hence lie in $X^{\bullet}(\T)$. It follows $\cO_0=\cO_{{\rm alg},0}$. This implies  
$\cO^{\fr{p}}_{0}=\cO^{\fr{p}}_{\rm alg,0}$ and 
$$ \cO^P_0=\cO^P \cap \cO_0=\cO^P \cap\cO_{{\rm alg},0}= \cO^{\fr{p}}_{\rm alg,0}.$$\end{proof}


\vskip8pt

\subsection{Compatibility with geometric induction I}

We keep all notation from the preceding subsection. Lemma \ref{UgPDgP} provides an 
inclusion $D(\fr{g},P)\subseteq \wUg{P}$. Given $\uM\in \cO^P$ we may form the $\wUg{P}$-module 

$$\w{\uM}:=\wUg{P}\utimes{D(\fr{g},P)} \uM.$$ 

\begin{prop}\label{prop-coad1} One has $\w{\uM}\in \cC_{ \wUg{P}}$ and this yields a functor

$$ \cO^P\rightarrow \cC_{ \wUg{P}},\;\; \uM\rightsquigarrow\w{\uM}. $$


\end{prop}
\begin{proof} We show that $\w{\uM}$ is $\wUg{P}$-coadmissible. This can be proved along the lines of \cite[4.3]{ScSt} and is solely based on the fact that an object from $\cO^P$ can be regarded as a $D(\fr{g},P)$-module which is finitely generated over $U(\fr{g})$. As a $\wUg{P_0}$-module we have $\w{\uM}=\wUg{P_0}\utimes{D(\fr{g},P_0)} \uM$ by Lemma \ref{UgPDgP}. Let $p_1,\ldots,p_r$ be a set of topological generators for $P_0$ and let $m_1,\ldots,m_s$ be a set of $U(\fr{g})$-module generators for $M$. The $\wUg{P_0}$-module
$\wUg{P_0}\otimes_{U(\fr{g})} M$ is finitely presented and hence  coadmissible by \cite[Corollary 3.4v]{ST}. Let $\cK$ be the $\wUg{P_0}$-submodule generated by the finitely many elements $p_i\otimes m_j - 
1\otimes p_i m_j$. Then $\cK$ is coadmissible by \cite[Corollary 3.4iv]{ST} and it suffices to see that $\cK$ equals the kernel of the natural surjection $\wUg{P_0}\otimes_{U(\fr{g})} M\rightarrow \w{\uM}$. Let $p \cdot x := \Ad(p)(x)$ denote the adjoint action of $p \in G$ on an element $x \in U(\fr{g})$; then 
\[ p_i \otimes x m_j - 1 \otimes p_i x m_j = p_i x \otimes m_j - p_i\cdot x \otimes p_i m_j = (p_i \cdot x)( p_i \otimes m_j - 1 \otimes p_i m_j)\]
so $\cK$ contains all elements of the form $p_i \otimes m - 1 \otimes p_i m$ with $m \in M$. Because $D(\fr{g},P_0)$ is generated as a $K$-algebra by $U(\fr{g})$ and $D(P_0,K)$, it remains to see that each element of the form $\delta\otimes m-1\otimes\delta m$ with $\delta\in D(P_0,K)$ also belongs to $\cK$. For any $m \in M$, the map $D(P_0) \to \wUg{P_0} \otimes_{U(\fr{g})} M$, $\delta \mapsto \delta \otimes m - 1 \otimes \delta m$ is continuous. Hence, if $\delta_n\rightarrow\delta$ is a convergent sequence in $D(P_0,K)$, then for any $m\in M$

$$(\delta_n\otimes m-1\otimes\delta_n m)\rightarrow (\delta\otimes m-1\otimes\delta m)$$

\vskip8pt

is a convergent sequence in the coadmissible module $\wUg{P_0}\otimes_{U(\fr{g})} M$. Since $\cK$ is closed in $\wUg{P_0}\otimes_{U(\fr{g})} M$ by \cite[Lemma 3.6]{ST}, we are
therefore reduced to show that $\delta_n\otimes m-1\otimes\delta_n m\in \cK$ for all $n$.
Since the abstract group ring $K[P_0]$ is dense in $D(P_0,K)$ we may assume $\delta_n\in K[P_0]$ and then, by linearity, even $\delta_n\in P_0$.
Since the $p_i$ topologically generate the group $P_0$, the assertion follows. This shows that $\w{\uM}$ is $\wUg{P}$-coadmissible. It is clear that its formation is functorial in $\uM$.
\end{proof}
We recall the full subcategory $ \cC_{\wUg{P},0}$ of $\cC_{\wUg{P}}$ formed by those objects, which are annihilated by $\fr{m}_0$. According to 
\ref{prop_comp} we have the localization functor 
$$\Loc^{\wUg{P}}_\bX: \cC_{\wUg{P},0} \rightarrow \cC_{\bX/P}$$
and we may thus form the localization  $\Loc^{\wUg{P}}_\bX(\w{\uM})$, whenever 
$\fr{m}_0 M=0$.
\vskip5pt
On the other hand, using the identification between $D(G,K)$ and $\wUg{G}$ we may in this case localize the 
$D(G,K)$-module $\cF_P^G(\uM)'=D(G,K)\utimes{D(\fr{g},P)} \uM$ and form $\Loc^{\wUg{G}}_\bX(\cF_P^G(\uM)')$. 

In order to compare the two sheaves on $\bX$, we note that  there is a  
canonical map $$\uM\longrightarrow \cF_P^G(\uM)'=D(G,K)\utimes{D(\fr{g},P)} \uM, \;\; m \mapsto 1\otimes m.$$ 

\begin{prop}\label{prop-compatible} Let $\uM\in \cO_0^P.$ 
The canonical map $\uM\rightarrow \cF_P^G(\uM)'$ induces an isomorphism $$\ind_P^G(\Loc^{\wUg{P}}_\bX(\w{\uM}))\congs \Loc^{\wUg{G}}_\bX(\cF_P^G(\uM)')$$
in $\cC_{\bX/G}$ which is natural in $\uM$. In particular, the diagram of functors 

\[ \xymatrix{\cO_0^P   \ar[rr]^{\cF_P^G(-)' } \ar[d] _{\Loc^{\wUg{P}}_\bX\circ \w{(-)}} && \cC_{D(G,K),0} \ar[d]^{ \Loc^{D(G,K)}_\bX} \\ \cC_{\bX/P}  \ar[rr]_{\ind_P^G} &&  \cC_{\bX/G}.}\]
is commutative up to natural isomorphism.
\end{prop}
\begin{proof} 
The canonical map $\uM\rightarrow \cF_P^G(\uM)'$ extends to a continuous $\wUg{P}$-linear map
$$f_{\uM}: \w{\uM}\rightarrow \cF_P^G(\uM)', a\otimes m\mapsto \iota(a)\otimes m.$$ Here, $\iota$ denotes the inclusion $\wUg{P}\subseteq\wUg{G}=D(G,K)$. As in the formula (\ref{Loc(f)}) above, this gives a morphism $$\Loc(f_{\uM}): \Loc^{\wUg{P}}_\bX(\w{\uM})\longrightarrow \Loc^{\wUg{G}}_\bX(\cF_P^G(\uM)')$$ in $\Frech(P - \cD_\bX)$. Applying Proposition \ref{Adjunction} to $\Loc(f_{\uM})$ results in a morphism 
$$\ind_P^G(\Loc^{\wUg{P}}_\bX(\w{\uM}))\longrightarrow\Loc^{\wUg{G}}_\bX(\cF_P^G(\uM)')$$
in $\cC_{\bX/G}$ which is natural in $\uM$, as required. 

\vskip8pt
In the rest of the proof we will show that this morphism is an isomorphism. We argue locally.
Let $\cN= \Loc^{\wUg{P}}_\bX(\w{\uM})$ and $\cM= \Loc^{\wUg{G}}_\bX(\cF_P^G(\uM)')$. Let $\bU\in \bX_w(\cT)$ and let $H$ be a $\bU$-small subgroup of $G$ contained in $G_0$. Choose a system of representatives $S$ for the $(H,P)$-double cosets in $G$. By \cite[2.3.6/7]{EqDCapTwo}, there is a canonical isomorphism
\[\ind_P^G(\cN)(\bU)=\bigoplus_{s\in S} \w\cD(\bU,H)\underset{\w\cD(\bU,H\cap ^{s}P)}{\w\otimes}\cN_s(\bU) \]
in $\cC_{\w\cD(\bU,H)}$ where $\cN_s={\rm Res}^{^{s}P}_{H\cap ^{s}P}[s]s_{*}\cN$ and $[s]s_{*}$ is the twisting functor
$\cC_{\bX/P} \rightarrow \cC_{\bX/^{s}P}$ from \cite[2.2.4]{EqDCapTwo}. On the other hand, there is the canonical isomorphism \cite[3.5.6]{EqDCap}
\[\cM (\bU)\cong  \w\cD(\bU,H)\underset{D(H,K)}{\w\otimes} \cF_P^G(\uM)'.\]
By Lemma \ref{twist} we have
\[  D(H,K)s\wUg{P} \utimes{\wUg{P}} \w{\uM}\cong D(H,K)\underset{\wUg{H\cap ^{s}P}}{\w\otimes} [s]\w{\uM} \]
and so, using the double coset decomposition from Proposition \ref{decomp},

\[ \begin{array}{lll}  \cF_P^G(\uM)'&\simeq & D(G,K)\utimes{\wUg{P}} \w{\uM}\\
&\simeq & \big(\bigoplus_{s\in S} D(H,K)s\wUg{P}\big) \utimes{\wUg{P}} \w{\uM}\\
&\simeq & \bigoplus_{s\in S} D(H,K)s\wUg{P} \utimes{\wUg{P}} \w{\uM}\\
&\simeq & \bigoplus_{s\in S} D(H,K)\underset{\wUg{H\cap ^{s}P}}{\otimes} [s]\w{\uM}
\end{array}\]
in $\cC_{D(H,K)}$. It follows that
\[ \begin{array}{lll} \cM(\bU) &\simeq& \w\cD(\bU,H)\wotimes{D(H,K)} \cF_P^G(\uM)'\\
 &\simeq & \w\cD(\bU,H)\wotimes{D(H,K)} \big(
\bigoplus_{s\in S} D(H,K)\wotimes{\wUg{H\cap ^{s}P}} [s]\w{\uM}  \big)\\
&\simeq&\bigoplus_{s\in S} \w\cD(\bU,H)\wotimes{\wUg{H\cap ^{s}P}} [s]\w{\uM}\\
&\simeq& \bigoplus_{s\in S} \w\cD(\bU,H)\wotimes{\w\cD(\bU,H\cap ^{s}P)} \big(\big( \Loc_{\bX}^{\wUg{^{s}P}} [s]\w{\uM}\big)(\bU)\big)
\end{array}\]
where \[ \big(\Loc_{\bX}^{\wUg{^{s}P}} [s]\w{\uM}\big)(\bU) = \w\cD(\bU,H\cap ^{s}P) \wotimes{\wUg{H\cap ^{s}P}} [s]\w{\uM}.\]
By Proposition \ref{compatwist} \[\Loc_{\bX}^{\wUg{^{s}P}} [s]\w{\uM}\simeq [s]s_* \Loc_{\bX}^{\wUg{P}}\w{\uM}=[s]s_*\cN\]
and so we arrive at the isomorphism
\[\cM(\bU)\simeq \bigoplus_{s\in S} \w\cD(\bU,H)\wotimes{\w\cD(\bU,H\cap ^{s}P)} \cN_s(\bU) \simeq \ind_P^G(\cN)(\bU).\]
Tracing through the definitions, one checks that it is induced from the morphism of sheaves in question $\ind_P^G(\cN)\rightarrow \cM$. The proof of the proposition is complete.
\end{proof}


In the following we aim at finding a simpler description of the left-vertical arrow 
$$ \Loc^{\wUg{P}}_\bX\circ \w{(-)}: \cO_0^P\longrightarrow \cC_{\bX/P}$$
from the diagram of \ref{prop-compatible}, which avoids the use of the auxiliary ring $\wUg{P}$.

\subsection{Extensions for $\fr{g}$-modules}

Recall the extension functor from 
\cite{SchmidtVerma}

$$E_{\fr{g}}: U(\fr{g}) {\rm -mod} \longrightarrow \w{U}(\fr{g}) {\rm -mod},\;\; M\rightsquigarrow \w{M} := \w{U}(\fr{g})\otimes_{U(\fr{g})} M.$$
Denote by $\w{\cO}^{\fr{p}}$ the essential image under $E_{\fr{g}}$ of the parabolic BGG category $\cO^{\fr{p}}$ (denoted by $\hat{\cO}^{\fr{p}}$ in \cite{SchmidtVerma}). The following proposition summarizes some basic properties.

\begin{prop}\label{prop-extension}  \hsp

\be 

\item The functor $E_{\fr{g}}$ is exact and faithful. 
\item If $M$ is a finitely generated $U(\fr{g})$-module, then $\w{M}$ is $\w{U}(\fr{g})$-coadmissible.
\item  $E_{\fr{g}}$ induces an equivalence of categories
 $\cO^{\fr{p}}\congs \w{\cO}^{\fr{p}}$. 
 \item $\w{\cO}^{\fr{p}} \subset\cC_{\w{U}(\fr{g})}$ is closed under passage to submodules and quotients.
\ee
\end{prop}
\begin{proof}
Part (a) follows from \cite[Prop. 3.6]{ScSt} and \cite[Theorem 3.1]{DCapThree}. Finitely presented 
modules are coadmissible, whence (b).
Part (c) is \cite[Theorem 4.3.1]{SchmidtVerma} and part (d) is \cite[Lem. 3.6.4]{SchmidtVerma}.
\end{proof}

If $\uM\in\cO^P$, then the canonical map $\iota: \w{U}(\fr{g}) \rightarrow \w{U}(\fr{g},P)$ induces a canonical $\w{U}(\fr{g})$-linear map $$\w{M}=\w{U}(\fr{g})\utimes{U(\fr{g})} M \longrightarrow \w{\uM}=\wUg{P}\utimes{D(\fr{g},P)} \uM,\;\; x\otimes m\mapsto \iota(x)\otimes m.$$

Source and target are coadmissible modules over $\w{U}(\fr{g})$ and 
$\w{U}(\fr{g},P)$ respectively, according to \ref{prop-extension} and \ref{prop-coad1}, and hence, carry their canonical topologies.
\begin{prop} \label{prop-bijectiveGlobal} \hsp
\be 
\item The map $\w{U}(\fr{g}) \rightarrow \w{U}(\fr{g},P)$ is a continuous injection.
\item The map 
$ \w{M} \longrightarrow \w{\uM}$
is a continuous bijection for any $\uM\in\cO^P$.

\ee
\end{prop}
\begin{proof}
The map in question factors through the map 
 $\w{U}(\fr{g}) \rightarrow \w{U}(\fr{g},P_0)$ and it suffices to establish the claim for the latter map. But this latter map equals the projective limit over all pairs $(\cL, N)$ of the canonical maps
  $\hK{U(\cL)}  \rightarrow   \hK{U(\cL)} \rtimes_N P_0 $ 
 where $\cL$ is a $P_0$-stable Lie lattice in $\fr{g}$ and $N$ is an open subgroup of $(P_0)_{\cL}$ which is normal in $P_0$. Each of these maps is a continuous injection, by definition of the crossed products $\hK{U(\cL)} \rtimes_N P_0$ \cite[2.2.3/4]{EqDCap} and their topology \cite[3.4.8 and 6.2.9]{EqDCap}. This shows (a).
 
As a $\wUg{P_0}$-module, we have $\w{\uM}=\wUg{P_0}\otimes_{D(\fr{g},P_0)} \uM$ and this holds as topological modules with respect to the canonical topology on the right-hand side. In order to prove the statement of the lemma, we may therefore replace the group $P$ by $P_0$. We start with the continuity. 
Our map in question factors through the map   
$$\w{M}=\w{U}(\fr{g})\utimes{U(\fr{g})} M \longrightarrow \wUg{P_0}\utimes{U(\fr{g})} M,\quad\quad x\otimes m\mapsto \iota(x)\otimes m.$$
As observed in the proof of \ref{prop-coad1}, the $\wUg{P_0}$-module
 $\w{\uM}$ equals the quotient of the finitely presented $\wUg{P_0}$-module $\wUg{P_0}\otimes_{U(\fr{g})} M$
 by a coadmissible submodule. Its canonical topology equals therefore the quotient topology. It therefore suffices to establish the continuity of the map
$\w{M}\rightarrow \wUg{P_0}\otimes_{U(\fr{g})} M$.
The $\w{U}(\fr{g})$-module $\w{M}$ is finitely presented and its canonical topology can therefore be defined as the quotient topology with respect to a finite presentation as $\w{U}(\fr{g})$-module. Using the very same presentation to define the canonical topology of the finitely presented $\wUg{P_0}$-module $\wUg{P_0}\otimes_{U(\fr{g})} M$, we are therefore reduced to show the continuity of the map $\w{U}(\fr{g})^{\oplus n} \rightarrow \w{U}(\fr{g},P_0)^{\oplus n}$ induced by $\iota$. This follows from part (a).

For the bijectivity, it suffices to see that the analogous map 
\[ f: \w{M} \longrightarrow \wUg{P_0}\utimes{D(\fr{g},P_0)} \uM \]
is bijective. By construction, the morphism $f$ equals the projective limit over all pairs $(\cL, N)$ of the morphisms
\[f_{\cL,N}: \hK{U(\cL)} \utimes{U(\fr{g})} M \longrightarrow  ( \hK{U(\cL)} \rtimes_N P_0 )  \utimes{D(\fr{g},P_0)} \uM\]
where $\cL$ is a $P_0$-stable Lie lattice in $\fr{g}$ and $N$ is an open subgroup of $(P_0)_{\cL}$ which is normal in $P_0$.
It suffices to see that each map $f_{\cL,N}$ is bijective. Each map $f_{\cL,N}$ is visibly surjective. For the injectivity, we construct an explicit left inverse map for $f_{\cL,N}$, following the argument in \cite[Lem. 4.6]{ScSt}. We observe that 
 $\hK{U(\cL)} \utimes{U(\fr{g})} M$ has a natural $P_0$-action given by $p.(x\otimes m):=\Ad(p)(x)\otimes pm$. On the one hand, this $P_0$-action is compatible with the crossed product structure and yields a $\hK{U(\cL)} \rtimes_N P_0$-action, i.e. a map 
  $$  ( \hK{U(\cL)} \rtimes_N P_0 )  \utimes{K}  ( \hK{U(\cL)} \utimes{U(\fr{g})} M )
  \longrightarrow  \hK{U(\cL)} \utimes{U(\fr{g})} M, $$
 which we write as  $\lambda\otimes (u\otimes m)\mapsto \lambda * (u\otimes m)$,
 for $\lambda\in \hK{U(\cL)} \rtimes_N P_0, u\in  \hK{U(\cL)}, m\in M$.
 We denote the restriction of this action to the ring $D(\fr{g},P_0)$ via the natural homomorphism $D(\fr{g},P_0)\rightarrow \hK{U(\cL)} \rtimes_N P_0$ by the same symbol.
  Precomposing with the map induced by $\uM\longrightarrow  \hK{U(\cL)} \utimes{U(\fr{g})} M, m\mapsto 1\otimes m$ yields a map 

 $$  \tilde{f}^{-1}_{\cL,N}: ( \hK{U(\cL)} \rtimes_N P_0 )  \utimes{K} \uM\longrightarrow  \hK{U(\cL)} \utimes{U(\fr{g})} M,$$
 given by $\lambda\otimes m\mapsto \lambda * (1\otimes m)$.
 Note that $\delta*(1\otimes m)=1\otimes\delta m$ if $\delta\in D(\fr{g},P_0)$. Indeed, 
 $D(\fr{g},P_0)$ is generated by $U(\fr{g})$ and $D(P_0,K)$ and the group ring $K[P_0]$ is dense in $D(P_0,K)$. The claim thus follows by continuity of the actions on the Banach module $\hK{U(\cL)} \utimes{U(\fr{g})} M$. Using this, one computes that  $$ \tilde{f}^{-1}_{\cL,N}(\lambda\delta \otimes m)= (\lambda\delta)*(1\otimes m)=\lambda*(\delta*(1\otimes m))= \lambda*(1\otimes\delta m)= \tilde{f}^{-1}_{\cL,N}(\lambda \otimes \delta.m)$$
 for $\lambda\in \hK{U(\cL)} \rtimes_N P_0, \delta\in D(P_0,K), m\in M$. We see that
 $\tilde{f}^{-1}_{\cL,N}$
 factors into a map 
 $$ f^{-1}_{\cL,N}:( \hK{U(\cL)} \rtimes_N P_0 )  \utimes{D(\fr{g},P_0)} \uM\longrightarrow  \hK{U(\cL)} \utimes{U(\fr{g})} M.$$
 This is the required left inverse map (as our notation suggests). 
 To check the identity $f^{-1}_{\cL,N}\circ f_{\cL,N}={\rm id}$, it suffices to consider elements of the form $1\otimes m$ for $m\in M$, since both maps $f^{-1}_{\cL,N}$ and $f_{\cL,N}$ are  $\hK{U(\cL)} $-linear - where it is obvious. 
 \end{proof}

We recall at this point the localization equivalence for the full subcategory $\cC_{\w{U}(\fr{g}),0}$, consisting of objects $M\in \cC_{\w{U}(\fr{g})}$ with $\fr{m}_0M=0$.

\begin{thm}\label{thm-localizationg} The functor $\Loc^{\w{U}(\fr{g})}_\bX$ induces an equivalence of categories 
$$\cC_{\w{U}(\fr{g}),0}\congs \cC_{\bX}.$$
A quasi-inverse is given by the global sections functor $H^0(\bX,-)$.
\end{thm}
\begin{proof}
This is announced in \cite[Theorem E]{DCapOne} and follows from \cite[6.4.9]{EqDCap}.\end{proof}

\subsection{Compatibility with geometric induction II}

Let $\uM\in \cO^P_0$. 
There is the composed morphism 
$$\w{M} \longrightarrow \w{\uM}\longrightarrow \Loc^{\wUg{P}}_\bX(\w{\uM}) $$ which is $U(\fr{g})$-linear and whose target is a $\w{\cD}_{\bX}$-module. It therefore induces a $\w{\cD}_{\bX}$-linear morphism
$$\Loc^{\w{U}(\fr{g})}_\bX(\w{M}) \longrightarrow \Loc^{\wUg{P}}_\bX(\w{\uM})$$
which is natural in $\uM$. 

\begin{prop}\label{prop-bijectiveLocal} Let $\uM\in \cO_0^P$. The morphism
$$\Loc^{\w{U}(\fr{g})}_\bX(\w{M}) \congs \Loc^{\wUg{P}}_\bX(\w{\uM})$$
is an isomorphism.
\end{prop}
\begin{proof}
The map $\w{M} \rightarrow \w{\uM}$ induces a morphism of presheaves on $\bX_w(\cT)$ 
 \[f: \cP^{\w{U}(\fr{g})}_\bX(\w{M})\longrightarrow \cP^{\wUg{P}}_\bX(\w{\uM}).\]
Given an affinoid $\bU\in  \bX_w(\cT)$ we have 
\[\cP^{\wUg{P}}_\bX(\w{\uM})(\bU) = \invlim \w\cD(\bU,H) \underset{\wUg{H}}{\w\otimes} \w{\uM}\]
where, in the inverse limit, $H$ runs over all the $\bU$-small subgroups of $P$. Given a $H$-stable affine formal model $\cA$ of $A=\cO(U)$,
 we have 
\[\w\cD(\bU,H) \underset{\wUg{H}}{\w\otimes} \w{\uM} = \underset{(\cL,N)\in\cI(H)}{\invlim\limits}{} \big( \hK{U(\cL)} \rtimes_N H \big)\underset{\wUg{H}}{\otimes} \w{\uM}  \]

where $\cI(H)$ denotes the set of all $\cA$-trivialising pairs, i.e. the set of pairs $(\cL, N)$, where $\cL$ is an $H$-stable $\cA$-Lie lattice in ${\rm Der}_K(A)$ and $N$ is an open subgroup of $H_{\cL}$ which is normal in $H$ \cite[3.3.1]{EqDCap}. Moreover, the canonical topology on left-hand side of this equality (which is a coadmissible $\w\cD(\bU,H)$-module)
equals the projective limit topology. The map $f(\bU)$ is continuous in the canonical topologies of source and target. Moreover, $f(\bU)$ is the 
projective limit of maps 
\[f(\bU)_{\cL, N}:    \hK{U(\cL)}\underset{\w{U}(\fr{g})}{\otimes} \w{M}    \rightarrow       \big( \hK{U(\cL)} \rtimes_N H \big)\underset{\wUg{H}}{\otimes} \w{\uM}.\]
Each of the maps $f(\bU)_{\cL, N}$ is visibly surjective. The injectivity follows by constructing an explicit left inverse, very similar to the proof of \ref{prop-bijectiveGlobal}(b). Passing to the limit, we see that $f(\bU)$ is bijective. This proves the proposition.
\end{proof}

\vskip8pt
 
Let us define $\Loc^{U(\fr{g})}_\bX$ to be the composite of the Beilinson-Bernstein localization functor 
$$\Loc^{U(\fr{g})}_\X: {\rm coh}(U(\fr{g})_0)\rightarrow {\rm coh}(\cD_{\X})$$ from \cite{BB}, followed by rigid analytification \ref{prop-analytification2}
$$\rho^*: {\rm coh}(\cD_{\X})\rightarrow  {\rm coh}(\cD_{\bX}).$$
Recall the extension functor $E_{\bX}$ from \ref{Analytification}.

\begin{thm} \label{thm-compatible} The functor $E_{\bX}\circ \Loc^{U(\fr{g})}_\bX$, restricted to the category $\cO_0^P$, takes values in $\cC_{\bX/P}$. The resulting
diagram of functors 

\[ \xymatrix{\cO_0^{P}  \ar[rr]^{\cF_P^G(-)' } \ar[d] _{E_{\bX}\circ  \Loc^{U(\fr{g})}_\bX}&& \cC_{D(G,K),0} \ar[d]^{ \Loc^{D(G,K)}_\bX} \\ \cC_{\bX/P}  \ar[rr]_{\ind_P^G} &&  \cC_{\bX/G}.}\]
is commutative up to natural isomorphism.

\end{thm}

\begin{proof} Let $\uM\in\cO_0^{P}$. One has
$$E_{\bX}\circ \Loc_{\bX}^{U(\fr{g})}(M)=\w\cD_{\bX}\otimes_{U(\fr{g})} M=
\Loc^{\w{U}(\fr{g})}_\bX(\w{M})\simeq \Loc^{\wUg{P}}_\bX(\w{\uM})$$
by contracting tensor products and by Proposition \ref{prop-bijectiveLocal}. In particular, this object lies in $\cC_{\bX/P}$, so that its geometric induction $\ind_P^G(-)$ is well-defined. Now Proposition \ref{prop-compatible} implies the commutativity of the diagram appearing in the theorem.
\end{proof}

\section{Irreducibility results}

We keep all the notations from the preceding section. 

\subsection{The support of irreducible representations}
As before, let $W$ be the Weyl group of $(\G,\T)$. Let $\X_w\subseteq\X$ be an algebraic Schubert variety associated with some $w\in W$. Suppose that the parabolic $\P$ equals the stabilizer in $\G$ of $\X_w$. Let $\bX_w$ be the associated rigid-analytic variety to $\X_w$. As usual, for a subset $S\subseteq\bX$, we denote by $\overline{S}$ its closure in the Huber space $\mathscr{P}(\bX)$. 

\begin{lem}\label{GPcocompact} The quotient space $G/P$ is compact.
\end{lem}
\begin{proof} One may apply \cite[Prop. 9.3]{BorelTits}, since $L$ is locally compact. 
\end{proof}
\begin{lem}\label{lem-Gschubert} We have $G \overline{\bX_w}=\overline{G\bX_w}$, i.e. the $G$-orbit of $\overline{\bX_w}$ is closed in 
$\mathscr{P}(\bX)$.
\end{lem} 
\begin{proof} Because $\overline{G \bX_w}$ is a closed and $G$-stable subset of $\sP(\bX)$ which contains $\bX_w$, it also contains $G \overline{\bX_w}$. For the reverse inclusion it suffices to show that $G \overline{\bX_w}$ is closed.  The subset $\X_w$ of $\X$ is stabilized by $\P\subset\G$. Hence $\bX_w$ and $\overline{\bX_w}$ are stabilized by $P$. Now $G/P$ is compact by Lemma \ref{GPcocompact}, so if $H$ is any open compact subgroup of $G$ then $H \backslash G / P$ is finite by \cite[2.2.1]{EqDCapTwo}. Choose $g_1,\ldots,g_m \in G$ such that $G = H g_1 P \cup \cdots \cup H g_m P$; then $G \overline{\bX_w} = \cup_{i=1}^m H g_i \overline{\bX_w}$ is a finite union of the sets $H g_i \overline{\bX_w} = g_i H^{g_i} \overline{\bX_w}$ which are closed by \cite[2.1.15]{EqDCapTwo}, and is therefore itself closed. \end{proof}
 
As before, let $\cO$ be the classical BGG category relative to $\fr{b}\subseteq\fr{g}$
and consider the irreducible module $L_w:=L(-w(\rho)-\rho)\in\cO_0$.
Since $\P$ stabilizes $\X_w$, one has $L_w\in\cO^{\mathfrak{p}}$ for $\fr{p}=\Lie(\P)$. 
\begin{prop}\label{prop-supp} One has
 \[ \Supp E_{\bX} (\rho^{*}\cL_w)=\overline{\bX_w} \]
  for any $w\in W$.
\end{prop}
\begin{proof}
Let $\cL_w= \Loc^{U(\fr{g})}_\X (L_w)$.  
It is well-known that 
 $\Supp \cL_w=\X_w$ \cite[12.3.2]{HTT}. Then Lemma \ref{lem-zariskiclosedalg} and Corollary \ref{cor-supportpreserved} imply
\[ \Supp E_{\bX} (\rho^{*}\cL_w)= \tilde{\rho}^{-1}(\X_w)=\overline{\bX_w} \]
  for the 
  canonical map $\tilde{\rho}: \mathscr{P}(\bX)\rightarrow \X$.  
\end{proof}
\begin{lem}\label{SuppInd} We have $G  \Supp \cN  \subseteq \Supp \ind_P^G \cN \subseteq  \overline{G  \Supp \cN}$ for all $\cN \in \cC_{\bX/P}$. 
\end{lem}
\begin{proof}Recall the map $\alpha_{\cN} : \cN \to \ind_P^G \cN$ from $\S \ref{IndRes}$. Using the proof of \cite[Lemma 2.5.3]{EqDCapTwo} together with the standard argument that shows that a faithfully-flat ring map is universally injective \cite[\href{https://stacks.math.columbia.edu/tag/05CK}{Tag 05CK}]{stacks-project}  we see that $\alpha_{\cN}$ is injective on spaces of local sections. Therefore it is also injective on stalks, which implies
\[ \Supp \cN \quad \subseteq \quad \Supp \ind_P^G \cN.\]
Since the set on the right hand side is clearly $G$-stable, we obtain the first inclusion. For the second inclusion, let $\bU \in \bX_w(\cT)$ be such that $(G  \Supp \cN) \cap \tilde{\bU}$ is empty. Then $\Supp \cN \cap G \tilde{\bU} = \emptyset$ as well, so $\cN(s^{-1}\bU) = 0$ for all $s \in G$. We can now conclude from \cite[2.2.12]{EqDCapTwo} that $(\ind_P^G \cN)(\bU) = 0$ for every such $\bU$. Now if $x \in \Supp \ind_P^G\cN$ but $x$ does not lie in the closure of $G  \Supp \cN$, then we can find an open neighbourhood $\tilde{\bU}$ of $x$ such that $(G  \Supp \cN) \cap \tilde{\bU}$ is empty. Then $(\ind_P^G \cN)_x \neq 0$ but $(\ind_P^G\cN)(\bU) = 0$ by the above --- a contradiction. So $x \in \overline{G  \Supp \cN}$ as required.
\end{proof}
\begin{thm}\label{thm-support} $\Supp \Loc^{D(G,K)}_{\bX} (\cF_{P}^G(L_w)')= G \overline{\bX_w}$.
\end{thm}
\begin{proof}   Let $\cL_w= \Loc^{U(\fr{g})}_\X (L_w)$.  
According to Theorem \ref{thm-compatible} the statement amounts to 
 \[\Supp \ind_{P}^G \circ E_{\bX} (\rho^{*}\cL_w)= G \overline{\bX_w}.\] According to Proposition \ref{prop-supp}, we have 
  \[ \Supp E_{\bX} (\rho^{*}\cL_w)=\overline{\bX_w}. \]
Applying Lemma \ref{SuppInd}, we see that 
   \[ G  \overline{\bX_w} \quad\subseteq \quad \Supp \ind_{P}^G \circ E_{\bX} (\rho^{*}\cL_w) \quad \subseteq \quad \overline{G  \bX_w}\]
  But the right-hand side equals $G\overline{\bX_w}$ by Lemma \ref{lem-Gschubert}. 
   \end{proof}
\begin{ex}
We discuss the two extreme cases $w=w_{o}$ and $w=1$.
In the first case, the module $L_{w_o}=L(0)$ equals the trivial 
$\mathfrak{g}$-representation and $P=G$. Then $\cF_{P}^G(L_w)$ equals the trivial one-dimensional $G$-representation. One has $\X_{w_o}=\X$ and so $G\overline{\bX}_{w_o}=\mathscr{P}(\bX)$.
\vskip5pt
 In the second case, the module $L_1=L(-2\rho)$ equals the antidominant Verma module $M(-2\rho)$ and $P$ is a minimal parabolic. If $K=L$, then $\cF_{P}^G(L_w)$ equals the principal series $G$-representation $\ind_B^G (2\rho)$ induced from the algebraic character $2\rho$ of the maximal torus $T$. Moreover, $\X_1= \B /\B \subseteq \G/\B=\X$ is the base point determined by $\B$ and $G\overline{\bX}_1=\X(L)$ equals the set of $L$-rational points of $\X$, viewed as a subset of $\mathscr{P}(\bX)$.
\end{ex}

\subsection{Geometric proofs of irreducibility}
We keep the notation from the preceding subsection. 

\begin{prop}\label{prop-irred} Let $w\in W$ and $\cL_w= \Loc^{U(\fr{g})}_\X (L_w)$.
\be 

\item The coadmissible $\w{U}(\fr{g})$-module $\w{L}_w$ is irreducible.
\item The coadmissible  $\w\cD_{\bX}$-module $\Loc^{\w{U}(\fr{g})}_\bX(\w{L}_w)$ is irreducible. 
\item Let $P\subseteq G$ be a parabolic subgroup with $L_w\in\cO_0^P$.   Then $E_{\bX} (\rho^{*}\cL_w)$ is an irreducible object in $\cC_{\bX/P}$.
\ee
\end{prop}

\begin{proof} Part (a) follows from the equivalence of categories
\ref{prop-extension}. This implies (b) by the localization equivalence 
 $\Loc^{\w{U}(\fr{g})}_\bX$, cf. \ref{thm-localizationg}. Finally, (c) follows from (b), since
  \[E_{\bX} (\rho^{*}\cL_w)=E_{\bX}\circ \Loc_{\bX}^{U(\fr{g})}(L_w)=\w\cD_{\bX}\utimes{U(\fr{g})} L_w=
\Loc^{\w{U}(\fr{g})}_\bX(\w{L}_w)\]
as $\w\cD_{\bX}$-modules.
\end{proof}

Let $P_w\subseteq G$ be a parabolic subgroup which is {\it maximal} for $L_w$ in the sense of Orlik-Stauch's \cite[Definition 5.2]{OrlikStrauchJH}. Letting $\mathfrak{p}_w=\Lie(P_w)\otimes_L K$, this means $L_w\in\cO^{\mathfrak{p}_w}$, but $L_w\notin\cO^{\mathfrak{q}}$ for any parabolic $\mathfrak{p}_w\subsetneq\mathfrak{q}$ strictly containing 
$\mathfrak{p}_w.$ Note that $L_w\in\cO^{\mathfrak{p}}$ for some $\fr{p}=\Lie(\P)$, implies that $\cL_w$ is $\P$-equivariant, whence $\P$ stabilizes $\X_w$. In particular, the stabilizer of 
$\bX_w$ in $G$ equals $P_w$. 

\vskip5pt 
In the main theorem \cite[Theorem 5.3]{OrlikStrauchJH} Orlik-Strauch show that
\[V_w:=\cF_{P_w}^G(L_w)\] is an irreducible locally analytic $G$-representation provided that (H1) 
$K=L$, i.e. $\G=\G_L$ is $L$-split, 
and (H2) that $p>2$ if the root
system of $\G$ has irreducible components of type $B$, $C$ or $F_4$, and $p>3$ if the root system has
irreducible components of type $G_2$. Their argument relies on
the delicate calculation of explicit formulae for the action of certain nilpotent generators on highest weight modules of the BGG category $\cO$, cf. \cite[Appendix]{OrlikStrauchJH}.

\vskip5pt

We deduce the irreducibility of $V_w$ in a geometric way 
that does not need the two hypotheses (H1) and (H2), whenever the geometric
conditions (a), (b) and (c) from Theorem \ref{MainResult2} are satisfied for the analytic Schubert variety $\bX_w$.

\begin{thm} Let $w\in W$. Suppose the following three conditions. 
\be 
\item $G\bZ_{w}=G_0\bZ_{w}$ with $G_0\subset G$ some compact open subgroup such that $G=G_0P_w$.
\item $\bX_{w}\setminus G\bZ_w$ is connected.
\item $\bX_{w}$ is smooth.
\ee
 The locally analytic $G$-representation $\cF_{P_w}^G(L_w)$ is irreducible.
\end{thm}
 \begin{proof}
 By the localization equivalence \ref{thm-localizationP} it suffices to check that
 $\Loc^{D(G,K)}_\bX(\cF_{P_w}^G(L_w)')$ is irreducible in $\cC_{\bX/G}$. According to 
 \ref{thm-compatible}, 
 \[\Loc^{D(G,K)}_\bX(\cF_{P_w}^G(L_w)') \simeq \ind_{P_w}^G \circ E_{\bX} (\rho^{*}\cL_w)\]in 
 $\cC_{\bX/G}$. Let $i: \bX_w\hookrightarrow \bX$ be the inclusion. Note that $\rho^{*}\cL_w=i_+^{\rm cl}\cO_{\bX_w}$, where $i_+^{\rm cl}$ denotes the classical push-forward from  $\Hol(\cD_{\bX_w})$ to $\Hol(\cD_{\bX})$, compare Proposition \ref{prop_SelfDual} and its proof.
 Since $i_+^{\rm cl}$ commutes with $E_{\bX_w}$ and $E_{\bX}$, we have $E_{\bX} (\rho^{*}\cL_w) =i_+\cO_{\bX_w}$. Now $\Res^{P_w}_1i_+^{P_w}\cO_{\bX_w}\simeq i_+\cO_{\bX_w},$
as $\wideparen{\mathcal{D}}_\bX$-modules by Proposition \ref{prop_resG1}. This implies 
$$\cN:=i_+^{P_w}\cO_{\bX_w}=E_{\bX} (\rho^{*}\cL_w) $$ as objects in $\cC_{\bX_w/{P_w}}$. But $\ind_{P_w}^G \cN$ is irreducible in $\cC_{\bX/G}$ by Theorem \ref{MainResult2}.  
\end{proof}

\bibliographystyle{plain}
\bibliography{references}

\end{document}